\documentclass[11pt]{amsart}
\usepackage{amssymb,amsmath,txfonts,mathrsfs,mathtools,titletoc,tikz,stmaryrd}
\usepackage{color}
\usepackage{graphicx}
\usepackage{float}
\usetikzlibrary{arrows.meta,positioning,calc,decorations.pathreplacing}
\usepackage{hyperref}
\usepackage[alphabetic]{amsrefs}
\allowdisplaybreaks
\hypersetup{colorlinks=true,linkcolor=blue,citecolor=blue,urlcolor=blue}

\newtheorem{theorem}{Theorem}[section]
\newtheorem{prop}[theorem]{Proposition}
\newtheorem{lemma}[theorem]{Lemma}
\newtheorem{remark}[theorem]{Remark}

\newtheorem{question}[theorem]{Question}
\newtheorem{definition}[theorem]{Definition}
\newtheorem{notation}[theorem]{Notation}
\newtheorem{cor}[theorem]{Corollary}
\newtheorem{example}[theorem]{Example}

\DeclareMathOperator{\Ric}{Ric}
\DeclareMathOperator{\Rc}{Rc}

\DeclareMathOperator{\Id}{Id}
\newcommand{\R}{\mathbb R}
\newcommand{\C}{\mathbb C}
\newcommand{\CP}{\mathbb{CP}}

\numberwithin{equation}{section}
\def\pf{{\it Proof:}~}
\begin{document}
\title[Eigenvalues on spheres]{Eigenvalues on spheres}
\author{Shengjie Lin}
\address{Department of Mathematical Sciences\\Tsinghua University, Beijing\\P. R. China}
\email{linsj667788@gmail.com}
\author{Haibin Wang}
\address{School of Mathematical Sciences\\Peking University, Beijing\\P. R. China}
\email{haibinwang@pku.edu.cn}
\author[Guoyi Xu]{Guoyi Xu\textsuperscript{*}}
\address{Department of Mathematical Sciences\\Tsinghua University, Beijing\\P. R. China}
\email{guoyixu@tsinghua.edu.cn}
\date{\today}
\begin{abstract}
For every smooth Riemannian metric on the two sphere whose Gaussian curvature is bounded below by one, we prove that each positive Laplace eigenvalue, counted with multiplicity, is no smaller than the corresponding eigenvalue of the unit round sphere. Equality at any positive position in the ordered spectrum forces the metric to be isometric to the unit round metric.

We further establish a sharp finite spectral counting comparison for Alexandrov two spheres with curvature bounded below by one. At every positive spectral threshold of the unit round sphere, the number of Laplace eigenvalues below or at that threshold, counted with multiplicity, does not exceed the corresponding number for the round sphere. Equality at any such threshold forces the Alexandrov sphere to be isometric to the unit round sphere.

As an application, we obtain the sharp Euclidean dimension bound for spaces of polynomial growth harmonic functions on complete three dimensional manifolds with nonnegative sectional curvature and positive asymptotic volume ratio, together with rigidity in the equality case.
\end{abstract}
\subjclass[2020]{Primary 58J50; Secondary 35P15, 53C21, 53C23,
53C24, 31C12}

\thanks{\textsuperscript{*}Corresponding author:
Guoyi Xu
(\href{mailto:guoyixu@tsinghua.edu.cn}
{guoyixu@tsinghua.edu.cn}).}

\thanks{Guoyi Xu was partially supported by NSFC 12141103.}
\maketitle
\titlecontents{section}[0em]{}{\hspace{.5em}}{}{\titlerule*[1pc]{.}\contentspage}
\titlecontents{subsection}[1.5em]{}{\hspace{.5em}}{}{\titlerule*[1pc]{.}\contentspage}
\tableofcontents

\section{Introduction}

Throughout the paper we use the conventions
\begin{align}
\mathbb{Z}^+=\{1,2,\ldots\},\qquad \mathbb{Z}_{\geq0}=\{0,1,2,\ldots\}. \nonumber
\end{align}
For a nonnegative self-adjoint operator \(P\) with compact resolvent and eigenvalues
\(\lambda_0(P)\leq\lambda_1(P)\leq\cdots\), counted with multiplicity, we use the operator-subscript notation
\begin{align}
N_P^{<}(\Lambda)&:=\#\{j\geq0:\lambda_j(P)<\Lambda\},\nonumber\\
\mathcal{N}_P(\Lambda)&:=\#\{j\geq0:\lambda_j(P)\leq\Lambda\}. \nonumber
\end{align}

\begin{notation}[Spectral conventions]
We use
\[
\Delta_g=\operatorname{div}_g\nabla,
\qquad
-\Delta_g\geq0.
\]
For a compact Riemannian manifold $(M,g)$, define
\[
\lambda_j(M,g):=\lambda_j(-\Delta_g),
\qquad j\in\mathbb Z_{\geq0}.
\]
For a compact Alexandrov space $X$ carrying its canonical
nonnegative Laplacian, define
\[
\lambda_j(X):=\lambda_j(-\Delta_X).
\]
All eigenvalues are counted with multiplicity and indexed from $0$.
\end{notation}

Let $(M^n,g)$ be a complete Riemannian manifold with nonnegative Ricci curvature. For $d\geq0$, denote by $\mathscr{H}_d(M^n)$ the vector space of harmonic functions on $M^n$ with polynomial growth of degree at most $d$. Yau's original finite-dimensionality question \cite[Problem~48, pp.~275--319]{Yau-OpenProblems1992} asks whether
\begin{align}
\dim \mathscr{H}_d(M^n)<\infty \nonumber
\end{align}
for every fixed growth degree $d\geq 0$. That formulation also asks whether $\mathbb R^n$ has the maximal dimension of polynomial-growth harmonic functions. 

In $1998$, Colding and Minicozzi proved the following Weyl-type bound \cite{CM-Weyl}.
\begin{theorem}[Colding--Minicozzi]\label{thm-CM-weyl-type}
If $(M^n,g)$ is complete and $\Rc\geq0$, then there exists $C(n)>0$ such that, for every $d\geq1$,
\begin{align}
\dim \mathscr{H}_d(M^n)\leq C(n)d^{n-1}. \nonumber
\end{align}
\end{theorem}

The same circle of ideas includes Colding--Minicozzi's work on harmonic functions with polynomial growth and their work on harmonic sections \cite{CM-JDG, CM-Ann, CM-CPAM}, also see \cite{P.Li}. Together with Theorem \ref{thm-CM-weyl-type}, these results give the correct Weyl growth exponent in the following sense.

Here ``Weyl growth exponent'' means only the power-law exponent predicted by Weyl's law, not the sharp leading Weyl constant. Let $(N^{n-1},h)$ be a closed Riemannian manifold with Riemannian metric $h$ and suppose $\Rc(h)\geq (n-2)h$. On the punctured metric cone
\begin{align}
C(N)=(0,\infty)\times N,\qquad h_C=dr^2+r^2h, \nonumber
\end{align}
the Ricci tensor is nonnegative.

If $-\Delta_h\varphi_i=\lambda_i(-\Delta_h)\varphi_i$, then a separated function $r^\alpha\varphi_i(\theta)$ is harmonic on the cone if and only if
\begin{align}
\alpha(\alpha+n-2)=\lambda_i(-\Delta_h). \label{eq-intro-indicial}
\end{align}

By (\ref{eq-intro-indicial}), we have 
\begin{align}
\dim \mathscr{H}_k(C(N))=\mathcal{N}_{-\Delta_h}\big(k(k+n-2)\big). \label{eq-intro-cone-count}
\end{align}

Note Weyl's law gives
\begin{align}
\lim_{\Lambda\rightarrow\infty}\frac{\mathcal{N}_{-\Delta_h}(\Lambda)}{\Lambda^{\frac{n- 1}{2}}}= \frac{\omega_{n- 1}\operatorname{Vol}(N)}{(2\pi)^{n- 1}},\nonumber
\end{align}
where $\omega_{n- 1}$ is the Euclidean volume of the unit ball in $\mathbb R^{n- 1}$. Thus the Weyl order of the counting function is $\Lambda^{\frac{n- 1}{2}}$. 

In the cone discussion above, evaluating at the round thresholds $\Lambda=k(k+n-2)\sim k^2$ gives the order $k^{n-1}$ for $\mathcal{N}_{-\Delta_h}(k(k+n-2))$, equivalently for the cone dimension in (\ref{eq-intro-cone-count}). Therefore Theorem \ref{thm-CM-weyl-type} has the sharp Weyl growth exponent $d^{n-1}$.

The sharp comparison part of Yau's question asks whether the Euclidean model gives the sharp upper bound at the integer degrees.
\begin{question}[Yau's sharp dimension question]\label{ques-Yau-sharp-dimension}
Let $(M^n,g)$ be complete with $\Rc\geq0$. Is it true that
\begin{align}
\dim \mathscr{H}_d(M^n)\leq \dim \mathscr{H}_d(\mathbb R^n) , \quad \quad \forall d\in \mathbb{Z}^+? \label{eq-Yau-sharp-bound}
\end{align}
Moreover, if equality holds for some $d\in\mathbb{Z}^+$, must $M^n$ be isometric to $\mathbb R^n$?
\end{question}

\begin{remark}\label{rem:Yau's sharp upper bound}
{When $n=2$, Li--Tam give an affirmative answer to the corresponding surface question \cite[Theorem 4.2]{LT91}. For $d=1$ and general $n$, the inequality follows from Li--Tam \cite{LT-linear}, and the rigidity is due to Cheeger--Colding--Minicozzi \cite{CCM}. 

Donnelly \cite{Donnelly} showed, for some $\epsilon>0$ there is a complete manifold $(M^5,g)$ with $\Rc\geq0$ such that
\begin{align}
\dim\mathscr{H}_{2-\epsilon}(M^5)>\dim\mathscr{H}_{2-\epsilon}(\mathbb R^5). \nonumber
\end{align}
This does not disprove the integer-degree version (\ref{eq-Yau-sharp-bound}), and the integer-degree problem remains a natural sharp form of the question.
}
\end{remark}

For the Euclidean cone $\mathbb R^n=C(\mathbb S^{n-1})$, (\ref{eq-intro-cone-count}) gives the usual dimension of harmonic polynomials of degree at most $k$. Hence the sharp integer-degree form of Yau's question suggests the following question, which can be viewed as a compact cross-section version of Question \ref{ques-Yau-sharp-dimension}.

\begin{question}\label{ques-Yau-sharp-compac}
{Assume $(N^{n- 1}, h)$ has $\Rc(h)\geq (n-2)h$, do we have 
\begin{align}
\mathcal{N}_{-\Delta_h}\big(k(k+n-2)\big)\leq \mathcal{N}_{-\Delta_{\mathbb S^{n-1}}}\big(k(k+n-2)\big),\qquad \forall k\in\mathbb{Z}_{\geq0}? \label{eq-intro-cross-section-counting-question}
\end{align}
And if the equality in (\ref{eq-intro-cross-section-counting-question}) holds for some $k\in \mathbb{Z}^+$, must $N^{n- 1}$ be isometric to $\mathbb{S}^{n-1}$?
}
\end{question}

\begin{remark}\label{rem:CM-Polya}
{This point of view is already present in \cite[pp.~260--262]{CM-Weyl}. Colding and Minicozzi \cite[Page $261$]{CM-Weyl} (also see \cite[Conjecture $3$]{Milman}) pointed out that 
\begin{quote}
``it has been a long standing open problem whether for a closed manifold $N^{n-1}$ with Ricci curvature greater or equal to $n-2$, 
\begin{align}
\lambda_i(\mathbb S^{n-1})\leq \lambda_i(N^{n-1}),\qquad i\geq1, \label{eigenvalue-ineq-one-to-one}
\end{align} 
with equality for some positive integer $i= \max\{j\geq i| \lambda_j(\mathbb{S}^{n- 1})= \lambda_i(\mathbb{S}^{n-1})\}$ if and only if $N^{n-1}$ is isometric to $\mathbb{S}^{n-1}$."
\end{quote}

Aryan's recent work gives counterexamples to (\ref{eigenvalue-ineq-one-to-one}) in the conjecture for $n\geq 5$ \cite{Aryan2026}. We show that (\ref{eigenvalue-ineq-one-to-one}) does not hold for $n= 4$ in Section \ref{sec:harmonic-growth}. 

Note that the rigidity part of the above open problem is equivalent to the rigidity part of Question \ref{ques-Yau-sharp-compac}, and Question \ref{ques-Yau-sharp-compac} remains open for $n\geq 3$. In this paper, we answer Question \ref{ques-Yau-sharp-compac} affirmatively in the case $n=3$ and prove (\ref{eigenvalue-ineq-one-to-one}) in this dimension.
}
\end{remark}

This inequality (\ref{eq-intro-cross-section-counting-question}) should be understood as a \emph{spectral counting rigidity statement}. It compares entire eigenvalue distributions rather than individual eigenvalues.

The estimate (\ref{eq-intro-cross-section-counting-question}) is also a P\'olya-type spectral comparison. P\'olya's classical conjecture asks that the Dirichlet and Neumann eigenvalue counting functions of a Euclidean domain be bounded from above and from below, respectively, by the leading Weyl term \cite{Polya1954,Polya}; here the comparison model is the round sphere rather than a Euclidean domain.

Thus, in the present paper the phrase ``P\'olya-type'' is meant in this precise finite-counting sense: one asks for a non-asymptotic comparison at the exact round thresholds $k(k+n-2)$, not merely for this Weyl-order growth exponent.

\begin{remark}\label{rem Weyl law and finite idex}
{Weyl's law and Bishop--Gromov volume comparison imply that a non-round metric with the relevant Ricci lower bound has fewer eigenvalues than the round sphere at sufficiently high energy, but this asymptotic information does not decide the exact finite thresholds in (\ref{eq-intro-cross-section-counting-question}). The difficulty is precisely finite-index spectral comparison.
}
\end{remark}

One main result of this paper is the full finite-index comparison when the cross-section dimension is two.
\begin{theorem}\label{thm:main-eigenvalue}
Let $(S^2,g)$ be a smooth Riemannian manifold with sectional curvature $K_g\geq1$. Then
\begin{align}
\lambda_i(S^2,g)\geq\lambda_i(\mathbb S^2,g_{\mathrm{round}}),\qquad i\geq1. \label{eq-main-comparison}
\end{align}
Moreover, if equality holds in (\ref{eq-main-comparison}) for some $i_0\geq1$, then $(S^2,g)$ is isometric to $(\mathbb S^2,g_{\mathrm{round}})$.
\end{theorem}

One corollary of the above theorem is the following counting comparison result, which answers Question \ref{ques-Yau-sharp-compac} affirmatively.
\begin{cor}\label{cor:main-corollary}
Let $(S^2,g)$ be a smooth Riemannian manifold with sectional curvature $K_g\geq1$. Then
\begin{align}
\mathcal{N}_{-\Delta_g}\big(k(k+1)\big)\leq \mathcal{N}_{-\Delta_{\mathbb S^2}}\big(k(k+1)\big),\qquad \forall k\in\mathbb{Z}_{\geq0}. \label{eq-intro-cross-section-counting-question-1}
\end{align}
And if the equality in (\ref{eq-intro-cross-section-counting-question-1}) holds for some $k\in \mathbb{Z}^+$, then $(S^2,g)$ is isometric to $(\mathbb S^2,g_{\mathrm{round}})$.
\end{cor}

To study Question \ref{ques-Yau-sharp-dimension} for $(M^3,g)$, we establish the following strengthening of Corollary \ref{cor:main-corollary}.
\begin{theorem}\label{thm:main-counting-Aleandorff}
Let \(X\) be one Alexandrov two-sphere with curvature \(\geq1\). Then
\begin{align}
\mathcal{N}_{-\Delta_X}\big(l(l+1)\big)\leq \mathcal{N}_{-\Delta_{\mathbb S^2}}\big(l(l+1)\big),\qquad \forall l\in\mathbb{Z}_{\geq0}. \label{ineq:alex-counting-equality}
\end{align}
Moreover, if equality holds in \eqref{ineq:alex-counting-equality} for some $l\in \mathbb{Z}^+$, then $X$ is isometric to $(\mathbb S^2,g_{\mathrm{round}})$.
\end{theorem}

\begin{remark}\label{rem:counting-rigidity}
{Well-known example \ref{ex:football-single-eigenvalue-equality} shows that a \textbf{nonround} Alexandrov two-sphere $X$ with curvature \(\geq1\) may still satisfy \(\lambda_1(X)=\lambda_1(\mathbb{S}^2)= 2\). Thus the rigidity part in Theorem \ref{thm:main-eigenvalue} does not hold for Alexandrov spheres; the correct hypothesis is equality of the whole finite counting function at a round cluster threshold.
}
\end{remark}

As a consequence of Theorem \ref{thm:main-counting-Aleandorff}, we obtain the following sharp dimension estimate for polynomial growth harmonic functions on three-dimensional manifolds with nonnegative sectional curvature and positive asymptotic volume ratio.
\begin{theorem}\label{thm:three-dim-positive-avr-bound-1}
Let \((M^3,g)\) be a complete Riemannian manifold with \(K_g\geq0\) and \(\operatorname{AVR}(M)>0\).  Then
\begin{align}
 \dim\mathscr H_d(M)\leq\dim\mathscr H_d(\mathbb R^3)=(d+1)^2,
 \qquad d\in\mathbb Z_{\geq0}. \label{eq:positive-avr-sharp-bound-1}
\end{align}
Moreover, if equality holds in \eqref{eq:positive-avr-sharp-bound-1} for some $d\in \mathbb{Z}^+$, then \((M^3,g)\) is isometric to $\mathbb R^3$.
\end{theorem}

\begin{remark}\label{rem:lower-bound-of-dim-hf}
{From \cite{Xu-three-circle}, we know $\dim\mathscr H_d(M^3)\geq 2$, i.e. there exists a non-constant harmonic function of polynomial growth on such $(M^3, g)$. 
}
\end{remark}

The paper is written deliberately as a progression of languages.  The underlying algebra is always the same: one constructs a first-order operator, compares the two partner operators
\begin{align}
T^*T+c \quad\text{and}\quad TT^*+c, \nonumber
\end{align}
uses curvature to obtain a one-sided inequality, and then converts the kernel dimension of $T$ into a finite counting recursion by min--max.  What changes from one stage to the next is not the formal spectral mechanism, but the language needed to make the objects canonical. This general mechanism is discussed in Section \ref{sec:abstract-ladder-counting}.

\begin{enumerate}
\item[Part $1$]. It contains section $3--6$, and it deals with rotationally symmetric metrics on $S^{n-1}$.

For a rotationally symmetric metric, spherical-harmonic decomposition reduces the Laplacian to a family of one-dimensional radial operators.  A first-order factorization, together with the curvature inequalities, the identity
\[
\mathrm{dim}(\ker \mathbf B_m)=1,
\qquad
\ker \mathbf B_m^*=\{0\},
\]
and min--max yields a counting recursion in the angular degree.  Summing over the spherical-harmonic multiplicities gives the sharp comparison with the round sphere, as well as the corresponding rigidity statement.  When \(n=3\), this is precisely the usual Fourier-mode decomposition on circles.

Section~\ref{sec:rotational} sets up the rotationally symmetric model and fixes the real Hilbert-space conventions.  Sections~\ref{sec:rotational-ladder}--\ref{sec:rotational-comparison} prove the rotationally symmetric model.  Section~\ref{sec:rotational-ladder} constructs the separated radial operators and the radial ladder identity.  Section~\ref{sec:rotational-vanishing} proves the one-dimensional kernel identity \(\dim_{\mathbb R}\ker \mathbf{B}_m=1\) and the adjoint-kernel vanishing \(\ker \mathbf{B}_m^*=\{0\}\).  Section~\ref{sec:rotational-comparison} applies min--max, sums the angular multiplicities, and proves the rotational comparison and rigidity theorem.

\item[Part $2$]. It contains
Sections~\ref{sec:riemannian-ladder}--\ref{sec:complex-geometric-proof}
and treats arbitrary smooth metrics on \(S^2\).

Without rotational symmetry, the angular modes are replaced by the line
bundles
\[
E_m=(T^{1,0}S^2)^{\otimes m}.
\]
The associated first-order operators
\[
B_m:L^2(E_m)\longrightarrow L^2(E_{m+1})
\]
satisfy a global curvature-dependent ladder identity.  The kernel formulas
\[
\dim_{\mathbb C}\ker B_m=2m+1,
\qquad
\ker B_m^*=\{0\},
\]
then give, through spectral pairing and min--max, a recursion for
\[
A_m=B_m^*B_m+m(m+1)I.
\]
Iteration yields the sharp finite counting inequalities, and the equality
case is handled by a strict min--max argument.

Section~\ref{sec:riemannian-ladder} constructs the global
two-dimensional ladder on the bundles \(E_m\), defines \(B_m\) and its
formal adjoint, identifies \(B_0^*B_0\) with \(-\Delta_g\), and proves the
Riemannian ladder identity.  Section~\ref{sec:Vanishing-theorem} proves
\[
\dim_{\mathbb C}\ker B_m=2m+1,
\qquad
\ker B_m^*=\{0\}.
\]
Section~\ref{sec:comparison-rigidity} applies the abstract counting
mechanism to the shifted operators \(A_m\), proves the ordered eigenvalue
comparison, and obtains smooth rigidity from the strict form of the ladder
inequality.

Section~\ref{sec:complex-geometric-proof} reformulates the same construction
on the Riemann surface \((S^2,J)\).  The bundles are identified with powers
of the anticanonical bundle, and the operators \(B_m\) become Dolbeault
operators through
\[
\Lambda^{0,1}T^*S^2\otimes K^{-m}\simeq K^{-m-1}.
\]
Riemann--Roch and Serre duality give the kernel dimensions, while the
Bochner--Kodaira identity supplies the curvature-dependent ladder
inequality.  This provides a shorter structural explanation of the smooth
counting argument and, more importantly, introduces the language that
continues to make sense for singular conformal metrics.

\item[Part $3$]. It contains
Sections~\ref{sec:alexandrov-preliminaries}--
\ref{sec:alexandrov-comparison-rigidity}
and extends the comparison and rigidity argument to Alexandrov
\(2\)-spheres with curvature at least \(1\).

After fixing a conformal identification of the Alexandrov sphere with
\(\mathbb{CP}^1\), the smooth Gaussian curvature is replaced by the
curvature measure \(\omega_X\), and the lower curvature bound is encoded by
the nonnegative defect measure
\[
\nu_X:=\omega_X-dA_X\geq0.
\]
The singular argument separates the comparison statement from the equality
case.  Smooth approximation and Mosco convergence preserve the one-sided
ladder inequality, whereas rigidity requires recovering more precise
information about the defect measure.  

Section~\ref{sec:alexandrov-preliminaries} fixes the conformal
parametrization
\[
F:\mathbb{CP}^1\longrightarrow X,
\]
represents the metric by a singular conformal factor
\[
g_X=e^{2u}g_0,
\]
and introduces the curvature measure, the defect measure, the singular
Hermitian bundle metrics, the Hilbert spaces \(H_m^X\), and the maximal
operators \(B_m^X\).  It also constructs the heat-regularized metrics
\(g_\tau\), proves that \(K_{g_\tau}\geq1\), and records the metric and
spectral convergence needed to pass from the smooth theorem to \(X\).

Section~\ref{sec:alexandrov-weak-bk} places the smooth and singular
operators in fixed Hilbert spaces and proves Mosco convergence of their
graphs, adjoint graphs, and shifted quadratic forms.  Passing the smooth
Bochner--Kodaira identity to the limit gives the weak singular form order
\[
\operatorname{Dom}(\mathfrak c_m^X)
 \subset \operatorname{Dom}(\mathfrak a_{m+1}^X),
\qquad
\mathfrak c_m^X\geq\mathfrak a_{m+1}^X.
\]
The same section proves the holomorphic kernel bound
\[
\ker B_m^X\subset H^0(\mathbb{CP}^1,K_{\mathbb{CP}^1}^{-m}),
\qquad
\dim_{\mathbb C}\ker B_m^X\leq2m+1,
\]
and establishes the compact-resolvent properties required by the abstract
counting recursion.

Section~\ref{sec:alexandrov-atomic-obstruction} analyzes the atomic case.
Equality in a round counting threshold first saturates every rung of the
singular ladder and produces a top-rung space of dimension \(2l+1\).  If
the curvature-defect measure had an atom at \(p\), the maximal adjoint-domain
condition would force every section in that top-rung space to vanish at
\(p\).  The resulting loss of one holomorphic degree of freedom contradicts
the saturated dimension.  Consequently, equality at a round counting
threshold forces \(\nu_X\) to be atomless.

Section~\ref{sec:alexandrov-atomless-defect} treats the complementary case.
When the curvature measure is atomless, the weak form inequality can be
sharpened on the saturated top rung to the exact defect identity
\[
\mathfrak c_{l-1}^X(s,s)-\mathfrak a_l^X(s,s)
=
2l\int_X |s|_{h_{l,X}}^2\,d\nu_X,
\qquad s\in V_l^X.
\]
The proof recovers the curvature term directly from the maximal adjoint
domain, using the Poincar\'e--Lelong formula, truncation, and the vanishing of
the boundary fluxes around the zeros of the holomorphic section.

Section~\ref{sec:alexandrov-comparison-rigidity} assembles these ingredients.
Smooth approximation gives the ordered eigenvalue comparison, while the
singular ladder gives the finite counting inequalities at the round
thresholds.  In the equality case,
Section~\ref{sec:alexandrov-atomic-obstruction} first eliminates atoms.
The exact identity from
Section~\ref{sec:alexandrov-atomless-defect}, together with saturation of
the holomorphic top rung, then forces
\[
\nu_X=0.
\]
Gauss--Bonnet and the equality case of Alexandrov volume comparison finally
show that \(X\) is isometric to the unit round sphere.

\item[Part $4$]. It contains Sections~\ref{sec:S3-counterexample} and
\ref{sec:three-dimensional-yau} and describes both the higher-dimensional
limitation of the spectral comparison and its three-dimensional
harmonic-growth application.

Section~\ref{sec:S3-counterexample} shows that the two-dimensional ordered
eigenvalue comparison does not extend directly to higher-dimensional
cross-sections.  It constructs an explicit smooth conformal perturbation of
the round metric on \(S^3\) satisfying
\[
\operatorname{Ric}_g\geq2g
\]
for which a sufficiently high eigenvalue lies strictly below the
corresponding round eigenvalue.  Thus the success of the line-bundle ladder
on \(S^2\) reflects a genuinely two-dimensional structure rather than a
general consequence of a Ricci lower bound.

Section~\ref{sec:three-dimensional-yau} applies the Alexandrov counting
theorem to complete three-manifolds with nonnegative sectional curvature
and positive asymptotic volume ratio.  It first obtains the unique tangent
cone at infinity
\[
C(X),
\]
shows that the cross-section \(X\) is an Alexandrov two-sphere with curvature at
least \(1\), and applies the finite-dimensional comparison theorem for
polynomial-growth harmonic functions together with the counting estimate
of Section~\ref{sec:alexandrov-comparison-rigidity}.  This gives
\[
\dim\mathscr H_d(M)\leq
\mathcal N_{-\Delta_X}\bigl(d(d+1)\bigr)
\leq(d+1)^2
=
\dim\mathscr H_d(\mathbb R^3).
\]
If equality holds at a positive integer degree, Alexandrov counting
rigidity makes \(X\) the unit round sphere.  The tangent cone is therefore
\(\mathbb R^3\), and the equality case of Bishop--Gromov comparison forces
\(M\) itself to be isometric to \(\mathbb R^3\).  The section also explains,
through the spherical-football example, why equality of a single
Alexandrov eigenvalue is insufficient for rigidity, and concludes with a
discussion of the collapsed regime \(\operatorname{AVR}(M)=0\).

\end{enumerate}

Thus the paper proceeds from the abstract ladder-counting mechanism to a
rotational model, then to an invariant smooth construction and its
complex-geometric reformulation, and finally to the singular Alexandrov
setting required by the tangent-cone application.  The last part records
both the sharp three-dimensional consequence and the obstruction to a
straightforward higher-dimensional extension.

\subsection*{Proof flowchart for the main conclusions}
Each diagram below has a unique terminal node.  The first diagram ends at
Theorem~\ref{thm:main-eigenvalue}, and the second ends at
Theorem~\ref{thm:three-dim-positive-avr-bound-1}.  Every solid arrow records a
direct logical dependence.  The rotational model and the independent
$S^3$ counterexample are not included, since neither is an input
to these two final conclusions.

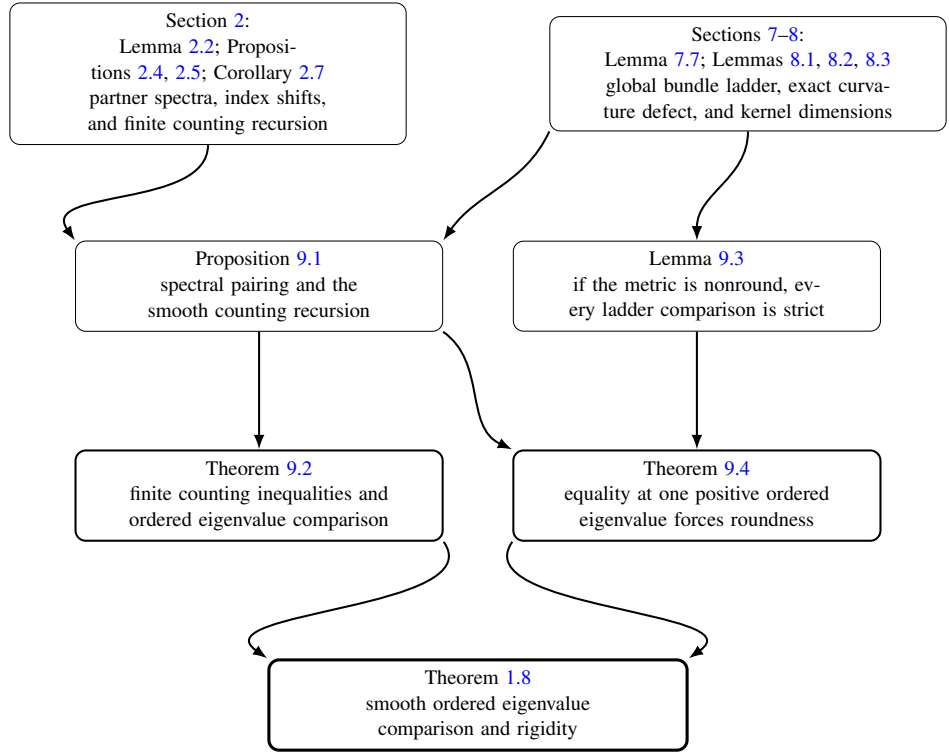
\begin{figure}[H]
\centering
\resizebox{0.98\textwidth}{!}{%
\begin{tikzpicture}[
  proofnode/.style={draw, rounded corners, fill=white, align=center,
    text width=4.75cm, minimum height=0.92cm, inner sep=4pt,
    font=\scriptsize},
  proofwide/.style={proofnode, text width=5.15cm},
  proofmilestone/.style={proofnode, thick},
  proofresult/.style={proofnode, very thick, text width=5.45cm},
  proofarrow/.style={-{Latex[length=2mm]}, thick},
  x=1cm,y=0.82cm
]
\node[proofwide] (A) at (-3.70,0)
  {Section~\ref{sec:abstract-ladder-counting}:\\
   Lemma~\ref{lem:abstract-partner-spectra};
   Propositions~\ref{prop:abstract-ladder-recursion},
   \ref{prop:abstract-ladder-iteration};
   Corollary~\ref{cor:abstract-two-sphere-ladder}\\
   partner spectra, index shifts, and finite counting recursion};

\node[proofwide] (S) at (3.70,0)
  {Sections~\ref{sec:riemannian-ladder}--\ref{sec:Vanishing-theorem}:\\
   Lemma~\ref{lem:cr-ladder};
   Lemmas~\ref{lem:conformal-invariance-Bm},
   \ref{lem:ker-Bm-dimension-new},
   \ref{lem:adjoint-kernel-new}\\
   global bundle ladder, exact curvature defect, and kernel dimensions};

\node[proofnode] (Rec) at (-3.00,-3.55)
  {Proposition~\ref{prop:crucial induction ineq}\\
   spectral pairing and the smooth counting recursion};

\node[proofnode] (Strict) at (3.00,-3.55)
  {Lemma~\ref{lem:smooth-strict-form-comparison}\\
   if the metric is nonround, every ladder comparison is strict};

\node[proofmilestone] (Comp) at (-3.00,-7.05)
  {Theorem~\ref{thm:cluster-bottom-comparison}\\
   finite counting inequalities and ordered eigenvalue comparison};

\node[proofmilestone] (Rig) at (3.00,-7.05)
  {Theorem~\ref{thm:rigidity-main}\\
   equality at one positive ordered eigenvalue forces roundness};

\node[proofresult] (Main) at (0,-10.55)
  {Theorem~\ref{thm:main-eigenvalue}\\
   smooth ordered eigenvalue comparison and rigidity};

\draw[proofarrow] (A.south) to[out=-90,in=125] (Rec.north west);
\draw[proofarrow] (S.south west) to[out=-115,in=55] (Rec.north east);
\draw[proofarrow] (S.south) to[out=-90,in=90] (Strict.north);
\draw[proofarrow] (Rec.south) to[out=-90,in=90] (Comp.north);
\draw[proofarrow] (Rec.south east) to[out=-35,in=155] (Rig.north west);
\draw[proofarrow] (Strict.south) to[out=-90,in=90] (Rig.north);
\draw[proofarrow] (Comp.south east) to[out=-55,in=145] (Main.north west);
\draw[proofarrow] (Rig.south west) to[out=-125,in=35] (Main.north east);
\end{tikzpicture}%
}
\caption{Proof dependencies for Theorem~\ref{thm:main-eigenvalue}.  It is the
only terminal node in the diagram.}
\end{figure}

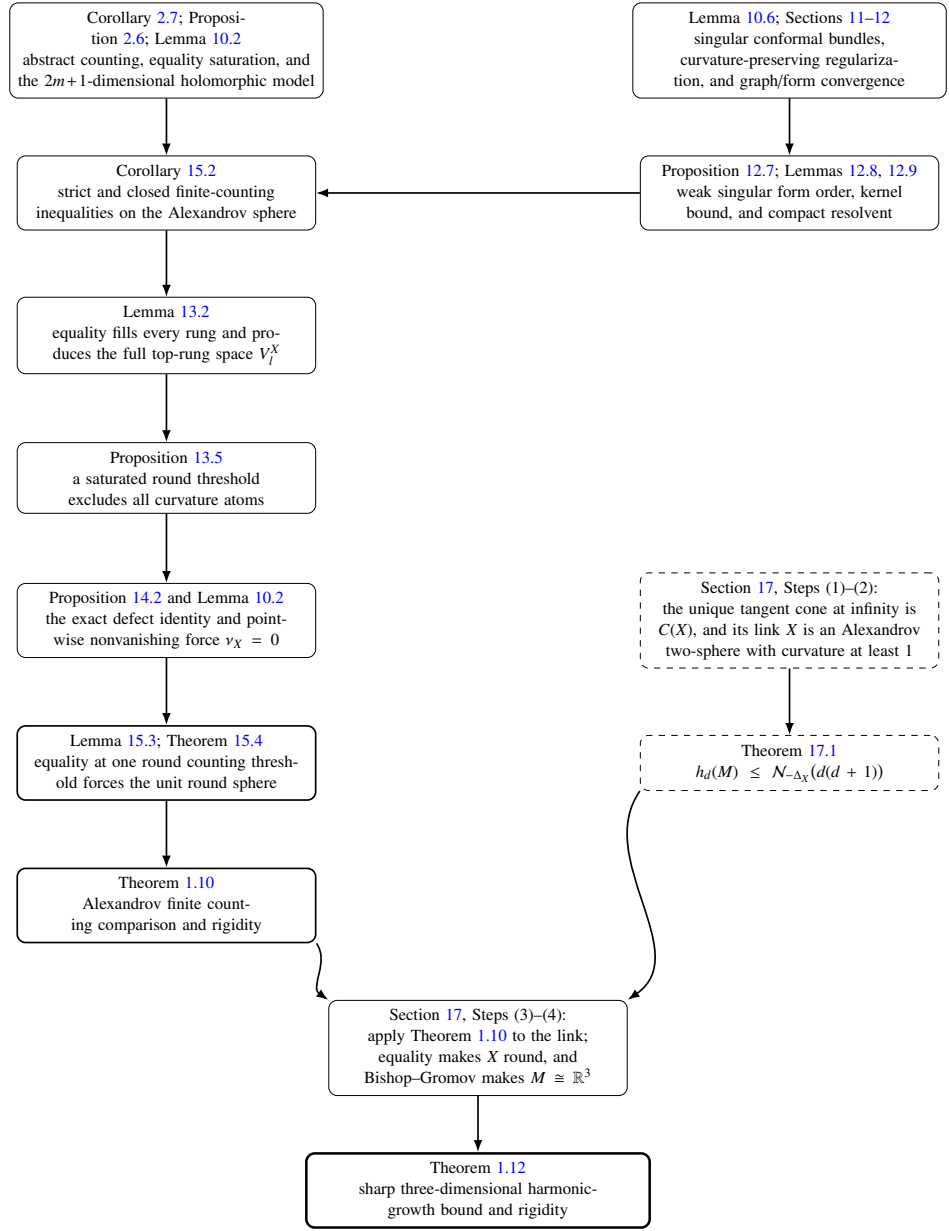
\begin{figure}[H]
\centering
\resizebox{0.98\textwidth}{!}{%
\begin{tikzpicture}[
  proofnode/.style={draw, rounded corners, fill=white, align=center,
    text width=4.70cm, minimum height=0.92cm, inner sep=4pt,
    font=\scriptsize},
  proofwide/.style={proofnode, text width=4.95cm},
  proofmilestone/.style={proofnode, thick},
  proofresult/.style={proofnode, very thick, text width=5.45cm},
  proofexternal/.style={proofnode, dashed},
  proofarrow/.style={-{Latex[length=2mm]}, thick},
  x=1cm,y=0.67cm
]
\node[proofwide] (Abs) at (-5.20,0)
  {Corollary~\ref{cor:abstract-two-sphere-ladder};
   Proposition~\ref{prop:abstract-ladder-saturation};
   Lemma~\ref{lem:CP1-anticanonical-cohomology}\\
   abstract counting, equality saturation, and the
   \(2m+1\)-dimensional holomorphic model};

\node[proofwide] (Sing) at (5.20,0)
  {Lemma~\ref{lem:complex-curvature-ladder};
   Sections~\ref{sec:alexandrov-preliminaries}--\ref{sec:alexandrov-weak-bk}\\
   singular conformal bundles, curvature-preserving regularization,
   and graph/form convergence};

\node[proofnode] (Count) at (-5.20,-3.55)
  {Corollary~\ref{cor:alex-counting-comparison-by-approx}\\
   strict and closed finite-counting inequalities on the Alexandrov sphere};

\node[proofnode] (Weak) at (5.20,-3.55)
  {Proposition~\ref{prop:weak-Bochner-Kodaira};
   Lemmas~\ref{lem:atomic-kernel-holomorphic},
   \ref{lem:alex-ladder-compactness}\\
   weak singular form order, kernel bound, and compact resolvent};

\node[proofnode] (Sat) at (-5.20,-7.10)
  {Lemma~\ref{lem:atomic-counting-saturation}\\
   equality fills every rung and produces the full top-rung space
   \(V_l^X\)};

\node[proofnode] (Atom) at (-5.20,-10.65)
  {Proposition~\ref{prop:atomic-alexandrov-counting-obstruction}\\
   a saturated round threshold excludes all curvature atoms};

\node[proofnode] (Defect) at (-5.20,-14.20)
  {Proposition~\ref{prop:atomless-BK-defect-inequality} and
   Lemma~\ref{lem:CP1-anticanonical-cohomology}\\
   the exact defect identity and pointwise nonvanishing force
   \(\nu_X=0\)};

\node[proofmilestone] (AlexRig) at (-5.20,-17.75)
  {Lemma~\ref{lem:alex-zero-defect-roundness};
   Theorem~\ref{claim:alexandrov-link-rigidity-needed}\\
   equality at one round counting threshold forces the unit round sphere};

\node[proofmilestone] (AlexMain) at (-5.20,-21.30)
  {Theorem~\ref{thm:main-counting-Aleandorff}\\
   Alexandrov finite counting comparison and rigidity};

\node[proofexternal] (Cone) at (5.20,-14.20)
  {Section~\ref{sec:three-dimensional-yau}, Steps (1)--(2):\\
   the unique tangent cone at infinity is \(C(X)\), and its link
   \(X\) is an Alexandrov two-sphere with curvature at least \(1\)};

\node[proofexternal] (Huang) at (5.20,-17.75)
  {Theorem~\ref{thm:huang-finite-dimensional-comparison}\\
   \(h_d(M)\leq
   \mathcal N_{-\Delta_X}\bigl(d(d+1)\bigr)\)};

\node[proofnode] (Apply) at (0,-24.85)
  {Section~\ref{sec:three-dimensional-yau}, Steps (3)--(4):\\
   apply Theorem~\ref{thm:main-counting-Aleandorff} to the link;
   equality makes \(X\) round, and Bishop--Gromov makes
   \(M\cong\mathbb R^3\)};

\node[proofresult] (Yau) at (0,-28.40)
  {Theorem~\ref{thm:three-dim-positive-avr-bound-1}\\
   sharp three-dimensional harmonic-growth bound and rigidity};

\draw[proofarrow] (Abs.south) to[out=-90,in=90] (Count.north);
\draw[proofarrow] (Sing.south) to[out=-90,in=90] (Weak.north);
\draw[proofarrow] (Weak.west) to[out=180,in=0] (Count.east);
\draw[proofarrow] (Count.south) to[out=-90,in=90] (Sat.north);
\draw[proofarrow] (Sat.south) to[out=-90,in=90] (Atom.north);
\draw[proofarrow] (Atom.south) to[out=-90,in=90] (Defect.north);
\draw[proofarrow] (Defect.south) to[out=-90,in=90] (AlexRig.north);
\draw[proofarrow] (AlexRig.south) to[out=-90,in=90] (AlexMain.north);

\draw[proofarrow] (Cone.south) to[out=-90,in=90] (Huang.north);

\draw[proofarrow] (AlexMain.south east) to[out=-55,in=145] (Apply.north west);
\draw[proofarrow] (Huang.south west) to[out=-125,in=35] (Apply.north east);
\draw[proofarrow] (Apply.south) to[out=-90,in=90] (Yau.north);
\end{tikzpicture}%
}
\caption{Proof dependencies for Theorem~\ref{thm:three-dim-positive-avr-bound-1}.
Theorem~\ref{thm:main-counting-Aleandorff} is an intermediate input, and the
harmonic-growth theorem is the only terminal node.}
\end{figure}

\section{The abstract ladder-counting mechanism}\label{sec:abstract-ladder-counting}

The purpose of this section is to isolate the functional-analytic mechanism that is common to all the geometric realizations used below.  The geometry determines the Hilbert spaces, the closed first-order operators, their domains, their kernel dimensions, and the form inequality between consecutive levels.  Once those facts have been verified, the remaining spectral argument is independent of whether the operators arise from a rotational decomposition, a smooth Riemannian bundle, a Dolbeault complex, or a singular Alexandrov metric.

The common mechanism has six steps.
\begin{enumerate}
\item The positive spectra of the partner operators $T^*T$ and $TT^*$ agree, including multiplicity.
\item The space $\ker T$ gives an unpaired block at the bottom of $T^*T$ and therefore produces an index shift.
\item The condition $\ker T^*=\{0\}$ removes the corresponding unpaired block on the $TT^*$ side.
\item A quadratic-form inequality compares the $TT^*$ partner at one level with the $T^*T$ operator at the next level.
\item The min--max principle converts this comparison into a recursion for strict and closed counting functions.
\item The recursion terminates when the shift at the final level reaches the prescribed round threshold.
\end{enumerate}
This is precisely the spectral mechanism behind the different geometric stages of the proof.

All Hilbert spaces in this section may be real or complex.  We write $\mathbb F\in\{\mathbb R,\mathbb C\}$ for the ground field, and all dimensions and spectral multiplicities are taken over $\mathbb F$.

\begin{notation}\label{not:abstract-ladder-package}
Let $\{H_m\}_{m\geq0}$ be Hilbert spaces, and let
\begin{align}
B_m:\operatorname{Dom}(B_m)\subset H_m\longrightarrow H_{m+1}, \qquad m\geq0, \nonumber
\end{align}
be closed densely defined operators.  Let
\begin{align}
0\leq\alpha_0<\alpha_1<\alpha_2<\cdots . \nonumber
\end{align}
Define the closed quadratic forms
\begin{align}
\mathfrak a_m(u,v) &:=\langle B_mu,B_mv\rangle_{H_{m+1}} +\alpha_m\langle u,v\rangle_{H_m}, &\operatorname{Dom}(\mathfrak a_m)&=\operatorname{Dom}(B_m), \nonumber\\
\mathfrak c_m(s,t) &:=\langle B_m^*s,B_m^*t\rangle_{H_m} +\alpha_m\langle s,t\rangle_{H_{m+1}}, &\operatorname{Dom}(\mathfrak c_m)&=\operatorname{Dom}(B_m^*). \nonumber
\end{align}
Let $A_m$ and $C_m$ be the associated self-adjoint operators.  Equivalently,
\begin{align}
A_m=B_m^*B_m+\alpha_m I \quad\text{in }H_m, \qquad C_m=B_mB_m^*+\alpha_m I \quad\text{in }H_{m+1}. \nonumber
\end{align}

We assume that, for every $m\geq0$, both $A_m$ and $C_m$
have compact resolvent.  Since $A_m,C_m\geq0$, this is equivalent
to requiring the bounded operators
\begin{align}
(A_m+I)^{-1}:H_m\longrightarrow H_m, \qquad (C_m+I)^{-1}:H_{m+1}\longrightarrow H_{m+1} \nonumber
\end{align}
to be compact.

Set
\begin{align}
r_m:=\dim_{\mathbb F}\ker B_m, \qquad s_m:=\dim_{\mathbb F}\ker B_m^*. \nonumber
\end{align}
These dimensions are finite because the eigenspaces of $A_m$ and $C_m$ at $\alpha_m$ are finite dimensional.
\end{notation}

\begin{lemma}\label{lem:abstract-partner-spectra}
In the setting of Notation~\ref{not:abstract-ladder-package}, the positive spectra of $B_m^*B_m$ and $B_mB_m^*$ agree, including multiplicity.  More precisely, for every $\lambda>0$, the map $B_m$ gives an isomorphism
\begin{align}
B_m: \ker(B_m^*B_m-\lambda I) \longrightarrow \ker(B_mB_m^*-\lambda I), \nonumber
\end{align}
whose inverse is $\lambda^{-1}B_m^*$.  Consequently,
\begin{align}
&\lambda_{j+r_m}(A_m)=\lambda_{j+s_m}(C_m), \qquad j=0,1,2,\ldots , \label{eq:abstract-partner-index-shift}\\
&N_{A_m}^{<}(\Lambda)-N_{C_m}^{<}(\Lambda) =r_m-s_m, \quad \quad \forall \Lambda> \alpha_m, \nonumber\\
&\mathcal N_{A_m}(\Lambda)-\mathcal N_{C_m}(\Lambda) =r_m-s_m, \quad \quad \forall \Lambda\geq \alpha_m,. \label{eq:abstract-partner-counting-defect}
\end{align}
\end{lemma}

\begin{proof}
Suppose that $u\in\operatorname{Dom}(B_m^*B_m)$ and
\begin{align}
B_m^*B_mu=\lambda u, \qquad \lambda>0. \nonumber
\end{align}
Then $B_mu\neq0$.  Moreover, $B_mu\in\operatorname{Dom}(B_m^*)$, and
$B_m^*(B_mu)=\lambda u\in\operatorname{Dom}(B_m)$.  Hence
\begin{align}
B_mB_m^*(B_mu)=\lambda B_mu. \nonumber
\end{align}
Conversely, if $v\in\operatorname{Dom}(B_mB_m^*)$ and
\begin{align}
B_mB_m^*v=\lambda v, \qquad \lambda>0, \nonumber
\end{align}
then $B_m^*v\neq0$ and
\begin{align}
B_m^*B_m(B_m^*v)=\lambda B_m^*v. \nonumber
\end{align}
The two maps are inverse after multiplication by $\lambda^{-1}$.  Thus the positive eigenspaces have the same dimensions.

The eigenspace of $A_m$ at $\alpha_m$ is $\ker B_m$, while the eigenspace of $C_m$ at $\alpha_m$ is $\ker B_m^*$.  Removing these two bottom eigenspaces leaves the same ordered sequence of shifted positive eigenvalues.  This proves \eqref{eq:abstract-partner-index-shift}.  The two counting identities in \eqref{eq:abstract-partner-counting-defect} follow by counting the bottom eigenspaces and the paired positive spectrum separately.
\end{proof}


The quadratic-form order \(C_m\geq A_{m+1}\) will always mean both the
form-domain inclusion and the inequality displayed below.  On a
singular space, the formal commutator identity on smooth sections of the
regular set does not by itself determine the closed realizations of the two
operators.  Thus both the domain inclusion and the form inequality must be
established at the level of the closed quadratic forms.

\begin{lemma}\label{lem:abstract-adjoint-kernel-gap}
Fix $m\geq0$.  Assume that the consecutive forms satisfy
\begin{align}
\operatorname{Dom}(\mathfrak c_m) &\subset\operatorname{Dom}(\mathfrak a_{m+1}), \nonumber\\
\mathfrak c_m(v,v) &\geq\mathfrak a_{m+1}(v,v), \qquad v\in\operatorname{Dom}(\mathfrak c_m). \label{eq:abstract-ladder-form-order}
\end{align}
Then $\ker B_m^*=\{0\}$.
\end{lemma}

\begin{proof}
Let $v\in\ker B_m^*$.  Then $v\in\operatorname{Dom}(\mathfrak c_m)$, so \eqref{eq:abstract-ladder-form-order} gives $v\in\operatorname{Dom}(\mathfrak a_{m+1})$ and
\begin{align}
\alpha_{m+1}\|v\|_{H_{m+1}}^2 &\leq\mathfrak a_{m+1}(v,v) \leq\mathfrak c_m(v,v) =\alpha_m\|v\|_{H_{m+1}}^2. \nonumber
\end{align}
Since $\alpha_{m+1}>\alpha_m$, one has $v=0$.
\end{proof}

\begin{prop}\label{prop:abstract-ladder-recursion}
Fix $m\geq0$ and assume \eqref{eq:abstract-ladder-form-order}.  Then, for every $j\geq0$,
\begin{align}
\lambda_{j+r_m}(A_m) =\lambda_j(C_m) \geq\lambda_j(A_{m+1}). \label{eq:abstract-ladder-eigenvalue-step}
\end{align}
Consequently, for every $\Lambda\in\mathbb R$,
\begin{align}
N_{A_m}^{<}(\Lambda) &\leq r_m+N_{A_{m+1}}^{<}(\Lambda), \quad \quad 
\mathcal N_{A_m}(\Lambda) &\leq r_m+\mathcal N_{A_{m+1}}(\Lambda). \label{eq:abstract-ladder-counting-step}
\end{align}
\end{prop}

\begin{proof}
Lemma~\ref{lem:abstract-adjoint-kernel-gap} gives $\ker B_m^*=\{0\}$.  Therefore Lemma~\ref{lem:abstract-partner-spectra} gives the equality in \eqref{eq:abstract-ladder-eigenvalue-step}.

The form inequality \eqref{eq:abstract-ladder-form-order} and the min--max principle give
\begin{align}
\lambda_j(C_m)\geq\lambda_j(A_{m+1}), \qquad j\geq0. \nonumber
\end{align}
This proves \eqref{eq:abstract-ladder-eigenvalue-step}. 

Equivalently, the number of eigenvalues of $C_m$ below, or below and at, a fixed threshold does not exceed the corresponding number for $A_{m+1}$, this proves \eqref{eq:abstract-ladder-counting-step}.
\end{proof}

\begin{prop}\label{prop:abstract-ladder-iteration}
Let $0\leq q\leq l$, and assume that \eqref{eq:abstract-ladder-form-order} holds for $m=q,q+1,\ldots,l-1$.  Then
\begin{align}
N_{A_q}^{<}(\alpha_l) &\leq\sum_{m=q}^{l-1}r_m, \nonumber\\
\mathcal N_{A_q}(\alpha_l) &\leq\sum_{m=q}^{l}r_m. \label{eq:abstract-ladder-terminal-counts}
\end{align}
Here the first sum is understood to be zero when $q=l$.  
\end{prop}

\begin{proof}
Since
\begin{align}
A_l=B_l^*B_l+\alpha_l I\geq\alpha_l I, \nonumber
\end{align}
one has
\begin{align}
N_{A_l}^{<}(\alpha_l)=0, \qquad \mathcal N_{A_l}(\alpha_l)=\dim_{\mathbb F}\ker B_l=r_l. \nonumber
\end{align}
Iterating the two inequalities in \eqref{eq:abstract-ladder-counting-step} from level $q$ to level $l$ gives \eqref{eq:abstract-ladder-terminal-counts}. 
\end{proof}

\begin{prop}\label{prop:abstract-ladder-saturation}
Let $l\in\mathbb Z^+$, and assume that \eqref{eq:abstract-ladder-form-order} holds for $m=0,1,\ldots,l-1$.  Suppose that
\begin{align}
&r_m\leq\rho_m, \qquad 0\leq m\leq l, \nonumber\\
&\mathcal N_{A_0}(\alpha_l) =\sum_{m=0}^{l}\rho_m. \label{eq:abstract-ladder-saturation-hypothesis}
\end{align}
Then
\begin{align}
&r_m=\rho_m, \qquad 0\leq m\leq l, \label{eq:abstract-ladder-maximal-kernels}\\
&\mathcal N_{A_m}(\alpha_l) &=r_m+\mathcal N_{A_{m+1}}(\alpha_l), \nonumber\\
&\mathcal N_{C_m}(\alpha_l) &=\mathcal N_{A_{m+1}}(\alpha_l), \qquad 0\leq m\leq l-1. \label{eq:abstract-ladder-saturated-steps}
\end{align}
Define 
\begin{align}
E_l&:=\ker(C_{l-1}-\alpha_l I), \quad \quad V_l:=\ker B_l\cap\operatorname{Dom}(B_{l-1}^*). \nonumber
\end{align}
Then
\begin{align}
&E_l=V_l=\ker B_l, \qquad \dim_{\mathbb F}E_l=\rho_l, \label{eq:abstract-ladder-top-space}\\
&\mathfrak c_{l-1}(s,s) =\mathfrak a_l(s,s) =\alpha_l\|s\|_{H_l}^2, \quad \quad \forall s\in E_l. \label{eq:abstract-ladder-top-form-equality}
\end{align}
\end{prop}

\begin{proof}
Set
\begin{align}
n_m:=\mathcal N_{A_m}(\alpha_l), \qquad q_m:=\mathcal N_{C_m}(\alpha_l). \nonumber
\end{align}
By Lemma~\ref{lem:abstract-adjoint-kernel-gap}, $\ker B_m^*=\{0\}$ for $0\leq m\leq l-1$.  Since $\alpha_l\geq\alpha_m$, Lemma~\ref{lem:abstract-partner-spectra} gives
\begin{align}
n_m=r_m+q_m, \qquad 0\leq m\leq l-1. \nonumber
\end{align}
The form inequality gives $q_m\leq n_{m+1}$, while
\begin{align}
n_l=\mathcal N_{A_l}(\alpha_l)=r_l. \nonumber
\end{align}
Therefore
\begin{align}
n_0 &\leq r_0+n_1 \leq r_0+r_1+n_2 \leq\cdots \leq\sum_{m=0}^{l}r_m \leq\sum_{m=0}^{l}\rho_m. \nonumber
\end{align}
The first and last terms are equal by \eqref{eq:abstract-ladder-saturation-hypothesis}.  Hence equality holds at every stage.  This proves \eqref{eq:abstract-ladder-maximal-kernels} and \eqref{eq:abstract-ladder-saturated-steps}.

At the last step,
\begin{align}
\mathcal N_{C_{l-1}}(\alpha_l) =\mathcal N_{A_l}(\alpha_l) =r_l=\rho_l. \nonumber
\end{align}
Moreover,
\begin{align}
C_{l-1}\geq A_l\geq\alpha_l I \nonumber
\end{align}
in the quadratic-form sense.  Thus every eigenvalue of $C_{l-1}$ counted by $\mathcal N_{C_{l-1}}(\alpha_l)$ is equal to $\alpha_l$, and consequently
\begin{align}
\dim_{\mathbb F}E_l=r_l. \nonumber
\end{align}

Let $s\in E_l$.  Then $s\in\operatorname{Dom}(C_{l-1})\subset\operatorname{Dom}(\mathfrak c_{l-1})$ and
\begin{align}
\mathfrak c_{l-1}(s,s)=\alpha_l\|s\|_{H_l}^2. \nonumber
\end{align}
Using the form inequality and the definition of $\mathfrak a_l$, we obtain
\begin{align}
\alpha_l\|s\|_{H_l}^2 &\leq\mathfrak a_l(s,s) \leq\mathfrak c_{l-1}(s,s) =\alpha_l\|s\|_{H_l}^2. \nonumber
\end{align}
Hence $B_ls=0$, and $s\in\operatorname{Dom}(B_{l-1}^*)$.  Therefore
\begin{align}
E_l\subset V_l\subset\ker B_l. \nonumber
\end{align}
The first and last spaces have the same dimension $r_l$, so all three spaces coincide.  This proves \eqref{eq:abstract-ladder-top-space}, and the same argument gives \eqref{eq:abstract-ladder-top-form-equality}.
\end{proof}

\begin{cor}\label{cor:abstract-two-sphere-ladder}
Assume that
\begin{align}
\alpha_m=m(m+1), \qquad r_m\leq2m+1, \qquad m\geq0, \nonumber
\end{align}
and that \eqref{eq:abstract-ladder-form-order} holds at every level.  Then, for every $l\in\mathbb Z^+$,
\begin{align}
N_{A_0}^{<}\bigl(l(l+1)\bigr) \leq l^2, \quad \quad \mathcal N_{A_0}\bigl(l(l+1)\bigr) \leq(l+1)^2. \label{eq:abstract-two-sphere-threshold-counts}
\end{align}
If equality holds in the closed-counting estimate in \eqref{eq:abstract-two-sphere-threshold-counts}, then all conclusions of Proposition~\ref{prop:abstract-ladder-saturation} hold with
\begin{align}
\rho_m=2m+1, \qquad 0\leq m\leq l. \nonumber
\end{align}
\end{cor}

\begin{proof}
Proposition~\ref{prop:abstract-ladder-iteration} and the identities
\begin{align}
\sum_{m=0}^{l-1}(2m+1)=l^2, \qquad \sum_{m=0}^{l}(2m+1)=(l+1)^2 \nonumber
\end{align}
give \eqref{eq:abstract-two-sphere-threshold-counts}.  The eigenvalue statement is the strict-counting formulation of the first inequality.  The equality assertion follows from Proposition~\ref{prop:abstract-ladder-saturation}.
\end{proof}

\begin{remark}[Dictionary for the four geometric realizations]
The abstract results above will be used through the following identifications.
\begin{enumerate}
\item In the rotationally symmetric proof, $\mathbb F=\mathbb R$, the spaces $H_m$ are the radial Hilbert spaces, $B_m=\mathbf B_m$, and $\alpha_m=m(m+n-2)$.  The geometric work is the identification of the correct pole-compatible domains, the one-dimensional kernel computation $r_m=1$, and the radial curvature form inequality.  Proposition~\ref{prop:abstract-ladder-iteration} then supplies the spectral recursion.

\item In the smooth Riemannian proof, $\mathbb F=\mathbb C$, $H_m=L^2(E_m)$, $B_m$ is the closed bundle operator, and $\alpha_m=m(m+1)$.  The bundle construction, the curvature identity, and the kernel dimension $r_m=2m+1$ are the geometric inputs.  Corollary~\ref{cor:abstract-two-sphere-ladder} supplies the counting argument.

\item In the complex-geometric proof, the same ladder is written on the anticanonical powers $K^{-m}$.  Riemann--Roch and Serre duality provide the kernel information, while the Bochner--Kodaira identity provides the form inequality.  The Hilbert-space conclusion is again Corollary~\ref{cor:abstract-two-sphere-ladder}.

\item In the Alexandrov proof, $H_m=H_m^X$ and $B_m=B_m^X$ are the singular closed operators with their maximal Hilbert-space domains.  The singular analysis establishes compactness, the weak form inequality, and the cohomological estimate $r_m\leq2m+1$.  The closed-counting part of Corollary~\ref{cor:abstract-two-sphere-ladder} gives the finite bound, while Proposition~\ref{prop:abstract-ladder-saturation} gives the full top-rung equality space when the round count is saturated.
\end{enumerate}
Thus the later sections need only verify the geometric and domain-theoretic hypotheses appropriate to their setting; the partner-spectrum algebra, the index shift, the min--max recursion, and the equality propagation will not be repeated.
\end{remark}

\part{Rotationally symmetric case}

\section{Rotationally symmetric model and notation}\label{sec:rotational}

This part deals with the rotationally symmetric model in arbitrary dimension.  Its purpose is not to reduce the general theorem to a symmetric case, but to display the geometric input to the abstract ladder mechanism of Section~\ref{sec:abstract-ladder-counting} in the elementary language of separated variables.

\begin{definition}
Let $n\geq3$.  A smooth metric $g$ on $S^{n-1}$ is called rotationally symmetric if
\begin{align}
g=dr^2+f(r)^2g_{\mathbb S^{n-2}}, \qquad 0\leq r\leq a, \nonumber
\end{align}
where $g_{\mathbb S^{n-2}}$ is the unit round metric.  Smoothness at the two poles means that $f$ is smooth on $[0,a]$, positive on $(0,a)$, and
\begin{align}
f(0)=f(a)=0, \qquad f'(0)=1, \qquad f'(a)=-1, \nonumber
\end{align}
together with the usual even-order compatibility conditions at the poles.
\end{definition}

\begin{notation}
Throughout this part,
\begin{align}
g=dr^2+f(r)^2g_{\mathbb S^{n-2}}. \nonumber
\end{align}
All Hilbert spaces, inner products, forms, and dimensions are taken over the real field.  For $m\in\mathbb Z_{\geq0}$, let $\mathcal H_m(\mathbb R^{n-1})$ denote the restrictions to $\mathbb S^{n-2}$ of the homogeneous harmonic polynomials of degree $m$ on $\mathbb R^{n-1}$, and set
\begin{align}
d_{n-2,m} &:=\dim_{\mathbb R}\mathcal H_m(\mathbb R^{n-1}), & \lambda_m &:=m(m+n-3). \nonumber
\end{align}
Thus $\lambda_m$ is the Laplace eigenvalue corresponding to the space of degree-$m$ spherical harmonics on $\mathbb{S}^{n- 2}$.

We define 
\begin{align}
\mathscr D_m :=\Bigl\{u\in C^\infty([0,a]):\ &u(r)=r^m\psi(r^2) =(a-r)^m\varphi\bigl((a-r)^2\bigr) \nonumber\\
&\text{for some }\psi,\varphi\in C^\infty(\mathbb R)\Bigr\}. \nonumber
\end{align}
The radial Hilbert space is
\begin{align}
H :=L^2\bigl((0,a),f^{n-2}\,dr;\mathbb R\bigr), \qquad \langle u,v\rangle_H :=\int_0^a uvf^{n-2}\,dr. \label{eq:H-inner}
\end{align}
For later use, write
\begin{align}
K_{\mathrm{rad}} :=-\frac{f''}{f}, \qquad K_{\mathrm{tan}} :=\frac{1-(f')^2}{f^2} \qquad\text{on }(0,a). \nonumber
\end{align}
When $n\geq4$, these are respectively the radial and tangential sectional curvatures.  When $n=3$, $K_{\mathrm{tan}}$ is the auxiliary radial quantity which occurs in the separated form comparison.  The pole compatibility of $f$ implies that both quantities extend smoothly to the poles.
\end{notation}

\begin{definition}\label{def:Lm-operator}
For $m\geq0$, define on $\mathscr D_m$ the densely defined nonnegative symmetric form
\begin{align}
\mathfrak l_m[u,v]:=\int_0^a\left(u'v'+\frac{\lambda_m}{f^2}uv\right)f^{n-2}\,dr. \nonumber
\end{align}
For $m=0$, the angular term is identically zero; in particular, this definition does not involve a product of zero with a potentially divergent integral.
Lemma~\ref{lem:Hilbert-space-Lm} below shows that this form is closable in $H$.  We denote its closure, again, by $\mathfrak l_m$, and $\mathbf L_m$ denotes the nonnegative self-adjoint operator associated with the closed form $\mathfrak l_m$.  Thus
\begin{align}
\operatorname{Dom}(\mathbf L_m) :=\Bigl\{u\in\operatorname{Dom}(\mathfrak l_m):\ &\text{there exists }w\in H\text{ such that} \nonumber\\
&\mathfrak l_m[u,v]=\langle w,v\rangle_H \text{ for every }v\in\operatorname{Dom}(\mathfrak l_m)\Bigr\}, \label{eq:Lm-domain-def}
\end{align}
and $\mathbf L_m u=w$.  On the smooth core its formal expression is
\begin{align}
\mathcal L_m u :=-u''-(n-2)\frac{f'}{f}u' +\frac{m(m+n-3)}{f^2}u. \nonumber
\end{align}
\end{definition}

\begin{definition}\label{def:Bm-for-rotation-symm}
For $m\geq0$, define
\begin{align}
\mathcal B_m u :=u'-m\frac{f'}{f}u, \qquad u\in\mathscr D_m. \label{eq:Bm-rot}
\end{align}
We first regard this as the densely defined operator
\begin{align}
\mathcal B_m:\mathscr D_m\subset H\longrightarrow H. \nonumber
\end{align}
Its formal adjoint expression with respect to \eqref{eq:H-inner} is
\begin{align}
\mathcal B_m^\dagger v :=-v'-(m+n-2)\frac{f'}{f}v. \label{eq:Bm-adj-rot}
\end{align}
For $v\in\mathscr D_{m+1}$, the pole conditions imply that $\mathcal B_m^\dagger v\in H$.  Moreover, for $u\in\mathscr D_m$ and $v\in\mathscr D_{m+1}$, integration by parts gives
\begin{align}
\langle\mathcal B_m u,v\rangle_H =\langle u,\mathcal B_m^\dagger v\rangle_H. \nonumber
\end{align}
Since $\mathscr D_{m+1}$ is dense in $H$, $\mathcal B_m$ is closable.  We set
\begin{align}
\mathbf B_m :=\overline{\mathcal B_m}: \operatorname{Dom}(\mathbf B_m)\subset H\longrightarrow H, \nonumber
\end{align}
where
\begin{align}
\operatorname{Dom}(\mathbf B_m) =\overline{\mathscr D_m}^{\,\|\cdot\|_{\mathbf B_m}}, \qquad \|u\|_{\mathbf B_m}^2 :=\|u\|_H^2+\|\mathcal B_m u\|_H^2. \nonumber
\end{align}
Its Hilbert-space adjoint is denoted by
\begin{align}
\mathbf B_m^*: \operatorname{Dom}(\mathbf B_m^*)\subset H\longrightarrow H. \nonumber
\end{align}
Thus $v\in\operatorname{Dom}(\mathbf B_m^*)$ exactly when there exists $w\in H$ such that
\begin{align}
\langle\mathbf B_m u,v\rangle_H =\langle u,w\rangle_H, \qquad u\in\operatorname{Dom}(\mathbf B_m), \nonumber
\end{align}
and then $\mathbf B_m^*v=w$.  On $\mathscr D_{m+1}$, $\mathbf B_m^*$ agrees with the expression in \eqref{eq:Bm-adj-rot}.
\end{definition}

\begin{definition}\label{def:rot-bold-Am-operator}
For $m\geq0$, define the closed forms
\begin{align}
\mathfrak a_m[u,v] &:=\langle\mathbf B_m u,\mathbf B_m v\rangle_H +m(m+n-2)\langle u,v\rangle_H, & \operatorname{Dom}(\mathfrak a_m) &:=\operatorname{Dom}(\mathbf B_m), \nonumber\\
\mathfrak c_m[v,w] &:=\langle\mathbf B_m^*v,\mathbf B_m^*w\rangle_H +m(m+n-2)\langle v,w\rangle_H, & \operatorname{Dom}(\mathfrak c_m) &:=\operatorname{Dom}(\mathbf B_m^*). \nonumber
\end{align}
Let $\mathbf A_m$ and $\mathbf C_m$ be their associated self-adjoint operators.  Equivalently,
\begin{align}
\operatorname{Dom}(\mathbf A_m) &=\{u\in\operatorname{Dom}(\mathbf B_m): \mathbf B_m u\in\operatorname{Dom}(\mathbf B_m^*)\}, \nonumber\\
\mathbf A_m u &=\mathbf B_m^*\mathbf B_m u+m(m+n-2)u, \nonumber\\
\operatorname{Dom}(\mathbf C_m) &=\{v\in\operatorname{Dom}(\mathbf B_m^*): \mathbf B_m^*v\in\operatorname{Dom}(\mathbf B_m)\}, \nonumber\\
\mathbf C_m v &=\mathbf B_m\mathbf B_m^*v+m(m+n-2)v. \nonumber
\end{align}
\end{definition}

\section{Frequency decomposition and the radial ladder}\label{sec:rotational-ladder}

\begin{theorem}\label{thm:Fischer-decomposition}
Let $d\geq2$.  For $j\geq0$, let $\mathcal P_j(\mathbb R^d)$ be the homogeneous polynomials of degree $j$, and let
\begin{align}
\mathcal H_j(\mathbb R^d) :=\{P\in\mathcal P_j(\mathbb R^d):\Delta_{\mathbb R^d}P=0\}. \nonumber
\end{align}
Every $P_j\in\mathcal P_j(\mathbb R^d)$ has a unique decomposition
\begin{align}
P_j(x) =H_j(x)+|x|^2H_{j-2}(x)+\cdots+|x|^{2q}H_{j-2q}(x), \qquad q=\lfloor j/2\rfloor, \nonumber
\end{align}
where $H_{j-2\ell}\in\mathcal H_{j-2\ell}(\mathbb R^d)$.  Consequently,
\begin{align}
P_j|_{\mathbb S^{d-1}} \in\mathcal H_j(\mathbb S^{d-1})\oplus \mathcal H_{j-2}(\mathbb S^{d-1})\oplus\cdots. \label{eq:Fischer-sphere-restriction}
\end{align}
In particular, if $Y\in\mathcal H_m(\mathbb S^{d-1})$, then
\begin{align}
\int_{\mathbb S^{d-1}}P_j(\theta)Y(\theta)\,dV_{\mathbb S^{d-1}}(\theta)=0 \label{eq:Fischer-orthogonality}
\end{align}
unless $j=m+2\ell$ for some $\ell\geq0$.
\end{theorem}

\begin{proof}
The decomposition is the standard Fischer decomposition; see \cite[Proposition~5.5 and Theorem~5.7, pp.~76--77]{AxlerBourdonRamey2001}.  Restriction to the unit sphere and orthogonality of spherical harmonics of different degrees give \eqref{eq:Fischer-sphere-restriction} and \eqref{eq:Fischer-orthogonality}.
\end{proof}

\begin{lemma}\label{lem:rot-angular-coefficients-Dm}
Let $w\in C^\infty(S^{n-1};\mathbb R)$, let $m\geq0$, and let $\varphi\in\mathcal H_m(\mathbb R^{n-1})|_{\mathbb S^{n-2}}$.  Then 
\begin{align}
u(r) :=\int_{\mathbb S^{n-2}}w(r,\theta)\varphi(\theta)\,dV_{\mathbb S^{n-2}}\in \mathscr D_m. \nonumber
\end{align}
\end{lemma}

\begin{proof}
Smoothness on $(0,a)$ follows by differentiation under the integral sign.  Near $r=0$, use smooth normal coordinates $x=r\theta\in\mathbb R^{n-1}$ and write $w(r,\theta)=W_0(r\theta)$ with $W_0$ smooth.  For every $N$,
\begin{align}
W_0(r\theta) =\sum_{j=0}^{m+2N+1}r^jP_j(\theta)+O(r^{m+2N+2}), \nonumber
\end{align}
where $P_j$ is homogeneous of degree $j$.  By Theorem~\ref{thm:Fischer-decomposition}, the integral of $P_j$ against the degree-$m$ harmonic $\varphi$ vanishes unless $j=m+2\ell$.  Hence
\begin{align}
u(r)=\sum_{\ell=0}^{N}c_\ell r^{m+2\ell}+O(r^{m+2N+2}) \nonumber
\end{align}
for every $N$.  Therefore $r^{-m}u(r)$ extends smoothly across $r=0$ and has only even Taylor coefficients.  It follows that, near $r=0$,
\begin{align}
u(r)=r^m\psi_0(r^2) \nonumber
\end{align}
for a smooth function $\psi_0$.

The same argument in normal coordinates at the other pole gives
\begin{align}
u(r)=(a-r)^m\varphi_0\bigl((a-r)^2\bigr) \nonumber
\end{align}
near $r=a$.  The functions $r^{-m}u(r)$ and $(a-r)^{-m}u(r)$, viewed respectively as functions of $r^2$ and $(a-r)^2$, are smooth on their compact half-intervals and therefore admit smooth extensions to $\mathbb R$.  This is exactly the definition of $u\in\mathscr D_m$.
\end{proof}

\begin{lemma}\label{lem:Hilbert-space-Lm}
For every $m\geq0$:
\begin{enumerate}
\item the form $\mathfrak l_m$ initially defined on $\mathscr D_m$ is closable, and, with $\mathfrak l_m$ denoting its closure, the embedding
\begin{align}
\operatorname{Dom}(\mathfrak l_m)\hookrightarrow H \label{eq:compact-form-domain}
\end{align}
is compact;
\item $\mathbf L_m$ is nonnegative, self-adjoint, and has compact resolvent;
\item every eigenfunction of $\mathbf L_m$ belongs to $\mathscr D_m$.
\end{enumerate}
\end{lemma}

\begin{proof}
Fix a real $L^2(\mathbb S^{n-2})$-normalized spherical harmonic $Y_m$ of degree $m$ and define
\begin{align}
J_m:H\longrightarrow L^2(S^{n-1},g;\mathbb R), \qquad J_m u(r,\theta):=u(r)Y_m(\theta). \nonumber
\end{align}
Then $J_m$ is an isometry and, for $u,v\in\mathscr D_m$,
\begin{align}
\langle J_mu,J_mv\rangle_{L^2(S^{n-1},g)} &=\langle u,v\rangle_H, \nonumber\\
\int_{S^{n-1}}\langle\nabla(J_mu),\nabla(J_mv)\rangle_g\,dV_g &=\mathfrak l_m[u,v]. \label{eq:Jm-form-identity}
\end{align}
The pole conditions defining $\mathscr D_m$ imply $J_m\mathscr D_m\subset C^\infty(S^{n-1})$.  Thus
\begin{align}
\|J_mu\|_{H^1(S^{n-1},g)}^2 =\|u\|_H^2+\mathfrak l_m[u,u], \qquad u\in\mathscr D_m. \label{eq:Jm-H1-norm}
\end{align}
If $u_j\to0$ in $H$ and $\mathfrak l_m[u_j-u_k,u_j-u_k]\to0$, then $J_mu_j$ converges in $H^1$ and converges to $0$ in $L^2$; hence its $H^1$ limit is $0$.  This proves closability.  After taking closures, \eqref{eq:Jm-H1-norm} identifies $\operatorname{Dom}(\mathfrak l_m)$ isometrically with a closed subspace of $H^1(S^{n-1},g)$.  Rellich compactness on the closed manifold therefore gives \eqref{eq:compact-form-domain}.

The representation theorem for closed nonnegative forms gives the nonnegative self-adjoint operator $\mathbf L_m$ characterized by \eqref{eq:Lm-domain-def}; see \cite[Proposition~3.1.7, p.~25]{NonnenmacherSpectralTheory}.  The compactness of the form-domain embedding implies that $(I+\mathbf L_m)^{-1}$ is compact.

It remains to prove the last assertion.  Let $\mathbf L_m u=\lambda u$.  For $\Phi\in C^\infty(S^{n-1})$, set
\begin{align}
v_\Phi(r) :=\int_{\mathbb S^{n-2}}\Phi(r,\theta)Y_m(\theta)\,d\theta. \nonumber
\end{align}
Lemma~\ref{lem:rot-angular-coefficients-Dm} gives $v_\Phi\in\mathscr D_m$.  Approximating $u$ by elements of $\mathscr D_m$ in the form norm and using \eqref{eq:Jm-form-identity}, we obtain
\begin{align}
\int_{S^{n-1}}\langle\nabla(J_mu),\nabla\Phi\rangle_g\,dV_g =\mathfrak l_m[u,v_\Phi] =\lambda\langle u,v_\Phi\rangle_H =\lambda\int_{S^{n-1}}J_mu\,\Phi\,dV_g. \nonumber
\end{align}
Thus $J_mu$ is a weak eigenfunction of $-\Delta_g$.  Elliptic regularity gives
\begin{align}
J_mu\in C^\infty(S^{n-1}); \nonumber
\end{align}
see \cite[Theorem~44, p.~68]{CanzaniLaplacian}.  Projecting onto $Y_m$ and applying Lemma~\ref{lem:rot-angular-coefficients-Dm} yields $u\in\mathscr D_m$.
\end{proof}

\begin{lemma}\label{lem:rot-separation}
For every $\Lambda\in\mathbb R$,
\begin{align}
N_{-\Delta_g}^{<}(\Lambda) =\sum_{m=0}^{\infty}d_{n-2,m} N_{\mathbf L_m}^{<}(\Lambda). \nonumber
\end{align}
The same identity holds with $N^{<}$ replaced by the closed counting function $\mathcal N$.
\end{lemma}

\begin{proof}
For each $m$, choose a real orthonormal basis of $\mathcal H_m(\mathbb S^{n-2})$, denoted by
\begin{align}
\{\varphi_{m,\alpha}\}_{\alpha=1}^{d_{n-2,m}}. \nonumber
\end{align}
The spherical-harmonic decomposition
\begin{align}
L^2(\mathbb S^{n-2};\mathbb R) =\widehat{\bigoplus}_{m=0}^{\infty}\mathcal H_m(\mathbb S^{n-2}) \nonumber
\end{align}
induces the unitary decomposition
\begin{align}
L^2(S^{n-1},g;\mathbb R) \cong \widehat{\bigoplus}_{m=0}^{\infty} \widehat{\bigoplus}_{\alpha=1}^{d_{n-2,m}}H, \qquad \{u_{m,\alpha}\} \longmapsto \sum_{m,\alpha}u_{m,\alpha}(r)\varphi_{m,\alpha}(\theta). \nonumber
\end{align}
For a finite family with $u_{m,\alpha}\in\mathscr D_m$, orthogonality and integration by parts on $\mathbb S^{n-2}$ give
\begin{align}
\int_{S^{n-1}}|\nabla w|_g^2\,dV_g =\sum_{m,\alpha}\mathfrak l_m[u_{m,\alpha},u_{m,\alpha}]. \nonumber
\end{align}
Conversely, Lemma~\ref{lem:rot-angular-coefficients-Dm} shows that the angular coefficients of every smooth function lie in the corresponding $\mathscr D_m$.  Finite spherical-harmonic truncations of smooth functions converge in $H^1$, so closure of the preceding identity gives an orthogonal direct-sum decomposition of the Dirichlet form.  Hence
\begin{align}
-\Delta_g \cong \widehat{\bigoplus}_{m=0}^{\infty} \widehat{\bigoplus}_{\alpha=1}^{d_{n-2,m}}\mathbf L_m. \nonumber
\end{align}
The strict and closed counting identities follow, including multiplicities.  The sums are finite at each fixed threshold because
\begin{align}
\mathbf L_m \geq \frac{m(m+n-3)}{\max_{[0,a]}f^2}\,I, \nonumber
\end{align}
so only finitely many angular degrees can contribute below a prescribed number.
\end{proof}

\begin{lemma}
For every $m\geq0$,
\begin{align}
\mathcal B_m\mathscr D_m\subset\mathscr D_{m+1}. \label{eq:Bm-core-mapping-rot}
\end{align}
\end{lemma}

\begin{proof}
Near $r=0$, smoothness of the warping function gives
\begin{align}
f(r)=rF(r^2), \qquad \frac{f'}{f}=\frac1r+rH_0(r^2), \nonumber
\end{align}
with $F$ smooth and positive and $H_0$ smooth.  If $u(r)=r^m\psi(r^2)$, then
\begin{align}
\mathcal B_m u =r^{m+1}\bigl(2\psi'(r^2)-mH_0(r^2)\psi(r^2)\bigr). \nonumber
\end{align}
The analogous calculation at $r=a$ gives the required $(a-r)^{m+1}$ expansion.  Hence \eqref{eq:Bm-core-mapping-rot} holds.
\end{proof}

\begin{lemma}\label{lem:rot-adjoint-graph-core}
For every $m\geq0$,
\begin{align}
\operatorname{Dom}(\mathbf B_m^*) =\overline{\mathscr D_{m+1}}^{\,\|\cdot\|_{\mathbf B_m^*}}, \qquad \|v\|_{\mathbf B_m^*}^2 :=\|v\|_H^2+\|\mathbf B_m^*v\|_H^2. \nonumber
\end{align}
\end{lemma}

\begin{proof}
Let $v\in\operatorname{Dom}(\mathbf B_m^*)$ and set $w=\mathbf B_m^*v$.  On compact subintervals of $(0,a)$,
\begin{align}
-v'-(m+n-2)\frac{f'}{f}v=w \nonumber
\end{align}
holds distributionally, so $v\in H^1_{\mathrm{loc}}(0,a)$ and
\begin{align}
\bigl(f^{m+n-2}v\bigr)'=-f^{m+n-2}w. \label{eq:rot-adjoint-integrating-factor}
\end{align}
The integration constant at either endpoint is zero: otherwise $v\sim cf^{-(m+n-2)}$, which is not in $L^2((0,a),f^{n-2}dr)$ for $n\geq3$.

Near $r=0$, \eqref{eq:rot-adjoint-integrating-factor} and Cauchy--Schwarz imply
\begin{align}
\int_{\varepsilon}^{2\varepsilon}|v|^2f^{n-2}\,dr=o(\varepsilon^2) \qquad\text{as }\varepsilon\downarrow0. \label{eq:rot-adjoint-cutoff-estimate}
\end{align}
The same estimate holds near $r=a$.  Choose a cutoff $\chi_\varepsilon$ which vanishes in the two $\varepsilon$-neighborhoods of the endpoints, equals one outside the two $2\varepsilon$-neighborhoods, and satisfies $|\chi_\varepsilon'|\leq C\varepsilon^{-1}$.  Then $\chi_\varepsilon v\to v$ in $H$ and, by \eqref{eq:rot-adjoint-cutoff-estimate},
\begin{align}
\mathcal B_m^\dagger(\chi_\varepsilon v) =\chi_\varepsilon w-\chi_\varepsilon'v \longrightarrow w \qquad\text{in }H. \nonumber
\end{align}
After cutting off, ordinary mollification on compact subintervals produces a sequence in $C_c^\infty(0,a)\subset\mathscr D_{m+1}$ converging to $v$ in the graph norm of $\mathbf B_m^*$.
\end{proof}

\begin{lemma}\label{lem:Bm-qm-form-domain-rot}
For every $m\geq0$,
\begin{align}
\operatorname{Dom}(\mathbf B_m) =\operatorname{Dom}(\mathfrak l_m). \label{eq:Bm-qm-domain-rot}
\end{align}
Moreover, the graph norm of $\mathcal B_m$ and the $\mathfrak l_m$-form norm are equivalent on $\mathscr D_m$.
\end{lemma}

\begin{proof}
For $u\in\mathscr D_m$, expanding the square and integrating the mixed term by parts gives
\begin{align}
\|\mathcal B_m u\|_H^2-\mathfrak l_m[u,u]=m\int_0^a u^2f^{n-2}\left[\frac{(m+n-3)((f')^2-1)+ff''}{f^2}\right]dr. \label{eq:Bm-qm-norm-difference}
\end{align}
The pole compatibility conditions make the coefficient in brackets smooth at both endpoints: its numerator is $O(f^2)$.  Hence the right-hand side of \eqref{eq:Bm-qm-norm-difference} is bounded in absolute value by a constant times $\|u\|_H^2$.  Therefore
\begin{align}
\|u\|_H^2+\|\mathcal B_m u\|_H^2 \asymp \|u\|_H^2+\mathfrak l_m[u,u] \qquad\text{on }\mathscr D_m. \nonumber
\end{align}
The two completions of $\mathscr D_m$ inside $H$ consequently coincide, proving \eqref{eq:Bm-qm-domain-rot}.
\end{proof}

\begin{lemma}[Radial ladder identity]
For every $m\geq0$,
\begin{align}
\mathbf L_0 &=\mathbf B_0^*\mathbf B_0 \qquad\text{on }\mathscr D_0, \label{eq:rot-L0-factorization}\\
\mathbf B_m\mathbf B_m^*- \mathbf B_{m+1}^*\mathbf B_{m+1} &=(2m+n-1)K_{\mathrm{rad}} \qquad\text{on }\mathscr D_{m+1}. \label{eq:ladder-comm-rot}
\end{align}
\end{lemma}

\begin{proof}
Using \eqref{eq:Bm-rot} and \eqref{eq:Bm-adj-rot}, one obtains on the indicated cores
\begin{align}
\mathbf B_m^*\mathbf B_m &=-\frac{d^2}{dr^2}-(n-2)\frac{f'}{f}\frac{d}{dr} +m(m+n-3)\frac{(f')^2}{f^2}+m\frac{f''}{f}, \nonumber\\
\mathbf B_m\mathbf B_m^* &=-\frac{d^2}{dr^2}-(n-2)\frac{f'}{f}\frac{d}{dr} -(m+n-2)\frac{f''}{f} +(m+1)(m+n-2)\frac{(f')^2}{f^2}. \nonumber
\end{align}
The first formula with $m=0$ gives \eqref{eq:rot-L0-factorization}; subtraction gives \eqref{eq:ladder-comm-rot}.
\end{proof}

\section{Radial curvature and application of the abstract recursion}\label{sec:rotational-vanishing}

\begin{lemma}\label{lem:rot-curvature-ineq}
Assume $\Rc_g\geq(n-2)g$.  Then, on $(0,a)$,
\begin{align}
f''+f&\leq0, \label{eq:fconcave}\\
f^2+(f')^2&\leq1. \label{eq:one-minus-fprime}
\end{align}
Consequently,
\begin{align}
K_{\mathrm{rad}}\geq1, \qquad K_{\mathrm{tan}}\geq1. \nonumber
\end{align}
Moreover, $K_{\mathrm{rad}}\equiv1$ if and only if $(S^{n-1},g)$ is the unit round sphere.
\end{lemma}

\begin{proof}
For a unit vector $X$ tangent to a principal orbit,
\begin{align}
\Rc_g(\partial_r,\partial_r) &=(n-2)K_{\mathrm{rad}}, \nonumber\\
\Rc_g(X,X) &=K_{\mathrm{rad}}+(n-3)K_{\mathrm{tan}}. \nonumber
\end{align}
The radial component of the Ricci lower bound gives $K_{\mathrm{rad}}\geq1$, equivalently \eqref{eq:fconcave}.

Since $f''\leq-f<0$, $f'$ decreases from $1$ to $-1$ and vanishes at a unique point.  For
\begin{align}
E(r):=f(r)^2+f'(r)^2, \nonumber
\end{align}
one has
\begin{align}
E'(r)=2f'(r)(f(r)+f''(r)). \nonumber
\end{align}
Thus $E$ decreases before the zero of $f'$ and increases after it.  Since $E(0)=E(a)=1$, one obtains $E\leq1$, which is \eqref{eq:one-minus-fprime} and gives $K_{\mathrm{tan}}\geq1$.

If $K_{\mathrm{rad}}\equiv1$, then $f''+f=0$, $f(0)=0$, and $f'(0)=1$, so $f(r)=\sin r$.  The conditions at the second pole force $a=\pi$.  The converse is immediate.
\end{proof}

\begin{lemma}\label{lem:rot-Am-Cm-compact-regularity}
For every $m\geq0$,
\begin{align}
\operatorname{Dom}(\mathfrak a_m) &=\operatorname{Dom}(\mathfrak l_m), \label{eq:rot-am-domain-qm}\\
\operatorname{Dom}(\mathfrak c_m) &=\operatorname{Dom}(\mathfrak a_{m+1}) =\operatorname{Dom}(\mathfrak l_{m+1}). \label{eq:rot-cm-domain-qmplus}
\end{align}
On this common domain,
\begin{align}
\mathfrak c_m[v,v]-\mathfrak a_{m+1}[v,v] =(2m+n-1)\int_0^a(K_{\mathrm{rad}}-1)v^2f^{n-2}\,dr. \label{eq:rot-cm-amplus-form-difference}
\end{align}
The common form domains in \eqref{eq:rot-am-domain-qm} and \eqref{eq:rot-cm-domain-qmplus} embed compactly into $H$ when equipped with any of the indicated form norms.  Consequently, the operators $\mathbf A_m$ and $\mathbf C_m$ have compact resolvent.  In addition, every eigenfunction of $\mathbf C_m$ belongs to $\mathscr D_{m+1}$.
\end{lemma}

\begin{proof}
The first domain identity is Lemma~\ref{lem:Bm-qm-form-domain-rot}.  On $\mathscr D_{m+1}$, the radial ladder identity \eqref{eq:ladder-comm-rot}, together with
\begin{align}
(m+1)(m+n-1)-m(m+n-2)=2m+n-1, \nonumber
\end{align}
gives \eqref{eq:rot-cm-amplus-form-difference}.  Since $K_{\mathrm{rad}}$ is bounded on $[0,a]$, the $\mathfrak c_m$- and $\mathfrak a_{m+1}$-form norms are equivalent on $\mathscr D_{m+1}$.  Lemma~\ref{lem:rot-adjoint-graph-core} identifies the first completion with $\operatorname{Dom}(\mathfrak c_m)$, while the definition of $\mathbf B_{m+1}$ identifies the second with $\operatorname{Dom}(\mathfrak a_{m+1})$.  This proves \eqref{eq:rot-cm-domain-qmplus} and extends \eqref{eq:rot-cm-amplus-form-difference} to the closed domains.

By Lemma~\ref{lem:Bm-qm-form-domain-rot} and Definition~\ref{def:rot-bold-Am-operator}, the $\mathfrak a_m$-form norm is equivalent to the $\mathfrak l_m$-form norm on their common domain.  The bounded-potential identity \eqref{eq:rot-cm-amplus-form-difference} likewise makes the $\mathfrak c_m$- and $\mathfrak a_{m+1}$-form norms equivalent.  Therefore \eqref{eq:compact-form-domain}, applied at levels $m$ and $m+1$, gives the asserted compact embeddings.  The standard closed-form resolvent argument then gives compact resolvent for both $\mathbf A_m$ and $\mathbf C_m$.

For the regularity assertion, polarizing the calculation in Lemma~\ref{lem:Bm-qm-form-domain-rot} gives
\begin{align}
\mathfrak l_{m+1}[v,\phi]-\mathfrak c_m[v,\phi] =(m+n-2)\int_0^a \bigl((m+1)(K_{\mathrm{tan}}-1)-(K_{\mathrm{rad}}-1)\bigr) v\phi f^{n-2}\,dr. \label{eq:rot-qmplus-cm-potential}
\end{align}
The coefficient is smooth and bounded on $[0,a]$.  If $\mathbf C_m v=\gamma v$, then \eqref{eq:rot-qmplus-cm-potential} shows that $v\in\operatorname{Dom}(\mathbf L_{m+1})$ and that $J_{m+1}v$ satisfies a second-order elliptic equation with smooth coefficients on $S^{n-1}$.  Elliptic regularity and Lemma~\ref{lem:rot-angular-coefficients-Dm} give $v\in\mathscr D_{m+1}$.
\end{proof}

\begin{lemma}[Comparison with the separated Laplacian]
Assume $\Rc_g\geq(n-2)g$.  Then
\begin{align}
\mathbf L_m\geq\mathbf A_m \nonumber
\end{align}
in the quadratic-form sense.  Consequently, for every $j\geq0$ and every $\Lambda\in\mathbb R$,
\begin{align}
\lambda_j(\mathbf L_m) &\geq\lambda_j(\mathbf A_m), \label{eq:rot-LA-eigenvalue-comparison}\\
N_{\mathbf L_m}^{<}(\Lambda) &\leq N_{\mathbf A_m}^{<}(\Lambda). \label{eq:rot-LA-counting}
\end{align}
\end{lemma}

\begin{proof}
By Lemma~\ref{lem:Bm-qm-form-domain-rot}, equivalently by \eqref{eq:rot-am-domain-qm}, the closed forms $\mathfrak l_m$ and $\mathfrak a_m$ have the same domain.  For $u\in\mathscr D_m$, direct calculation gives
\begin{align}
\mathfrak l_m[u,u]-\mathfrak a_m[u,u] =m\int_0^a \bigl((m+n-3)(K_{\mathrm{tan}}-1)+(K_{\mathrm{rad}}-1)\bigr) u^2f^{n-2}\,dr. \label{eq:rot-qm-am-potential}
\end{align}
The coefficient is bounded, so this identity extends to every
$u\in\operatorname{Dom}(\mathfrak l_m)=\operatorname{Dom}(\mathfrak a_m)$; it also shows that the two form norms on the common domain are equivalent.  Lemma~\ref{lem:rot-curvature-ineq} makes the right-hand side nonnegative.  The min--max principle for the self-adjoint operators associated with $\mathfrak l_m$ and $\mathfrak a_m$, namely $\mathbf L_m$ and $\mathbf A_m$, gives \eqref{eq:rot-LA-eigenvalue-comparison}, and the counting inequality is equivalent to it.
\end{proof}

\begin{lemma}\label{lem:rot-kernel-dimension}
Assume $\Rc_g\geq(n-2)g$.  Then, for every $m\geq0$,
\begin{align}
\ker\mathbf B_m=\mathbb R f^m, \qquad \dim_{\mathbb R}\ker\mathbf B_m=1, \qquad \ker\mathbf B_m^*=\{0\}. \nonumber
\end{align}
In particular,
\begin{align}
\lambda_0(\mathbf A_m)=m(m+n-2), \label{eq:nubottom}
\end{align}
and this eigenvalue is simple.
\end{lemma}

\begin{proof}
If $u\in\operatorname{Dom}(\mathbf B_m)$ and $\mathbf B_m u=0$, then on $(0,a)$
\begin{align}
u'=m\frac{f'}{f}u \nonumber
\end{align}
in the distributional, hence classical, sense.  Thus $(f^{-m}u)'=0$, so $u=cf^m$.  Conversely, the pole expansions of $f$ show that $f^m\in\mathscr D_m$, and \eqref{eq:Bm-rot} gives $\mathbf B_m f^m=0$.

It remains to prove the adjoint-kernel assertion.  Apply Lemma~\ref{lem:abstract-adjoint-kernel-gap} to the following concrete realization of Notation~\ref{not:abstract-ladder-package}:
\begin{align}
&H_m=H_{m+1}=H, \qquad B_m=\mathbf B_m, \nonumber\\
&\alpha_m=m(m+n-2), \qquad \alpha_{m+1}=(m+1)(m+n-1). \nonumber
\end{align}
The required domain inclusion and form inequality \eqref{eq:abstract-ladder-form-order} are precisely \eqref{eq:rot-cm-domain-qmplus} and \eqref{eq:rot-cm-amplus-form-difference}, together with $K_{\mathrm{rad}}\geq1$ from Lemma~\ref{lem:rot-curvature-ineq}.  Since $\alpha_{m+1}-\alpha_m=2m+n-1>0$, Lemma~\ref{lem:abstract-adjoint-kernel-gap} yields $\ker\mathbf B_m^*=\{0\}$.  Formula \eqref{eq:nubottom} follows from the definition of $\mathbf A_m$.
\end{proof}

The preceding lemmas verify all hypotheses needed to use the separated radial result from Section~\ref{sec:abstract-ladder-counting}.  Explicitly, for every ladder level $m$ we take
\begin{align}
\mathbb F=\mathbb R, \qquad H_m=H, \qquad B_m=\mathbf B_m, \qquad \alpha_m=m(m+n-2). \label{eq:rot-abstract-dictionary}
\end{align}
The operators are closed and densely defined by Definition~\ref{def:Bm-for-rotation-symm}; the partner operators have compact resolvent by Lemma~\ref{lem:rot-Am-Cm-compact-regularity}; their kernel dimensions are $r_m=1$ by Lemma~\ref{lem:rot-kernel-dimension}; and the consecutive form order \eqref{eq:abstract-ladder-form-order} follows from \eqref{eq:rot-cm-domain-qmplus}, \eqref{eq:rot-cm-amplus-form-difference}, and Lemma~\ref{lem:rot-curvature-ineq}.  Thus the partner-spectrum algebra and the counting iteration need not be reproved in the radial setting.

\begin{prop}\label{prop:rot-counting-recursion-1}
Assume $\Rc_g\geq(n-2)g$.  For every $l\in\mathbb Z^+$,
\begin{align}
&N_{\mathbf L_m}^{<}\bigl(l(l+n-2)\bigr) \leq l-m, \qquad 0\leq m<l, \nonumber\\
&N_{\mathbf L_m}^{<}\bigl(l(l+n-2)\bigr)=0, \qquad m\geq l. \nonumber
\end{align}
\end{prop}

\begin{proof}
Apply Proposition~\ref{prop:abstract-ladder-iteration}  to the Hilbert spaces, operators, and shifts in \eqref{eq:rot-abstract-dictionary}, note $r_m=1$.  The verification of its hypotheses was given immediately above; we get 
\begin{align}
N_{\mathbf A_m}^{<}\bigl(l(l+n-2)\bigr) \leq l-m \qquad(0\leq m<l), \nonumber
\end{align}
and gives zero for $m\geq l$.  The comparison \eqref{eq:rot-LA-counting} transfers these bounds from $\mathbf A_m$ to $\mathbf L_m$.
\end{proof}

\section{The rotational comparison theorem and rigidity}\label{sec:rotational-comparison}

\begin{lemma}\label{lem:rot-branching-count}
For every $l\in\mathbb Z^+$,
\begin{align}
\sum_{m=0}^{l-1}d_{n-2,m}(l-m) =N_{-\Delta_{\mathbb S^{n-1}}}^{<}\bigl(l(l+n-2)\bigr). \nonumber
\end{align}
\end{lemma}

\begin{proof}
The standard dimension formula is
\begin{align}
d_{n-2,m} =\binom{n+m-2}{n-2}-\binom{n+m-4}{n-2}, \label{eq:rot-harmonic-dim-formula}
\end{align}
with the convention that the binomial coefficient vanishes when the upper index is smaller than $n-2$.  The branching identity
\begin{align}
\sum_{s=0}^{m}d_{n-2,s}=d_{n-1,m} \nonumber
\end{align}
follows by summing \eqref{eq:rot-harmonic-dim-formula}.  Therefore
\begin{align}
\sum_{m=0}^{l-1}d_{n-2,m}(l-m) &=\sum_{k=0}^{l-1}\sum_{m=0}^{k}d_{n-2,m} =\sum_{k=0}^{l-1}d_{n-1,k} =N_{-\Delta_{\mathbb S^{n-1}}}^{<}\bigl(l(l+n-2)\bigr). \nonumber
\end{align}
\end{proof}

\begin{prop}\label{prop:main-corollary-rotation-symm}
Let $n\geq3$, and let $(S^{n-1},g)$ be rotationally symmetric with $\Rc_g\geq(n-2)g$.  Then, for every $l\in\mathbb Z^+$,
\begin{align}
N_{-\Delta_g}^{<}\bigl(l(l+n-2)\bigr) \leq N_{-\Delta_{\mathbb S^{n-1}}}^{<}\bigl(l(l+n-2)\bigr). \nonumber
\end{align}
\end{prop}

\begin{proof}
By Lemma~\ref{lem:rot-separation} and Proposition~\ref{prop:rot-counting-recursion-1},
\begin{align}
N_{-\Delta_g}^{<}\bigl(l(l+n-2)\bigr) &=\sum_{m=0}^{\infty}d_{n-2,m} N_{\mathbf L_m}^{<}\bigl(l(l+n-2)\bigr) \leq\sum_{m=0}^{l-1}d_{n-2,m}(l-m). \nonumber
\end{align}
Lemma~\ref{lem:rot-branching-count} identifies the last expression with the round strict counting function.
\end{proof}

\begin{lemma}\label{lem:strict-ladder-eigenvalues-rot}
Let $\mathfrak s$ and $\mathfrak t$ be lower-bounded closed forms in the radial Hilbert space $H$ defined in (\ref{eq:H-inner}), with common form domain $\mathcal V :=\operatorname{Dom}(\mathfrak s) =\operatorname{Dom}(\mathfrak t). $

Let $S$ and $T$ be the self-adjoint operators associated with $\mathfrak s$ and $\mathfrak t$, respectively.  Suppose that
\begin{align}
\mathfrak s[v,v] =\mathfrak t[v,v]+\int_0^a P|v|^2f^{n- 2}, \qquad v\in\mathcal V, \nonumber
\end{align}
for a bounded function $P\geq0$.   Assume that the inclusion $\mathcal V\hookrightarrow\mathcal H$ is compact and that $P$ is positive on a nonempty open set $U\subseteq (0, a)$.  Assume also that, for every $j\geq0$, no nonzero element of the span of the first $j+1$ eigenfunctions of $S$ vanishes identically on $U$.  Then
\begin{align}
\lambda_j(S)>\lambda_j(T), \qquad j=0,1,2,\ldots . \label{eq:strict-nu-ladder-rot}
\end{align}
\end{lemma}

\begin{proof}
Let $F_j$ be the span of the first $j+1$ eigenfunctions of $S$, counted with multiplicity.  For every $0\neq v\in F_j$,
\begin{align}
\frac{\mathfrak s[v,v]}{\|v\|^2} \leq\lambda_j(S) ,\nonumber
\end{align}
and the spectral-subspace hypothesis gives$\int_0^a P|v|^2f^{n- 2}>0.$ 

Hence
\begin{align}
\frac{\mathfrak t[v,v]}{\|v\|^2} <\lambda_j(S) \qquad(0\neq v\in F_j). \nonumber
\end{align}
The unit sphere of $F_j$ is compact, so the supremum of the left-hand side over $F_j\setminus\{0\}$ is still strictly smaller than $\lambda_j(S)$.  The min--max principle for $T$ gives \eqref{eq:strict-nu-ladder-rot}.
\end{proof}

In the radial applications below, the spectral-subspace nonvanishing hypothesis follows from Lemmas~\ref{lem:Hilbert-space-Lm} and \ref{lem:rot-Am-Cm-compact-regularity}.  Indeed, after decomposing a finite spectral sum into its distinct eigenvalue components by polynomial spectral projections, each component satisfies a second-order ordinary differential equation with smooth coefficients on $(0,a)$; vanishing on an open interval therefore forces that component to vanish identically.  The same argument, with elliptic unique continuation, applies to the smooth bundle operators used later in the paper.

\begin{lemma}\label{lem:rot-strict-separated-spectrum}
Assume that $(S^{n-1},g)$ is rotationally symmetric, $\Rc_g\geq(n-2)g$, and $g$ is not round.  Then
\begin{align}
\lambda_j(\mathbf L_m) >(m+j)(m+j+n-2) \qquad\text{whenever }(m,j)\neq(0,0). \label{eq:rot-strict-separated-spectrum}
\end{align}
\end{lemma}

\begin{proof}
\textbf{Step (1)}. By Lemma~\ref{lem:rot-curvature-ineq}, nonroundness implies that there is a nonempty open interval $U\subset(0,a)$ on which
\begin{align}
K_{\mathrm{rad}}>1. \nonumber
\end{align}

We first compare the two consecutive partner operators.  Apply Lemma~\ref{lem:strict-ladder-eigenvalues-rot} in the Hilbert space $H$ with
\begin{align}
\mathfrak s=\mathfrak c_m, \qquad \mathfrak t=\mathfrak a_{m+1}, \qquad S=\mathbf C_m, \qquad T=\mathbf A_{m+1}. \nonumber
\end{align}
The operators $\mathbf C_m$ and $\mathbf A_{m+1}$ are, by Definition~\ref{def:rot-bold-Am-operator}, the self-adjoint operators associated with $\mathfrak c_m$ and $\mathfrak a_{m+1}$, respectively.  By \eqref{eq:rot-cm-domain-qmplus}, their forms have the common domain
\begin{align}
\operatorname{Dom}(\mathfrak c_m) =\operatorname{Dom}(\mathfrak a_{m+1}) =\operatorname{Dom}(\mathfrak l_{m+1}). \nonumber
\end{align}

The identity \eqref{eq:rot-cm-amplus-form-difference} holds on the common closed form domain and has the bounded potential
\begin{align}
P= (2m+n-1)(K_{\mathrm{rad}}-1).\nonumber
\end{align}
Hence \eqref{eq:compact-form-domain}, applied at level $m+1$, shows that the common form domain embeds compactly into $H$.  The displayed potential $P$ is nonnegative everywhere and positive on $U$.  

Finally, the regularity assertion in Lemma~\ref{lem:rot-Am-Cm-compact-regularity}, followed by the ordinary-differential-equation argument stated after Lemma~\ref{lem:strict-ladder-eigenvalues-rot}, verifies that no nonzero element of a finite $\mathbf C_m$-spectral subspace can vanish identically on $U$.  

Thus all hypotheses of Lemma~\ref{lem:strict-ladder-eigenvalues-rot} are satisfied, and
\begin{align}
\lambda_j(\mathbf C_m) >\lambda_j(\mathbf A_{m+1}), \qquad m,j\geq0. \label{eq:strict-nu-ladder-rot-1}
\end{align}

We next use the exact partner shift from Section~\ref{sec:abstract-ladder-counting}.  Apply Lemma~\ref{lem:abstract-partner-spectra} with $H_m=H$, $B_m=\mathbf B_m$, and $\alpha_m=m(m+n-2)$.  Lemma~\ref{lem:rot-kernel-dimension} gives $r_m=1$ and $s_m=0$, so \eqref{eq:abstract-partner-index-shift} becomes
\begin{align}
\lambda_j(\mathbf A_m) =\lambda_{j-1}(\mathbf C_m), \qquad j\geq1. \nonumber
\end{align}
Combining this with \eqref{eq:strict-nu-ladder-rot-1} and iterating yields
\begin{align}
\lambda_j(\mathbf A_m) &>\lambda_{j-1}(\mathbf A_{m+1}) >\cdots>\lambda_0(\mathbf A_{m+j}) \nonumber\\
&=(m+j)(m+j+n-2), \qquad j\geq1, \label{eq:strict-nu-rot}
\end{align}
where the last equality is \eqref{eq:nubottom}.

\textbf{Step (2)}. For $m=0$, \eqref{eq:rot-am-domain-qm} and \eqref{eq:rot-qm-am-potential} show that the closed forms $\mathfrak l_0$ and $\mathfrak a_0$ coincide.  Since $\mathbf L_0$ and $\mathbf A_0$ are their associated self-adjoint operators, respectively, one has $\mathbf L_0=\mathbf A_0$, and \eqref{eq:strict-nu-rot} proves the assertion for $j\geq1$.

\textbf{Step (3)}. Assume now that $m\geq1$.  We again apply Lemma~\ref{lem:strict-ladder-eigenvalues-rot} in $H$, this time with
\begin{align}
\mathfrak s=\mathfrak l_m, \qquad \mathfrak t=\mathfrak a_m, \qquad S=\mathbf L_m, \qquad T=\mathbf A_m. \nonumber
\end{align}
By Definition~\ref{def:Lm-operator}, $\mathbf L_m$ is the self-adjoint operator associated with $\mathfrak l_m$, while Definition~\ref{def:rot-bold-Am-operator} identifies $\mathbf A_m$ as the self-adjoint operator associated with $\mathfrak a_m$.  The form-domain identity \eqref{eq:rot-am-domain-qm} gives
\begin{align}
\operatorname{Dom}(\mathfrak l_m) =\operatorname{Dom}(\mathfrak a_m), \nonumber
\end{align}
On this common closed form domain, \eqref{eq:rot-qm-am-potential} reads
\begin{align}
\mathfrak l_m[u,u] &=\mathfrak a_m[u,u] \nonumber\\
&\quad+m\int_0^a \bigl((m+n-3)(K_{\mathrm{tan}}-1)+(K_{\mathrm{rad}}-1)\bigr) u^2f^{n-2}\,dr. \nonumber
\end{align}
The coefficient is bounded, so the $\mathfrak l_m$- and $\mathfrak a_m$-form norms are equivalent.  

Consequently, \eqref{eq:compact-form-domain}, applied at level $m$, shows that their common form domain embeds compactly into $H$.  Lemma~\ref{lem:rot-curvature-ineq} makes the potential nonnegative, and, because $m\geq1$, it is strictly positive on $U$.  

By Lemma~\ref{lem:Hilbert-space-Lm} and the ordinary-differential-equation argument stated after Lemma~\ref{lem:strict-ladder-eigenvalues-rot}, no nonzero element of a finite $\mathbf L_m$-spectral subspace vanishes identically on $U$.  

Hence Lemma~\ref{lem:strict-ladder-eigenvalues-rot} gives
\begin{align}
\lambda_j(\mathbf L_m)>\lambda_j(\mathbf A_m), \qquad j\geq0. \nonumber
\end{align}
Together with \eqref{eq:nubottom} and \eqref{eq:strict-nu-rot}, this proves \eqref{eq:rot-strict-separated-spectrum} in all cases except $(m,j)=(0,0)$, where both sides are zero.
\end{proof}

\begin{theorem}
Let $n\geq3$, and let $(S^{n-1},g)$ be a smooth rotationally symmetric Riemannian manifold with $\Rc_g\geq(n-2)g$.  Then
\begin{align}
\lambda_i(S^{n-1},g) \geq \lambda_i(\mathbb S^{n-1},g_{\mathrm{round}}), \qquad i\geq1. \label{eq-main-comparison-1}
\end{align}
Consequently,
\begin{align}
\mathcal N_{-\Delta_g}\bigl(l(l+n-2)\bigr) \leq \mathcal N_{-\Delta_{\mathbb S^{n-1}}}\bigl(l(l+n-2)\bigr), \qquad l\in\mathbb Z_{\geq0}. \label{eq:rot-nonstrict-counting-general}
\end{align}
Moreover, if equality holds in \eqref{eq-main-comparison-1} for some $i\geq1$, or if equality holds in \eqref{eq:rot-nonstrict-counting-general} for some $l\in\mathbb Z^+$, then $(S^{n-1},g)$ is isometric to the unit round sphere.
\end{theorem}

\begin{proof}
\textbf{Step (1)}. Fix $i\geq1$, and choose the unique $l\geq1$ such that the $i$-th round eigenvalue is $l(l+n-2)$.  Equivalently,
\begin{align}
N_{-\Delta_{\mathbb S^{n-1}}}^{<}\bigl(l(l+n-2)\bigr) \leq i < N_{-\Delta_{\mathbb S^{n-1}}}^{<}\bigl((l+1)(l+n-1)\bigr). \nonumber
\end{align}
Proposition~\ref{prop:main-corollary-rotation-symm} gives
\begin{align}
\lambda_{N_{-\Delta_{\mathbb S^{n-1}}}^{<}(l(l+n-2))}(S^{n-1},g) \geq l(l+n-2). \nonumber
\end{align}
Since the eigenvalues are nondecreasing and $i$ is not smaller than this index,
\begin{align}
\lambda_i(S^{n-1},g) \geq l(l+n-2) =\lambda_i(\mathbb S^{n-1},g_{\mathrm{round}}). \nonumber
\end{align}
This proves \eqref{eq-main-comparison-1}.  The closed-counting comparison \eqref{eq:rot-nonstrict-counting-general} follows immediately; for $l=0$, both connected manifolds have exactly one eigenvalue at zero.

\textbf{Step (2)}. For rigidity, suppose that $g$ is not round.  If
\begin{align}
\lambda_j(\mathbf L_m) \leq l(l+n-2), \nonumber
\end{align}
then Lemma~\ref{lem:rot-strict-separated-spectrum} implies $m+j\leq l-1$.  Hence Lemmas~\ref{lem:rot-separation} and \ref{lem:rot-branching-count} give, for every $l\geq1$,
\begin{align}
\mathcal N_{-\Delta_g}\bigl(l(l+n-2)\bigr) &=\sum_{m=0}^{\infty}d_{n-2,m} \mathcal N_{\mathbf L_m}\bigl(l(l+n-2)\bigr) \leq\sum_{m=0}^{l-1}d_{n-2,m}(l-m) \nonumber \\
&=N_{-\Delta_{\mathbb S^{n-1}}}^{<}\bigl(l(l+n-2)\bigr)\label{eq:rot-strict-closed-counting-nonround} \\
&<\mathcal N_{-\Delta_{\mathbb S^{n-1}}}\bigl(l(l+n-2)\bigr). \nonumber
\end{align}
Thus equality in \eqref{eq:rot-nonstrict-counting-general} is impossible for a nonround metric.

\textbf{Step (3)}. The same estimate also makes every positive ordered comparison strict.  Indeed, if the $i$-th round eigenvalue equals $l(l+n-2)$, then
\begin{align}
i\geq N_{-\Delta_{\mathbb S^{n-1}}}^{<}\bigl(l(l+n-2)\bigr), \nonumber
\end{align}
whereas \eqref{eq:rot-strict-closed-counting-nonround} shows that at most that many eigenvalues of $-\Delta_g$ are less than or equal to $l(l+n-2)$.  Therefore
\begin{align}
\lambda_i(S^{n-1},g)>l(l+n-2) =\lambda_i(\mathbb S^{n-1},g_{\mathrm{round}}). \nonumber
\end{align}
Equality in \eqref{eq-main-comparison-1} is likewise impossible unless $g$ is round.
\end{proof}

\part{Smooth Riemannian \(2\)-Spheres with \(K\ge 1\)}

Throughout this part, \(m\in\mathbb Z_{\geq0}\).  The basic objects are the
line bundles \(E_m=(T^{1,0}S^2)^{\otimes m}\), the Hilbert spaces
\(H_m=L^2(E_m)\), and the closed first-order operators
\(B_m:H_m\supset\operatorname{Dom}(B_m)\to H_{m+1}\).  They are constructed
in Section~\ref{sec:riemannian-ladder}.  The shifted forms
\(\mathfrak a_m,\mathfrak c_m\) and their operators \(A_m,C_m\) are fixed at
the beginning of Section~\ref{sec:comparison-rigidity}.

In this part, we always assume that $(S^2, g)$ is a smooth Riemannian sphere satisfying $K_g\geq 1$ unless otherwise mentioned.

\section{Riemannian ladder identity}\label{sec:riemannian-ladder}

The purpose of this section is to remove the artificial symmetry from Section~\ref{sec:rotational} for $S^2$.  In the rotational case, $m$ measured angular frequency.  For a general metric there is no distinguished angular coordinate, but there is still a canonical rotation group acting on oriented orthonormal frames. 

Fix the standard orientation on \(S^2\). The metric and the orientation define the bundle map $J:TS^2\longrightarrow TS^2$ by the rule that, for every local oriented orthonormal frame \(e_1,e_2\) with respect to $g$,
\begin{align}
Je_1=e_2,\qquad Je_2=-e_1. \nonumber
\end{align}
Thus \(J^2=-\Id\). Let
\begin{align}
T_{\C}S^2:=TS^2\otimes_{\R}\C. \nonumber
\end{align}

Since \(S^2\) is oriented, let \(dA_g\) denote the Riemannian area form of \(g\). 

\begin{definition}\label{def:Em bundle}
{Extend \(\C\)-linearly $J:TS^2\longrightarrow TS^2$ to \(T_{\C}S^2\), we define the \(+i\)-eigenbundle of \(J\) by
\begin{align}
T^{1,0}S^2:=\ker(J-i\Id)=\{W\in T_{\C}S^2:JW=iW\}. \nonumber
\end{align}
For each integer \(m\geq0\), define
\begin{align}
E_m:=(T^{1,0}S^2)^{\otimes m},\qquad E_0:=S^2\times\C. \nonumber
\end{align}
}
\end{definition}

The replacement of the $m$-th Fourier mode is realized geometrically as the line bundle $E_m=(T^{1,0}S^2)^{\otimes m}$. This is the first conceptual upgrade in the paper. 

In a local oriented orthonormal frame, set
\begin{align}
Z:=\frac{1}{\sqrt2}(e_1-i e_2). \nonumber
\end{align}
If another local oriented orthonormal frame is written as
\begin{align}
e_1'=\cos\alpha\,e_1+\sin\alpha\,e_2,\qquad e_2'=-\sin\alpha\,e_1+\cos\alpha\,e_2, \nonumber
\end{align}
then
\begin{align}
Z'=\frac{1}{\sqrt2}(e_1'-ie_2')=e^{i\alpha}Z. \label{eq:Z-frame-change-new}
\end{align}

In an oriented orthonormal frame, set
\begin{align}
Z^{\otimes 0}:=1,\qquad Z^{\otimes m}:=Z\otimes\cdots\otimes Z\quad \text{with }m\text{ tensor factors, for }m\geq1. \nonumber
\end{align}
Under the frame change \eqref{eq:Z-frame-change-new}, the local tensors satisfy
\begin{align}
(Z')^{\otimes m}=e^{im\alpha}Z^{\otimes m}. \nonumber
\end{align}
Thus a local section \(s\) of \(E_m\) is written as $s=u\cdot Z^{\otimes m}$, where \(u\) is a complex-valued function, and on an overlap the coefficient changes by
\begin{align}
u'=e^{-im\alpha}u. \nonumber
\end{align}
This gives a globally defined rank-one complex vector bundle.

The Hermitian metric $h_m$ on \(E_m\) is defined by declaring each frame \(Z^{\otimes m}\) to be unit:
\begin{align}
|u\cdot Z^{\otimes m}|^2=|u|^2. \nonumber
\end{align}
Since \(|e^{im\alpha}|=1\), this definition is independent of the chosen oriented orthonormal frame.

\begin{notation}
Let \(\pi_m:E_m\to S^2\) denote the bundle projection.  We use the following spaces of sections:
\begin{align}
C^\infty(E_m)&:=\{s:S^2\to E_m:\ s \text{ is smooth and } \pi_m\circ s=\operatorname{id}_{S^2}\}, \nonumber\\
\Gamma_{\mathrm{meas}}(E_m)&:=\{s:S^2\to E_m:\ s \text{ is Borel measurable and } \pi_m\circ s=\operatorname{id}_{S^2}\ \text{a.e.}\}/\sim, \nonumber
\end{align}
where \(\sim\) means equality \(dA_g\)-almost everywhere.  The Hilbert space \(L^2(E_m)\) is
\begin{align}
L^2(E_m):=\left\{s\in\Gamma_{\mathrm{meas}}(E_m): \int_{S^2}|s|_{h_m}^2\,dA_g<\infty\right\}. \nonumber
\end{align}
Its Hermitian inner product is
\begin{align}
\langle s,t\rangle_{L^2(E_m)}:=\int_{S^2}h_m(s,t)\,dA_g. \nonumber
\end{align}
We use the convention that the inner product is linear in the first variable and conjugate-linear in the second variable.  Since \(S^2\) is compact and \(E_m\) is smooth, \(C^\infty(E_m)\) is dense in \(L^2(E_m)\). 
\end{notation}

Let \(\omega\) be the Levi-Civita connection one-form in the oriented orthonormal frame \(e_1,e_2\), defined by
\begin{align}
\nabla e_1=\omega\otimes e_2,\qquad \nabla e_2=-\omega\otimes e_1. \label{eq:connection-form-new}
\end{align}
Then $\nabla Z=i\omega\otimes Z$.

With the sign convention in \eqref{eq:connection-form-new}, the curvature equation is
\begin{align}
d\omega=-K_g\,dA_g. \nonumber
\end{align}
For instance, on the unit round sphere $\mathbb{S}^2$ with \(e_1=\partial_r\) and \(e_2=(\sin r)^{-1}\partial_\theta\), one has \(\omega=\cos r\,d\theta\), hence \(d\omega=-\sin r\,dr\wedge d\theta=-dA\).

The connection induced on \(E_m\) is defined as follows:
\begin{align}
\nabla^{(m)}(u\cdot Z^{\otimes m})=(du+im\omega u)\otimes Z^{\otimes m}. \nonumber
\end{align}
For \(k\in\mathbb Z_{\geq0}\), define \(H^k(E_m)\) to be the Sobolev space of sections obtained by completing \(C^\infty(E_m)\) for the norm
\begin{align}
\|s\|_{H^k(E_m)}^2:=\sum_{r=0}^{k}\int_{S^2}\left|\big(\nabla^{(m)}\big)^r s\right|^2\,dA_g. \nonumber
\end{align}
In particular, \(H^0(E_m)=L^2(E_m)\). 

\begin{notation}
{For a vector field \(X\) on $(S^2, g)$ and any complex-valued function $u$ defined on $(S^2, g)$, define 
\begin{align}
D_X^{(m)}u:=Xu+im\omega(X)u, \quad \quad \forall m\in \mathbb{Z}_{\geq0}. \nonumber
\end{align}
}
\end{notation}

\begin{definition}\label{def:Bm operator}
The first-order differential expression is
\begin{align}
\mathcal B_m:C^\infty(E_m)\longrightarrow C^\infty(E_{m+1}), \quad \quad \mathcal B_m(u\cdot Z^{\otimes m}):=\bigl(D_{e_1}^{(m)}u+iD_{e_2}^{(m)}u\bigr)Z^{\otimes(m+1)}. \label{eq:Bm-definition-new}
\end{align}
It defines an initial densely defined operator:
\begin{align}
B_m^{(0)}:\operatorname{Dom}(B_m^{(0)})=C^\infty(E_m)\subset L^2(E_m) \longrightarrow L^2(E_{m+1}), \qquad B_m^{(0)}s:=\mathcal B_m s. \nonumber
\end{align}
The operator \(B_m^{(0)}\) is closable because its formal adjoint is also a first-order differential operator defined on the dense subspace \(C^\infty(E_{m+1})\).  We write
\begin{align}
B_m:=\overline{B_m^{(0)}}:\operatorname{Dom}(B_m)\subset L^2(E_m)\longrightarrow L^2(E_{m+1}) \nonumber
\end{align}
for its closure.  

Standard elliptic theory identifies \(\operatorname{Dom}(B_m)\) with \(H^1(E_m)\).  Thus
\begin{align}
B_m:\operatorname{Dom}(B_m)=H^1(E_m)\subset L^2(E_m)\longrightarrow L^2(E_{m+1}) . \nonumber
\end{align}
It is a closed unbounded operator in the Hilbert-space sense; it is not an everywhere-defined map on \(L^2(E_m)\).  When acting on smooth sections, we keep the notation \(B_m\) for the differential expression \(\mathcal B_m\).
\end{definition}

This is a global operator. Indeed, under the frame change \eqref{eq:Z-frame-change-new}, the connection form changes by
\begin{align}
\omega'=\omega+d\alpha. \nonumber
\end{align}
Using \(u'=e^{-im\alpha}u\), \(e_1'=\cos\alpha\,e_1+\sin\alpha\,e_2\), and \(e_2'=-\sin\alpha\,e_1+\cos\alpha\,e_2\), one obtains
\begin{align}
D_{e_1'}^{\prime(m)}u'+iD_{e_2'}^{\prime(m)}u' =e^{-i(m+1)\alpha}\bigl(D_{e_1}^{(m)}u+iD_{e_2}^{(m)}u\bigr). \nonumber
\end{align}
Multiplying by \((Z')^{\otimes(m+1)}=e^{i(m+1)\alpha}Z^{\otimes(m+1)}\) gives the same section of \(E_{m+1}\).

\begin{definition}\label{def:Bm* operator}
The Hilbert-space adjoint of the closed densely defined operator
\begin{align}
B_m:\operatorname{Dom}(B_m)=H^1(E_m)\subset L^2(E_m)\longrightarrow L^2(E_{m+1}) \nonumber
\end{align}
is denoted by \(B_m^*\).  Thus
\begin{align}
B_m^*:\operatorname{Dom}(B_m^*)\subset L^2(E_{m+1})\longrightarrow L^2(E_m) \nonumber
\end{align}
is the closed densely defined operator characterized by
\begin{align}
\langle B_ms,t\rangle_{L^2(E_{m+1})}=\langle s,B_m^*t\rangle_{L^2(E_m)}, \qquad s\in\operatorname{Dom}(B_m),\quad t\in\operatorname{Dom}(B_m^*). \nonumber
\end{align}
On smooth sections it agrees with the formal \(L^2\)-adjoint.  More precisely, \(C^\infty(E_{m+1})\subset\operatorname{Dom}(B_m^*)\), and if \(t=vZ^{\otimes(m+1)}\), then
\begin{align}
B_m^*(vZ^{\otimes(m+1)}) =-\bigl(D_{e_1}^{(m+1)}v-iD_{e_2}^{(m+1)}v\bigr)Z^{\otimes m}. \label{eq:Bm-star-local-new}
\end{align}
Again \(\operatorname{Dom}(B_m^*)=H^1(E_{m+1})\) by elliptic theory on the closed surface.
\end{definition}

{
\begin{remark}[The rotationally symmetric ladder as a symmetry reduction]

The intrinsic construction above contains the rotationally symmetric
two-sphere as a special case.  More precisely, this literal identification
concerns the case \(n=3\) in the rotationally symmetric part, beginning
with Section~\ref{sec:rotational}.

Let
\begin{align}
g=dr^2+f(r)^2\,d\theta^2 \nonumber
\end{align}
be a rotationally symmetric metric on \(S^2\), and use the adapted
oriented orthonormal frame
\begin{align}
e_1=\partial_r, \qquad e_2=\frac{1}{f(r)}\partial_\theta . \nonumber
\end{align}
Its Levi-Civita connection form is
\begin{align}
\omega=f'(r)\,d\theta, \qquad \omega(e_1)=0, \qquad \omega(e_2)=\frac{f'}{f}. \nonumber
\end{align}
Consequently, for a local section
\(s=u(r,\theta)Z^{\otimes m}\) of \(E_m\), the general operator \(B_m\)
takes the form
\begin{align}
B_m s = \left( \partial_r u+\frac{i}{f}\partial_\theta u -m\frac{f'}{f}u \right)Z^{\otimes(m+1)}. \label{eq:Bm-rotational-restriction}
\end{align}

In the adapted moving frame \(Z^{\otimes m}\), the \(m\)-th rotational
weight is carried by the frame itself.  Hence the complexification of the
\(m\)-th Fourier sector is represented by sections whose coefficient is
radial.  Restricting \eqref{eq:Bm-rotational-restriction} to \(u=u(r)\)
gives
\begin{align}
B_m\bigl(uZ^{\otimes m}\bigr) &= \left(u'-m\frac{f'}{f}u\right)Z^{\otimes(m+1)}= (\mathbf B_m u)Z^{\otimes(m+1)}, \nonumber
\end{align}
where \(\mathbf B_m\) is exactly the radial operator in
\eqref{eq:Bm-rot}.  

On the pole-compatible radial graph core $\left\{
vZ^{\otimes(m+1)}:
v\in\mathscr D_{m+1}\otimes_{\mathbb R}\mathbb C
\right\}$, the adjoint satisfies
\begin{align}
B_m^*\bigl(vZ^{\otimes(m+1)}\bigr) &= \left(-v'-(m+1)\frac{f'}{f}v\right)Z^{\otimes m} = (\mathbf B_m^*v)Z^{\otimes m}, \nonumber
\end{align}
which agrees with the formal adjoint expression
\eqref{eq:Bm-adj-rot} when \(n=3\).

The Hilbert-space structures also agree.  For a radial coefficient,
\begin{align}
\bigl\|uZ^{\otimes m}\bigr\|_{L^2(E_m)}^2 = 2\pi\int_0^a |u(r)|^2f(r)\,dr. \nonumber
\end{align}
Thus, up to the harmless factor \(2\pi\), this is the complexification of
the radial space \(H\) from Section~\ref{sec:rotational}.  Moreover,
smooth extension across the two poles imposes precisely the complexified
pole conditions
\begin{align}
u\in\mathscr D_m\otimes_{\mathbb R}\mathbb C. \nonumber
\end{align}
Taking graph closures therefore identifies the restriction of the closed
bundle operator \(B_m\) with the complexification of the closed radial
operator \(\mathbf B_m\), and similarly for their adjoints.

Finally, since
\begin{align}
K_g=-\frac{f''}{f}=K_{\mathrm{rad}}, \nonumber
\end{align}
the Riemannian ladder identity
\eqref{eq:global-ladder-identity}, when restricted to these
radial-coefficient subspaces, becomes
\begin{align}
\mathbf B_m\mathbf B_m^* - \mathbf B_{m+1}^*\mathbf B_{m+1} = 2(m+1)K_{\mathrm{rad}}, \nonumber
\end{align}
which is exactly the radial ladder identity
\eqref{eq:ladder-comm-rot} for \(n=3\).

Thus the rotationally symmetric proof on \(S^2\) is the
separated-variable, symmetry-reduced realization of the intrinsic
Riemannian construction.  For a general metric there is no global angular
coordinate, and the bundles \(E_m\) provide the invariant replacement for
the Fourier-mode decomposition.
\end{remark}
}

The key identity is the following Weitzenbock-type formula.
\begin{lemma}[Riemannian ladder identity]\label{lem:cr-ladder} 
On $C^\infty(S^2; \mathbb{C})$,
\begin{align}
B_0^*B_0=-\Delta_g \label{eq:B0-lap-new}
\end{align}
For every \(m\in\mathbb{Z}_{\geq0}\),
\begin{align}
B_mB_m^*-B_{m+1}^*B_{m+1}=2(m+1)K_g \label{eq:global-ladder-identity}
\end{align}
as identities of differential expressions on \(C^\infty(E_{m+1})\). The right-hand side means multiplication by the scalar function \(2(m+1)K_g\).
\end{lemma}

\begin{proof}
	\textbf{Step (1)}. We prove \(B_0^*B_0=-\Delta_g\).  For \(m=0\),
	\begin{align}
B_0u=(e_1u+i e_2u)Z. \nonumber
\end{align}
	Using the local formula for the adjoint,
	\begin{align}
B_0^*(vZ) = -\left(D_{e_1}^{(1)}v-iD_{e_2}^{(1)}v\right), \nonumber
\end{align}
	we get
	\begin{align}
B_0^*B_0u = -\left(D_{e_1}^{(1)}-iD_{e_2}^{(1)}\right) (e_1u+i e_2u). \nonumber
\end{align}
	Since
	\begin{align}
D_{e_1}^{(1)}=e_1+i\omega_1, \qquad D_{e_2}^{(1)}=e_2+i\omega_2, \nonumber
\end{align}
	where \(\omega_j=\omega(e_j)\), this becomes
	\begin{align}
B_0^*B_0u = -\left(e_1-i e_2+i\omega_1+\omega_2\right) (e_1u+i e_2u). \nonumber
\end{align}
	Expanding,
	\begin{align}
(e_1-i e_2)(e_1u+i e_2u) = e_1e_1u+e_2e_2u+i[e_1,e_2]u. \nonumber
\end{align}
	Since the Levi-Civita connection is torsion-free,
	\begin{align}
[e_1,e_2] = \nabla_{e_1}e_2-\nabla_{e_2}e_1 = -\omega_1e_1-\omega_2e_2. \nonumber
\end{align}
	Hence
	\begin{align}
i[e_1,e_2]u = -i\omega_1e_1u-i\omega_2e_2u. \nonumber
\end{align}
	Therefore
	\begin{align}
\left(e_1-i e_2+i\omega_1+\omega_2\right)(e_1u+i e_2u)=e_1e_1u+e_2e_2u+\omega_2e_1u-\omega_1e_2u. \nonumber
\end{align}
	On the other hand,
	\begin{align}
\operatorname{div}e_1=\omega_2, \qquad \operatorname{div}e_2=-\omega_1. \nonumber
\end{align}
	Thus
	\begin{align}
\Delta_g u = e_1e_1u+e_2e_2u +(\operatorname{div}e_1)e_1u +(\operatorname{div}e_2)e_2u = e_1e_1u+e_2e_2u+\omega_2e_1u-\omega_1e_2u. \nonumber
\end{align}
	Consequently,
	\begin{align}
B_0^*B_0u=-\Delta_g u. \nonumber
\end{align}

\textbf{Step (2)}. The assertion is local. Fix a point \(p\in S^2\), and choose the oriented orthonormal frame so that
\begin{align}
\nabla e_1(p)=0,\qquad \nabla e_2(p)=0. \nonumber
\end{align}
Then \(\omega_1(p)=\omega_2(p)=0\), but the first derivatives of \(\omega\) at \(p\) remain. Since \([e_1,e_2](p)=0\), we have
\begin{align}
-K_g(p)= d\omega(e_1,e_2)(p)=e_1\omega_2(p)-e_2\omega_1(p). \label{eq:domega-point-new}
\end{align}

Let \(s=u\cdot Z^{\otimes m}\) be a local section of \(E_m\). At the point \(p\), formulas \eqref{eq:Bm-definition-new} and \eqref{eq:Bm-star-local-new} give
\begin{align}
B_m^*B_m(u\cdot Z^{\otimes m}) &=-\bigl(D_{e_1}^{(m+1)}-iD_{e_2}^{(m+1)}\bigr) \bigl(D_{e_1}^{(m)}+iD_{e_2}^{(m)}\bigr)u\,Z^{\otimes m}. \nonumber
\end{align}
Expanding the right-hand side at \(p\), and using \(\omega_1(p)=\omega_2(p)=0\), \eqref{eq:domega-point-new} gives
\begin{align}
B_m^*B_m(u\cdot Z^{\otimes m}) &=\bigl[-e_1e_1u-e_2e_2u-im(e_1\omega_1+e_2\omega_2)u +m(e_1\omega_2-e_2\omega_1)u\bigr]Z^{\otimes m} \nonumber\\
&= \bigl[-e_1e_1u-e_2e_2u-im(e_1\omega_1+e_2\omega_2)u -mK_g u\bigr]Z^{\otimes m}. \label{eq:BstarB-expansion-new}
\end{align}

\textbf{Step (3)}. Next let \(s=u\cdot Z^{\otimes m}\) be a local section of \(E_m\), now viewed as the target bundle of \(B_{m-1}\). From \eqref{eq:Bm-definition-new} and \eqref{eq:Bm-star-local-new},
\begin{align}
B_{m-1}B_{m-1}^*(u\cdot Z^{\otimes m}) &=-\bigl(D_{e_1}^{(m-1)}+iD_{e_2}^{(m-1)}\bigr) \bigl(D_{e_1}^{(m)}-iD_{e_2}^{(m)}\bigr)u\,Z^{\otimes m}. \nonumber
\end{align}

The same expansion gives
\begin{align}
B_{m-1}B_{m-1}^*(u\cdot Z^{\otimes m}) &=\bigl[-e_1e_1u-e_2e_2u-im(e_1\omega_1+e_2\omega_2)u -m(e_1\omega_2-e_2\omega_1)u\bigr]Z^{\otimes m} \nonumber\\
&= \bigl[-e_1e_1u-e_2e_2u-im(e_1\omega_1+e_2\omega_2)u +mK_g u\bigr]Z^{\otimes m} .\label{eq:BBstar-expansion-new}
\end{align}

Subtracting \eqref{eq:BstarB-expansion-new} from \eqref{eq:BBstar-expansion-new} gives
\begin{align}
B_{m-1}B_{m-1}^*-B_m^*B_m=2mK_g \qquad\text{on }E_m. \nonumber
\end{align}
Replacing \(m\) by \(m+1\) proves \eqref{eq:global-ladder-identity}.
\end{proof}

\section{The Vanishing Theorem}\label{sec:Vanishing-theorem}

The ladder identity alone is not enough for a finite eigenvalue comparison.  The counting recursion also needs the exact size of the unpaired kernel and the absence of an adjoint kernel.  In the rotational proof, these facts were endpoint facts for the one-dimensional equation $B_mu=0$.  Globally, endpoint analysis is replaced by the geometry of holomorphic sections.

This section is therefore the bridge between the Riemannian ladder and complex analysis.  We need two facts:
\begin{align}
\dim_\C\ker B_m=2m+1,\qquad \ker B_m^*=\{0\}. \nonumber
\end{align}

The identity $\dim_\C\ker B_m=2m+1$ is first proved in an elementary way: conformal invariance reduces the problem to the round conformal sphere, and stereographic coordinates identify the kernel with polynomials of degree at most $2m$.  

The adjoint vanishing $\ker B_m^*=0$ is a Bochner consequence of the curvature term in the ladder identity.  

\begin{lemma}\label{lem:conformal-invariance-Bm}
Let \(\widehat g=e^{2\varphi}g\), where $\varphi\in C^\infty(S^2)$. Define \(\widehat B_m\) from \(\widehat g\) in the same way as \(B_m\) is defined from \(g\). Then
\begin{align}
\ker \widehat B_m=\ker B_m. \nonumber
\end{align}
\end{lemma}

\begin{proof}
Let \(e_1,e_2\) be a local oriented \(g\)-orthonormal frame and put
\begin{align}
\widehat e_j=e^{-\varphi}e_j, \qquad j=1,2. \nonumber
\end{align}
Then \(\widehat e_1,\widehat e_2\) is an oriented \(\widehat g\)-orthonormal frame, and $\widehat Z=e^{-\varphi}Z$.

If \(s=u\cdot Z^{\otimes m}=\widehat u\widehat Z^{\otimes m}\), then
\begin{align}
\widehat u=e^{m\varphi}u. \label{eq:u-hat-transform-new}
\end{align}
The connection one-form changes by
\begin{align}
\widehat\omega=\omega-d\varphi\circ J. \label{eq:omega-conformal-transform-new}
\end{align}
In components, if \(\varphi_j=e_j\varphi\), then
\begin{align}
\widehat\omega(e_1)=\omega(e_1)-\varphi_2,\qquad \widehat\omega(e_2)=\omega(e_2)+\varphi_1. \nonumber
\end{align}
Using \eqref{eq:u-hat-transform-new} and \eqref{eq:omega-conformal-transform-new}, we compute
\begin{align}
\widehat D_{\widehat e_1}^{(m)}\widehat u+i\widehat D_{\widehat e_2}^{(m)}\widehat u &=e^{(m-1)\varphi}\bigl(D_{e_1}^{(m)}u+iD_{e_2}^{(m)}u\bigr). \nonumber
\end{align}
Therefore
\begin{align}
\widehat B_m s &=e^{(m-1)\varphi}\bigl(D_{e_1}^{(m)}u+iD_{e_2}^{(m)}u\bigr)\widehat Z^{\otimes(m+1)} \nonumber\\
&=e^{(m-1)\varphi}\bigl(D_{e_1}^{(m)}u+iD_{e_2}^{(m)}u\bigr)e^{-(m+1)\varphi}Z^{\otimes(m+1)}=e^{-2\varphi}B_m s. \nonumber
\end{align}
This implies the conclusion.
\end{proof}

\begin{lemma}\label{lem:ker-Bm-dimension-new}
For every \(m\geq0\), $\dim_\C\ker B_m=2m+1.$
\end{lemma}

\begin{proof}
By the uniformization theorem, every smooth metric on \(S^2\) is conformal, after an orientation-preserving diffeomorphism, to the round metric \cite{JostRiemann}. 

The construction of \(E_m\) and \(B_m\) is natural under diffeomorphisms, and Lemma \ref{lem:conformal-invariance-Bm} shows that \(\ker B_m\) is unchanged by a conformal change of metric. 

Hence it is enough to compute \(\ker B_m\) for the round metric.

Let \(z=x+iy\) be the stereographic coordinate on \(\mathbb S^2\setminus\{\infty\}\). The round metric is
\begin{align}
g_{\mathrm{round}}=e^{2\psi}(dx^2+dy^2),\qquad e^\psi=\frac{2}{1+|z|^2}. \nonumber
\end{align}
Take
\begin{align}
e_1=e^{-\psi}\partial_x,\qquad e_2=e^{-\psi}\partial_y. \nonumber
\end{align}
Then
\begin{align}
\omega=-\psi_y\,dx+\psi_x\,dy. \label{eq:omega-stereo-new}
\end{align}
Let \(s=u\cdot Z^{\otimes m}\) be a section on \(\mathbb S^2\setminus\{\infty\}\). Since
\begin{align}
\partial_{\bar z}:=\frac12(\partial_x+i\partial_y), \nonumber
\end{align}
formulas \eqref{eq:Bm-definition-new} and \eqref{eq:omega-stereo-new} give
\begin{align}
B_m s=0 &\Longleftrightarrow \partial_{\bar z}u-m(\partial_{\bar z}\psi)u=0 \Longleftrightarrow \partial_{\bar z}(e^{-m\psi}u)=0. \nonumber
\end{align}

Thus $P(z):=e^{-m\psi}u(z)$ is an entire complex analytic function on \(\C\).

Now write the section itself in the coordinate vector frame. Since
\begin{align}
Z=\frac{1}{\sqrt2}(e_1-i e_2)=\sqrt2\,e^{-\psi}\partial_z,\qquad \partial_z:=\frac12(\partial_x-i\partial_y), \nonumber
\end{align}
we have
\begin{align}
s=u\cdot Z^{\otimes m} =e^{m\psi}P(z)\bigl(\sqrt2\,e^{-\psi}\partial_z\bigr)^{\otimes m} =2^{m/2}P(z)(\partial_z)^{\otimes m}. \label{eq:s-P-dz-new}
\end{align}

Near infinity, use the coordinate \(w=1/z\). Then
\begin{align}
\partial_z=-w^2\partial_w. \nonumber
\end{align}
Therefore \eqref{eq:s-P-dz-new} becomes
\begin{align}
s=2^{m/2}(-1)^m w^{2m}P(1/w)(\partial_w)^{\otimes m}. \nonumber
\end{align}
The section \(s\) is smooth at \(w=0\) if and only if the coefficient
\begin{align}
w^{2m}P(1/w) \nonumber
\end{align}
has no negative powers in its Laurent expansion at \(w=0\). Since \(P\) is entire, this is equivalent to saying that \(P\) is a polynomial of degree at most \(2m\):
\begin{align}
P(z)=a_0+a_1z+\cdots+a_{2m}z^{2m}. \nonumber
\end{align}
Conversely, every such polynomial gives a smooth global section satisfying \(B_ms=0\). Hence
\begin{align}
\ker B_m=\{2^{m/2}P(z)(\partial_z)^{\otimes m}:\ P\text{ is a polynomial of degree }\leq2m\}. \nonumber
\end{align}
The complex dimension is therefore \(2m+1\).
\end{proof}

\begin{lemma}\label{lem:adjoint-kernel-new}
For every \(m\geq0\),
\begin{align}
\ker B_m^*=\{0\}. \nonumber
\end{align}
\end{lemma}

\begin{proof}
We apply Lemma~\ref{lem:abstract-adjoint-kernel-gap} to the concrete Hilbert
spaces
\begin{align}
H_m=L^2(E_m),\qquad H_{m+1}=L^2(E_{m+1}), \nonumber
\end{align}
the concrete closed operator
\(B_m:\operatorname{Dom}(B_m)=H^1(E_m)\to L^2(E_{m+1})\), and the shifts
\(\alpha_m=m(m+1)\).  Indeed, Lemma~\ref{lem:cr-ladder} and \(K_g\geq1\)
give, first on smooth sections and then on all of \(H^1(E_{m+1})\) by
density,
\begin{align}
&\|B_m^*t\|_{L^2(E_m)}^2+m(m+1)\|t\|_{L^2(E_{m+1})}^2\geq \|B_{m+1}t\|_{L^2(E_{m+2})}^2+(m+1)(m+2)\|t\|_{L^2(E_{m+1})}^2 . \nonumber
\end{align}
Here both form domains are \(H^1(E_{m+1})\), so the domain inclusion required
in Lemma~\ref{lem:abstract-adjoint-kernel-gap} is automatic.  Since
\((m+1)(m+2)>m(m+1)\), that lemma yields \(\ker B_m^*=\{0\}\).
\end{proof}

\section{The comparison theorem and rigidity}
\label{sec:comparison-rigidity}

For the remainder of Part~II we use the following notation.  For every
\(m\geq0\), let
\begin{align}
H_m:=L^2(E_m),\qquad \alpha_m:=m(m+1), \nonumber
\end{align}
and define the closed forms
\begin{align}
\mathfrak a_m(s,t) &:=\langle B_ms,B_mt\rangle_{H_{m+1}} +m(m+1)\langle s,t\rangle_{H_m}, &\operatorname{Dom}(\mathfrak a_m)&=H^1(E_m), \nonumber\\
\mathfrak c_m(s,t) &:=\langle B_m^*s,B_m^*t\rangle_{H_m} +m(m+1)\langle s,t\rangle_{H_{m+1}}, &\operatorname{Dom}(\mathfrak c_m)&=H^1(E_{m+1}). \nonumber
\end{align}
Let \(A_m\) and \(C_m\) be their associated self-adjoint operators.  Thus
\begin{align}
A_m&=B_m^*B_m+m(m+1)I\quad\text{in }H_m, \nonumber\\
C_m&=B_mB_m^*+m(m+1)I\quad\text{in }H_{m+1}. \nonumber
\end{align}
Elliptic regularity and Rellich compactness on the closed surface show that
both operators have compact resolvent.  Their eigenvalues are counted with
complex multiplicity and are denoted by \(\lambda_j(A_m)\) and \(\lambda_j(C_m)\),
starting at \(j=0\).

The geometric content of the comparison is the following exact form identity:
for every \(t\in H^1(E_{m+1})\),
\begin{align}
\mathfrak c_m(t,t)-\mathfrak a_{m+1}(t,t) =2(m+1)\int_{S^2}(K_g-1)|t|^2\,dA_g. \label{eq:smooth-ladder-defect-form}
\end{align}
It follows from Lemma~\ref{lem:cr-ladder} on smooth sections and extends to
\(H^1(E_{m+1})\) by density.  In particular,
\begin{align}
\operatorname{Dom}(\mathfrak c_m) &=\operatorname{Dom}(\mathfrak a_{m+1}), \nonumber\\
\mathfrak c_m(t,t)&\geq\mathfrak a_{m+1}(t,t) \qquad\text{when }K_g\geq1. \label{Cm-greater-Am+1}
\end{align}

\begin{prop}\label{prop:crucial induction ineq}
For every \(m,j\geq0\),
\begin{align}
\lambda_{j+2m+1}(A_m)=\lambda_j(C_m)\geq\lambda_j(A_{m+1}). \label{eq:a-ladder-global}
\end{align}
Consequently, for every \(l\in\mathbb Z^+\),
\begin{align}
N_{A_m}^{<}\bigl(l(l+1)\bigr) \leq(2m+1)+N_{A_{m+1}}^{<}\bigl(l(l+1)\bigr). \nonumber
\end{align}
The same assertion holds for the closed counting functions \(\mathcal N\).
\end{prop}

\begin{proof}
Apply Proposition~\ref{prop:abstract-ladder-recursion} to the concrete spaces
\(H_m=L^2(E_m)\), the closed operator \(B_m\), and the shift
\(\alpha_m=m(m+1)\).  The compact-resolvent assumption in
Notation~\ref{not:abstract-ladder-package} was verified above.  The required
form-domain inclusion and form order are exactly
\eqref{Cm-greater-Am+1}.  Finally,
Lemma~\ref{lem:ker-Bm-dimension-new} gives
\(r_m=\dim_{\C}\ker B_m=2m+1\).  Formula
\eqref{eq:a-ladder-global} and both counting recursions are therefore the
specialization of \eqref{eq:abstract-ladder-eigenvalue-step} and
\eqref{eq:abstract-ladder-counting-step}.
\end{proof}

\begin{theorem}\label{thm:cluster-bottom-comparison}
Let \((S^2,g)\) be a smooth Riemannian two-sphere with \(K_g\geq1\).  Then,
for every \(l\in\mathbb Z^+\),
\begin{align}
N_{-\Delta_g}^{<}\bigl(l(l+1)\bigr)&\leq l^2, \label{eq:smooth-strict-count-comparison}\\
\mathcal N_{-\Delta_g}\bigl(l(l+1)\bigr)&\leq(l+1)^2. \label{eq:smooth-closed-count-comparison}
\end{align}
Equivalently,
\begin{align}
\lambda_i(S^2,g)\geq\lambda_i(\mathbb S^2,g_{\mathrm{round}}), \qquad i\geq1. \label{eq:smooth-ordered-comparison}
\end{align}
\end{theorem}

\begin{proof}
We apply Corollary~\ref{cor:abstract-two-sphere-ladder} to
\begin{align}
H_m=L^2(E_m),\qquad B_m, \qquad \alpha_m=m(m+1). \nonumber
\end{align}
The form order is \eqref{Cm-greater-Am+1}, compact resolvent was recorded
above, and Lemma~\ref{lem:ker-Bm-dimension-new} gives the required exact
kernel dimensions \(2m+1\).  Lemma~\ref{lem:cr-ladder} gives
\(A_0=B_0^*B_0=-\Delta_g\).  Hence
\eqref{eq:abstract-two-sphere-threshold-counts} is precisely
\eqref{eq:smooth-strict-count-comparison}--
\eqref{eq:smooth-closed-count-comparison}.

For \eqref{eq:smooth-ordered-comparison}, choose \(l\geq1\) so that
\(l^2\leq i\leq l^2+2l\).  The round eigenvalue at every index in this
cluster is \(l(l+1)\), while the strict-counting estimate gives
\(\lambda_{l^2}(g)\geq l(l+1)\).  Monotonicity of the ordered spectrum gives the
claim.
\end{proof}

\begin{lemma}\label{lem:smooth-strict-form-comparison}
Assume \(K_g\geq1\) and \(K_g\not\equiv1\).  Then, for every \(m,j\geq0\),
\begin{align}
\lambda_j(C_m)>\lambda_j(A_{m+1}). \nonumber
\end{align}
\end{lemma}

\begin{proof}
Fix \(m\) and set
\begin{align}
V:=2(m+1)(K_g-1). \nonumber
\end{align}
By \eqref{eq:smooth-ladder-defect-form}, \(C_m=A_{m+1}+V\) in the
quadratic-form sense.  The smooth function \(V\) is nonnegative and is
positive on a nonempty open set \(\Omega\).

Let \(S_j\) be the direct sum of the eigenspaces of \(C_m\) belonging to its
first \(j+1\) eigenvalues, counted with multiplicity.  We claim that
\begin{align}
\int_{S^2}V|s|^2\,dA_g>0 \qquad\text{for every }0\neq s\in S_j. \nonumber
\end{align}
Otherwise \(s\) vanishes on \(\Omega\).  Decompose \(s\) into the finite sum
of its components corresponding to the distinct eigenvalues of \(C_m\).

Because \(s\) vanishes on \(\Omega\), so do \(C_m^qs\) for all
\(q\geq0\).  A Vandermonde argument therefore shows that each individual eigencomponent vanishes on \(\Omega\). 

The unique-continuation theorem for
second-order elliptic operators with smooth coefficients forces each such
component to vanish identically, a contradiction.

The unit sphere of the finite-dimensional space \(S_j\) is compact, so there
is \(\delta_j>0\) such that
\begin{align}
\int_{S^2}V|s|^2\,dA_g\geq\delta_j\|s\|^2 \qquad (s\in S_j). \nonumber
\end{align}
Consequently,
\begin{align}
\sup_{0\neq s\in S_j} \frac{\mathfrak a_{m+1}(s,s)}{\|s\|^2} \leq\lambda_j(C_m)-\delta_j. \nonumber
\end{align}
The min--max principle for \(A_{m+1}\) now gives
\(\lambda_j(A_{m+1})<\lambda_j(C_m)\).
\end{proof}

\begin{theorem}\label{thm:rigidity-main}
Let \((S^2,g)\) be a smooth Riemannian two-sphere with \(K_g\geq1\).  If
\begin{align}
\lambda_i(S^2,g)=\lambda_i(\mathbb S^2,g_{\mathrm{round}}) \label{eq:smooth-individual-equality}
\end{align}
for some \(i\geq1\), or if
\begin{align}
\mathcal N_{-\Delta_g}\bigl(l(l+1)\bigr)=(l+1)^2 \label{eq:smooth-counting-equality}
\end{align}
holds for some \(l\in\mathbb Z^+\), then \((S^2,g)\) is isometric to the
unit round sphere.
\end{theorem}

\begin{proof}
Suppose first that \(K_g\not\equiv1\).  Combining
Proposition~\ref{prop:crucial induction ineq} with
Lemma~\ref{lem:smooth-strict-form-comparison} gives
\begin{align}
\lambda_{j+2m+1}(A_m)>\lambda_j(A_{m+1}) \qquad(m,j\geq0). \nonumber
\end{align}
For fixed \(l\geq1\), iterate this inequality for
\(m=0,1,\ldots,l-1\).  Since
\(\sum_{m=0}^{l-1}(2m+1)=l^2\), one obtains
\begin{align}
\lambda_{l^2}(g)=\lambda_{l^2}(A_0)> \lambda_0(A_l)=l(l+1); \nonumber
\end{align}
the last equality uses any nonzero element of \(\ker B_l\), whose existence
follows from Lemma~\ref{lem:ker-Bm-dimension-new}.  If
\eqref{eq:smooth-individual-equality} holds and
\(l^2\leq i\leq l^2+2l\), then
\begin{align}
\lambda_i(g)\geq\lambda_{l^2}(g)>l(l+1) =\lambda_i(\mathbb S^2,g_{\mathrm{round}}), \nonumber
\end{align}
a contradiction.  Thus individual eigenvalue equality forces
\(K_g\equiv1\).

Now assume \eqref{eq:smooth-counting-equality}, and set $i=(l+1)^2-1=l^2+2l$.  The counting equality implies $\lambda_i(S^2,g)\leq l(l+1)$, whereas the ordered comparison \eqref{eq:smooth-ordered-comparison} gives
\begin{align}
\lambda_i(S^2,g)\geq\lambda_i(\mathbb S^2,g_{\mathrm{round}})=l(l+1). \nonumber
\end{align}
Thus \eqref{eq:smooth-individual-equality} holds at the index $i$, and the preceding argument gives $K_g\equiv1$.

In either case, a complete simply connected surface of constant curvature
\(1\) is the unit round sphere.  Hence \((S^2,g)\) is round.
\end{proof}

Theorems~\ref{thm:cluster-bottom-comparison} and
\ref{thm:rigidity-main} prove the smooth comparison and rigidity assertions
of Theorem~\ref{thm:main-eigenvalue}.  The role of
Section~\ref{sec:abstract-ladder-counting} is now explicit: the present part
supplies the bundle operator, the curvature form order, and the kernel
dimension; Section~\ref{sec:abstract-ladder-counting} supplies the partner
spectra, index shifts, counting recursion, and equality propagation.

The proof is genuinely two-dimensional.  Section~\ref{sec:S3-counterexample}
shows that the direct \(S^3\) analogue with \(\operatorname{Ric}\geq2\) is
false, and the higher-dimensional examples in \cite{Aryan2026} show that a
Ricci lower bound alone does not yield the same finite-index rigidity.

\section{A complex-geometric proof}\label{sec:complex-geometric-proof}

This section gives a complex-geometric proof of Theorem~\ref{thm:main-eigenvalue}.  The point is that the bundles appearing in the ladder are powers of the anticanonical bundle on the Riemann surface underlying the oriented Riemannian sphere, and the first-order operators are Dolbeault operators after a metric identification.

We use Griffiths--Harris as the standard reference for the complex-geometric background: Dolbeault's theorem, the correspondence
between divisor classes and holomorphic line bundles on compact Riemann surfaces, Riemann--Roch for compact Riemann surfaces, global duality, Hermitian holomorphic bundles, Chern connections, curvature; see \cite[Ch.~0, \S3, p.~45; Ch.~0, \S5, pp.~66--79; Ch.~0, \S7, pp.~106--122; Ch.~1, \S1, pp.~129--139; Ch.~2, \S3, pp.~240--246; Ch.~5, \S4, pp.~705--707]{GriffithsHarrisPAG}.  

The one-dimensional Bochner--Kodaira formula needed below is proved explicitly for the reader's convenience.


If \(L\to\Sigma\) is a holomorphic line bundle over a compact Riemann surface, let \(\mathcal O(L)\) denote the sheaf of holomorphic sections of \(L\), and set
\begin{align}
h^q(L):=\dim_\C H^q(\Sigma,\mathcal O(L)), \qquad q=0,1. \nonumber
\end{align}

Let $K_\Sigma:=\Lambda^{1,0}T^*\Sigma$ be the canonical holomorphic line bundle, then $K_\Sigma^{-1}=T^{1,0}\Sigma.$

For each integer \(m\geq0\), define
\begin{align}
L_m:=K_\Sigma^{-m}=(T^{1,0}\Sigma)^{\otimes m}, \qquad L_0:=\Sigma\times\C. \nonumber
\end{align}
Thus \(L_m\) is the same smooth complex line bundle as the bundle \(E_m\) used in Sections~\ref{sec:riemannian-ladder}--\ref{sec:comparison-rigidity}, now equipped with its natural holomorphic structure.

\begin{theorem}[Dolbeault, Riemann--Roch, Serre duality]\label{thm:RR-Serre-used}
Let \(\Sigma\) be a compact Riemann surface of genus \(\mathfrak{g}\), and let \(L\to\Sigma\) be a holomorphic line bundle.  Then
\begin{align}
&H^{0,q}_{\bar\partial}(\Sigma,L)\simeq H^q(\Sigma,\mathcal O(L)), \qquad q=0,1. \nonumber\\
&h^0(L)-h^1(L)=\deg L+1-\mathfrak{g}, \nonumber\\
&H^1(\Sigma,\mathcal O(L))\simeq H^0(\Sigma,\mathcal O(K_\Sigma\otimes L^{-1}))^*. \nonumber
\end{align}
\end{theorem}

\begin{proof}
These are standard results.  For Dolbeault's theorem, see Griffiths--Harris \cite[Ch.~0, \S3, p.~45]{GriffithsHarrisPAG}.  For Riemann--Roch on compact Riemann surfaces, see \cite[Ch.~2, \S3, pp.~240--246]{GriffithsHarrisPAG}.  The divisor--line-bundle correspondence used to write the statement for holomorphic line bundles is recalled in \cite[Ch.~1, \S1, pp.~129--139]{GriffithsHarrisPAG}.  The Serre-duality statement used here is the one-dimensional line-bundle case of global duality; see \cite[Ch.~5, \S4, pp.~705--707]{GriffithsHarrisPAG}.
\end{proof}

\begin{lemma}[Anticanonical powers on the conformal sphere]
\label{lem:CP1-anticanonical-cohomology}
Let \(\Sigma\) be a compact Riemann surface biholomorphic to \(\CP^1\). Then
\begin{align}
h^0(L_m)=\dim_\C H^0(\Sigma,L_m)=2m+1, \label{eq:h0-Lm-2m1}
\end{align}
and
\begin{align}
H^1(\Sigma,\mathcal O(L_m))=0. \label{eq:H1-Lm-zero}
\end{align}
Finally, for every \(p\in\Sigma\), the evaluation map
\begin{align}
\operatorname{ev}_p:H^0(\Sigma,L_m)\longrightarrow (L_m)_p \nonumber
\end{align}
is surjective.  Consequently,
\begin{align}
\dim_\C\{s\in H^0(\Sigma,L_m):s(p)=0\}=2m. \label{eq:CP1-evaluation-kernel-dim}
\end{align}
\end{lemma}

\begin{proof}
\textbf{Step (1)}. Equivalently, \(\Sigma\) has genus \(\mathfrak g=0\).  The canonical bundle
of a compact Riemann surface of genus \(\mathfrak g\) has degree
\[
\deg K_\Sigma=2\mathfrak g-2;
\]
see \cite[\S5.4, Lemma~5.4.4, p.~215]{JostRiemann}.  Hence, by additivity
of degree under tensor products and dualization,
\begin{align}
\deg L_m
&=\deg\bigl(K_\Sigma^{-m}\bigr)
=-m\,\deg K_\Sigma
=2m.
\nonumber
\end{align}

Theorem \ref{thm:RR-Serre-used} gives
\begin{align}
h^0(L_m)-h^0(K_\Sigma\otimes L_m^{-1})=2m+1. \label{eq:RR-Lm}
\end{align}
But
\begin{align}
K_\Sigma\otimes L_m^{-1}=K_\Sigma^{m+1}, \qquad \deg(K_\Sigma^{m+1})=-2(m+1)<0. \nonumber
\end{align}
A holomorphic line bundle of negative degree has no nonzero holomorphic section: if \(s\not\equiv0\) is holomorphic, then its zero divisor is effective and has degree equal to the degree of the line bundle, which is impossible when the degree is negative.  Thus
\begin{align}
H^0(\Sigma,K_\Sigma^{m+1})=0. \nonumber
\end{align}
Consequently, \eqref{eq:RR-Lm} gives \eqref{eq:h0-Lm-2m1}.  Serre duality gives
\begin{align}
H^1(\Sigma,\mathcal O(L_m)) \simeq H^0(\Sigma,K_\Sigma\otimes L_m^{-1})^* =H^0(\Sigma,K_\Sigma^{m+1})^*=0, \nonumber
\end{align}
which proves \eqref{eq:H1-Lm-zero}.


\textbf{Step (2)}. Since
\begin{align}
K_{\CP^1}\simeq\mathcal O_{\CP^1}(-2), \nonumber
\end{align}
one has
\begin{align}
L_m\simeq\mathcal O_{\CP^1}(2m). \label{eq:CP1-anticanonical-O2m}
\end{align}
Moreover,

By \eqref{eq:CP1-anticanonical-O2m}, \(L_m\) is identified with \(\mathcal O_{\CP^1}(2m)\), and \(\mathcal O_{\CP^1}(2m)\) is globally generated for \(m\geq0\).  

Therefore for every \(p\in\Sigma\) there exists a holomorphic section of \(L_m\) which does not vanish at \(p\).  Since the target \((L_m)_p\) is one-dimensional, \(\operatorname{ev}_p\) is surjective.  Its kernel has codimension one in \(H^0(\Sigma,L_m)\), and \eqref{eq:h0-Lm-2m1} gives \eqref{eq:CP1-evaluation-kernel-dim}.
\end{proof}


Let \((S^2,g)\) be oriented.  The orientation and metric define the complex structure
\begin{align}
J:TS^2\longrightarrow TS^2, \qquad Je_1=e_2, \qquad Je_2=-e_1 \nonumber
\end{align}
for every local oriented \(g\)-orthonormal frame \(e_1,e_2\).  Denote the resulting compact Riemann surface by
\begin{align}
\Sigma:=(S^2,J). \nonumber
\end{align}

For a holomorphic line bundle \(L\to\Sigma\), write
\begin{align}
\Omega^{0,q}(\Sigma,L):=C^\infty\bigl(\Lambda^{0,q}T^*\Sigma\otimes L\bigr), \qquad q=0,1. \nonumber
\end{align}
The Dolbeault operator is
\begin{align}
\bar\partial_L:\Omega^{0,0}(\Sigma,L)\longrightarrow \Omega^{0,1}(\Sigma,L), \nonumber
\end{align}
and for \(L=L_m\) we write $\bar\partial_m:=\bar\partial_{L_m}$.

Let \(e_1,e_2\) be a local oriented orthonormal frame, with dual coframe \(\theta^1,\theta^2\).  Set
\begin{align}
Z:=\frac{1}{\sqrt2}(e_1-i e_2), \qquad \zeta:=\frac{1}{\sqrt2}(\theta^1+i\theta^2), \qquad \bar\zeta:=\frac{1}{\sqrt2}(\theta^1-i\theta^2). \nonumber
\end{align}
Then \(Z\) is a local unitary frame of \(T^{1,0}\Sigma=K_\Sigma^{-1}\), \(\zeta\) is the dual local unitary frame of \(K_\Sigma\), and \(\bar\zeta\) is a local unitary frame of \(\Lambda^{0,1}T^*\Sigma\).

The metric gives a unitary bundle isomorphism
\begin{align}
\mathcal I_m:\Lambda^{0,1}T^*\Sigma\otimes L_m\longrightarrow L_{m+1}, \quad\quad \mathcal I_m(\bar\zeta\otimes Z^{\otimes m})=Z^{\otimes(m+1)}. \nonumber
\end{align}

This definition is independent of the oriented orthonormal frame: if \(Z'=e^{i\alpha}Z\), then \(\bar\zeta'=e^{i\alpha}\bar\zeta\), and both sides acquire the same factor \(e^{i(m+1)\alpha}\).  In particular, $\mathcal I_m^*=\mathcal I_m^{-1}.$

Define
\begin{align}
B_m^{\mathrm{Dol}}:=\sqrt2\,\mathcal I_m\circ\bar\partial_m: C^\infty(L_m)\longrightarrow C^\infty(L_{m+1}). \label{eq:BDol-definition}
\end{align}

Let \(\omega\) be the Levi-Civita connection one-form defined by
\begin{align}
\nabla e_1=\omega\otimes e_2, \qquad \nabla e_2=-\omega\otimes e_1. \nonumber
\end{align}

The induced Chern connection on \(L_m\) is
\begin{align}
\nabla^{(m)}(u\cdot Z^{\otimes m})=(du+i\cdot m u\omega)\otimes Z^{\otimes m}. \nonumber
\end{align}

The complexified cotangent bundle splits as
\begin{align}
T^*\Sigma\otimes_{\mathbb R}\mathbb C = \Lambda^{1,0}T^*\Sigma\oplus \Lambda^{0,1}T^*\Sigma. \nonumber
\end{align}
Let
\begin{align}
\Pi^{1,0}:T^*\Sigma\otimes_{\mathbb R}\mathbb C\to \Lambda^{1,0}T^*\Sigma, \qquad \Pi^{0,1}:T^*\Sigma\otimes_{\mathbb R}\mathbb C\to \Lambda^{0,1}T^*\Sigma \nonumber
\end{align}
be the corresponding projections.  Equivalently, for a complex-valued one-form \(\alpha\),
\begin{align}
\Pi^{1,0}\alpha = \frac12\bigl(\alpha-i\,\alpha\circ J\bigr), \qquad \Pi^{0,1}\alpha = \frac12\bigl(\alpha+i\,\alpha\circ J\bigr), \nonumber
\end{align}
where \((\alpha\circ J)(X)=\alpha(JX)\).

We define
\begin{align}
(\nabla^{(m)})' := (\Pi^{1,0}\otimes \operatorname{Id}_{L_m})\circ \nabla^{(m)}, \qquad (\nabla^{(m)})'' := (\Pi^{0,1}\otimes \operatorname{Id}_{L_m})\circ \nabla^{(m)}. \nonumber
\end{align}
Thus
\begin{align}
\nabla^{(m)}=(\nabla^{(m)})'+(\nabla^{(m)})'', \nonumber
\end{align}
with
\begin{align}
(\nabla^{(m)})': C^\infty(L_m)\to \Omega^{1,0}(\Sigma,L_m), \qquad (\nabla^{(m)})'': C^\infty(L_m)\to \Omega^{0,1}(\Sigma,L_m). \nonumber
\end{align}

Writing
\begin{align}
D^{(m)}_{e_j}u:=e_ju+imu\omega(e_j), \qquad j=1,2, \nonumber
\end{align}

One obtains
\begin{align}
\bar\partial_m(u\cdot Z^{\otimes m}) =\frac{1}{\sqrt2}\bigl(D_{e_1}^{(m)}u+iD_{e_2}^{(m)}u\bigr)\bar\zeta\otimes Z^{\otimes m}. \label{eq:dbar-local-B}
\end{align}
Combining \eqref{eq:BDol-definition} and \eqref{eq:dbar-local-B},
\begin{align}
B_m^{\mathrm{Dol}}(u\cdot Z^{\otimes m})= \bigl(D_{e_1}^{(m)}u+iD_{e_2}^{(m)}u\bigr)Z^{\otimes(m+1)}, \nonumber
\end{align}
which is exactly the operator \(B_m\) from \eqref{eq:Bm-definition-new}.  Hence we write simply
\begin{align}
B_m=B_m^{\mathrm{Dol}}. \nonumber
\end{align}

As before, \(B_m\) denotes the closed operator initially defined on \(C^\infty(L_m)\):
\begin{align}
B_m:\operatorname{Dom}(B_m)=H^1(L_m)\subset L^2(L_m)\longrightarrow L^2(L_{m+1}), \nonumber
\end{align}
and \(B_m^*\) denotes its Hilbert-space adjoint.

For $m=0$, $L_0$ is trivial, and $B_0^*B_0=-\Delta_g$ is exactly \eqref{eq:B0-lap-new} from Lemma~\ref{lem:cr-ladder}.

We now apply the metric-independent cohomology from Section \ref{sec:complex-geometric-proof}.  Since \(\Sigma\) is a compact Riemann surface of genus zero, \(\Sigma\simeq\CP^1\).  By construction,
\begin{align}
\ker B_m=\ker\bar\partial_m=H^0(\Sigma,L_m). \nonumber
\end{align}
Applying Lemma~\ref{lem:CP1-anticanonical-cohomology} gives
\begin{align}
\dim_\C\ker B_m=2m+1. \label{eq:complex-ker-Bm-dim}
\end{align}

For the adjoint-kernel vanishing, the same lemma gives
\begin{align}
H^1(\Sigma,\mathcal O(L_m))=0. \nonumber
\end{align}
By Dolbeault's theorem, this is equivalent to
\begin{align}
H^{0,1}_{\bar\partial}(\Sigma,L_m)=0. \label{eq:Dolbeault-H01-zero}
\end{align}

Let \(\Psi\in\operatorname{Dom}(B_m^*)\) satisfy \(B_m^*\Psi=0\).  By elliptic regularity for \(\bar\partial_m^*\), it is enough to argue in the smooth category.  Put
\begin{align}
\eta:=\mathcal I_m^{-1}\Psi\in\Omega^{0,1}(\Sigma,L_m). \nonumber
\end{align}
Since \(B_m=\sqrt2\,\mathcal I_m\bar\partial_m\), the equation \(B_m^*\Psi=0\) is equivalent to
\begin{align}
\bar\partial_m^*\eta=0. \nonumber
\end{align}
Because \(\Sigma\) has complex dimension one, there are no \((0,2)\)-forms, so automatically
\begin{align}
\bar\partial_m\eta=0. \nonumber
\end{align}
Thus \(\eta\) defines a Dolbeault cohomology class in \(H^{0,1}_{\bar\partial}(\Sigma,L_m)\).  By \eqref{eq:Dolbeault-H01-zero}, this class is zero, so there exists \(s\in\Omega^{0,0}(\Sigma,L_m)\) such that
\begin{align}
\eta=\bar\partial_m s. \nonumber
\end{align}
Then
\begin{align}
\|\eta\|_{L^2}^2 =\langle \eta,\bar\partial_m s\rangle_{L^2} =\langle \bar\partial_m^*\eta,s\rangle_{L^2}=0. \nonumber
\end{align}
Hence \(\eta=0\), and therefore \(\Psi=0\).  We have proved
\begin{align}
\ker B_m^*=\{0\}. \label{eq:complex-adjoint-kernel-zero}
\end{align}
This is the complex-geometric replacement for Lemma~\ref{lem:adjoint-kernel-new}.

\begin{remark}
The proof of \eqref{eq:complex-ker-Bm-dim} and \eqref{eq:complex-adjoint-kernel-zero} uses only the complex geometry of \(\mathbb{CP}^1\), together with the smooth Hilbert adjoint after the metric is chosen.  It does not use the curvature lower bound \(K_g\geq1\).
\end{remark}

It remains to recover the curvature ladder.  This is the one place where the Hermitian metric and the above Chern connections \(\nabla^{(m)}\) are essential.

Let $\omega_\Sigma:=dA_g$ be the Kähler form of \((\Sigma,g,J)\).  For \(m\geq0\), set
\begin{align}
\Theta_m:=(\nabla^{(m)})^2. \nonumber
\end{align}
and on \(L_m\)-valued zero-forms define
\begin{align}
\Delta_m'=((\nabla^{(m)})')^*(\nabla^{(m)})', \qquad \Delta_m''=((\nabla^{(m)})'')^*(\nabla^{(m)})''. \nonumber
\end{align}
By \eqref{eq:dbar-local-B},
\begin{align}
(\nabla^{(m)})''=\bar\partial_m. \label{eq:nabla-double-prime-dbar}
\end{align}

\begin{theorem}[One-dimensional Bochner--Kodaira identity]\label{thm:one-dimensional-BK}
For every \(m\geq0\), on smooth sections of \(L_m\),
\begin{align}
\Delta_m'-\Delta_m''=\Lambda_{\omega_\Sigma}(i\Theta_m). \label{eq:one-dimensional-BK-statement}
\end{align}
Equivalently, if \(i\Theta_m=f\omega_\Sigma\), then
\begin{align}
\Delta_m'-\Delta_m''=f. \nonumber
\end{align}
\end{theorem}

\begin{remark}
{This is the one-dimensional line-bundle case of the classical
Bochner--Kodaira identity, specialized here to the bundles \(L_m\) and to the Chern connections \(\nabla^{(m)}\) already introduced above.  The differential-geometric method goes back to Bochner and Kodaira; see Kodaira's 1953 paper
\cite{Kodaira1953DifferentialGeometric}.  The vector-bundle form is
closely related to the Akizuki--Nakano/Nakano identities
\cite{AkizukiNakano1954, Nakano1955}.
We include the proof in this special case to fix the sign convention.
}
\end{remark}

\begin{proof}
The identity is local.  Fix \(p\in\Sigma\), and choose a local oriented orthonormal frame so that the associated unitary \((1,0)\)-coframe \(\zeta\), the dual vector \(Z\), and the unitary frame \(Z^{\otimes m}\) of \(L_m\) are normal at \(p\).  Thus
\begin{align}
\omega_\Sigma=i\zeta\wedge\bar\zeta \nonumber
\end{align}
at \(p\), the first-order frame-connection terms vanish at \(p\), and \([Z,\bar Z](p)=0\).  For a smooth section \(s\) of \(L_m\),
\begin{align}
(\nabla^{(m)})'s=(\nabla^{(m)}_Zs)\zeta, \qquad (\nabla^{(m)})''s=(\nabla^{(m)}_{\bar Z}s)\bar\zeta. \nonumber
\end{align}
At \(p\), the formal adjoints satisfy
\begin{align}
((\nabla^{(m)})')^*(a\zeta)=-\nabla^{(m)}_{\bar Z}a, \qquad ((\nabla^{(m)})'')^*(b\bar\zeta)=-\nabla^{(m)}_Zb. \nonumber
\end{align}
Therefore
\begin{align}
\Delta_m's=-\nabla^{(m)}_{\bar Z}\nabla^{(m)}_Zs, \qquad \Delta_m''s=-\nabla^{(m)}_Z\nabla^{(m)}_{\bar Z}s. \nonumber
\end{align}
Subtracting gives
\begin{align}
(\Delta_m'-\Delta_m'')s =(\nabla^{(m)}_Z\nabla^{(m)}_{\bar Z}-\nabla^{(m)}_{\bar Z}\nabla^{(m)}_Z)s =\Theta_m(Z,\bar Z)s. \nonumber
\end{align}
Write \(i\Theta_m=f\omega_\Sigma\).  Since \(\omega_\Sigma(Z,\bar Z)=i\), one has
\begin{align}
\Theta_m(Z,\bar Z)=f. \nonumber
\end{align}
Thus
\begin{align}
(\Delta_m'-\Delta_m'')s=f s =\Lambda_{\omega_\Sigma}(i\Theta_m)s. \nonumber
\end{align}
Since \(p\) was arbitrary, the identity holds globally.
\end{proof}

The next identity is not a new curvature formula; it is the preceding
one-dimensional Bochner--Kodaira identity rewritten in the ladder notation of this paper.
\begin{lemma}[Dolbeault curvature ladder]\label{lem:complex-curvature-ladder}
For every \(m\geq0\), as operators on \(C^\infty(L_{m+1})\),
\begin{align}
B_mB_m^*-B_{m+1}^*B_{m+1}=2(m+1)K_g. \nonumber
\end{align}
\end{lemma}

\begin{proof}
By \eqref{eq:nabla-double-prime-dbar} and the definition of \(\Delta_{m+ 1}''\),
\begin{align}
\Delta_{m+ 1}''=\bar\partial_{m+ 1}^*\bar\partial_{m+ 1}. \nonumber
\end{align}
By \eqref{eq:BDol-definition} and the unitarity of \(\mathcal I_{m+ 1}\),
\begin{align}
B_{m+ 1}^*B_{m+ 1}=2\bar\partial_{m+ 1}^*\bar\partial_{m+ 1}=2\Delta_{m+ 1}'' \qquad\text{on }C^\infty(L_{m+ 1}). \label{eq:Bell-Bell-dbar-Laplace}
\end{align}

For the \((\nabla^{(m+ 1)})'\)-part, define the unitary metric identification
\begin{align}
\mathcal J_{m}:\Lambda^{1,0}T^*\Sigma\otimes L_{m+ 1}\longrightarrow L_{m}, \qquad \mathcal J_{m}(\zeta\otimes Z^{\otimes(m+ 1)})=Z^{\otimes(m)}. \nonumber
\end{align}
The local formula for \(\nabla^{(m+ 1)}\) gives
\begin{align}
(\nabla^{(m+ 1)})'(u\cdot Z^{\otimes(m+ 1)}) =\frac1{\sqrt2}\bigl(D_{e_1}^{(m+ 1)}u-iD_{e_2}^{(m+ 1)}u\bigr)\zeta\otimes Z^{\otimes(m+ 1)}. \nonumber
\end{align}
The formal adjoint of \(B_{m}\) is locally
\begin{align}
B_{m}^*(u\cdot Z^{\otimes(m+ 1)}) =-\bigl(D_{e_1}^{(m+ 1)}u-iD_{e_2}^{(m+ 1)}u\bigr)Z^{\otimes(m)}. \nonumber
\end{align}
Hence
\begin{align}
B_{m}^*=-\sqrt2\,\mathcal J_{m}(\nabla^{(m+ 1)})'. \nonumber
\end{align}
Since \(\mathcal J_{m}\) is unitary,
\begin{align}
B_{m}B_{m}^*=2((\nabla^{(m+ 1)})')^*(\nabla^{(m+ 1)})'=2\Delta_{m+ 1}' \qquad\text{on }C^\infty(L_{m+ 1}). \label{eq:Bellminus-Bellminus-Dprime-Laplace}
\end{align}
Combining \eqref{eq:Bell-Bell-dbar-Laplace} and \eqref{eq:Bellminus-Bellminus-Dprime-Laplace},
\begin{align}
B_{m}B_{m}^*-B_{m+ 1}^*B_{m+ 1} =2(\Delta_{m+ 1}'-\Delta_{m+ 1}''). \label{eq:ladder-as-Laplace-difference}
\end{align}

It remains to evaluate \(\Theta_{m+ 1}\).  We use the curvature convention
\begin{align}
d\omega=-K_g\,dA_g. \nonumber
\end{align}
By the local formula for \(\nabla^{(m+ 1)}\),
\begin{align}
\Theta_{m+ 1}=(\nabla^{(m+ 1)})^2=i(m+ 1)\,d\omega=-i(m+ 1) K_g\omega_\Sigma, \qquad i\Theta_{m+ 1}=(m+ 1) K_g\omega_\Sigma. \nonumber
\end{align}
Therefore
\begin{align}
\Lambda_{\omega_\Sigma}(i\Theta_{m+ 1})=(m+ 1) K_g. \nonumber
\end{align}
By \eqref{eq:one-dimensional-BK-statement},
\begin{align}
\Delta_{m+ 1}'-\Delta_{m+ 1}''=(m+ 1) K_g. \nonumber
\end{align}
Substituting this into \eqref{eq:ladder-as-Laplace-difference} gives
\begin{align}
B_{m}B_{m}^*-B_{m+ 1}^*B_{m+ 1}=2(m+ 1) K_g \qquad\text{on }C^\infty(L_{m+ 1}). \nonumber
\end{align}
\end{proof}

We now apply the abstract mechanism of Section~\ref{sec:abstract-ladder-counting}. For the present Dolbeault realization, take
\begin{align}
H_m=L^2(L_m),\qquad B_m=\sqrt2\,\mathcal I_m\bar\partial_m, \qquad \alpha_m=m(m+1). \nonumber
\end{align}
The associated forms and self-adjoint operators are precisely the forms
\(\mathfrak a_m,\mathfrak c_m\) and operators \(A_m,C_m\) introduced in
Section~\ref{sec:comparison-rigidity}.  Ellipticity on the compact smooth
surface gives compact resolvent.  Lemma~\ref{lem:CP1-anticanonical-cohomology}
gives
\begin{align}
\dim_{\C}\ker B_m=2m+1, \nonumber
\end{align}
and the Bochner--Kodaira identity, in the form of
Lemma~\ref{lem:complex-curvature-ladder}, gives, for every
\(t\in H^1(L_{m+1})\),
\begin{align}
\mathfrak c_m(t,t)-\mathfrak a_{m+1}(t,t) =2(m+1)\int_{S^2}(K_g-1)|t|^2\,dA_g\geq0. \nonumber
\end{align}
Thus the hypotheses of Corollary~\ref{cor:abstract-two-sphere-ladder} hold
on the concrete spaces \(H_m=L^2(L_m)\) for the concrete closed operators
\(B_m\).  That corollary gives the counting and eigenvalue comparison in
Theorem~\ref{thm:cluster-bottom-comparison}.  The equality cases are then
handled by the strict form comparison and by
Proposition~\ref{prop:abstract-ladder-saturation}, exactly as in the proof of
Theorem~\ref{thm:rigidity-main}.  

\part{Alexandrov \(2\)-Spheres with Curvature \(\geq1\)}

Throughout this part, \(X\) is a compact Alexandrov surface homeomorphic to
\(S^2\) and with Alexandrov curvature bounded below by \(1\).   

\section{Preliminary results on Alexandrov \(2\)-spheres}\label{sec:alexandrov-preliminaries}

This section collects the external Alexandrov-surface, conformal-metric, and spectral-convergence results used later.  

From \cite[Proposition A.13]{Richard2012}, we know $X$ is a
surface of bounded integral curvature. 

Note 		\[
		\operatorname{Pot}(\mathbb{CP}^1,g_0)
		:=
		\left\{
		w\in L^1(\mathbb{CP}^1,dA_0)
		:
		\Delta_0 w\in\mathcal{M}(\mathbb{CP}^1)
		\right\},
		\]
where \(\mathcal{M}(\mathbb{CP}^1)\) denotes the space of finite signed Radon measures on \(\mathbb{CP}^1\). The condition $
		\Delta_0 w\in\mathcal{M}(\mathbb{CP}^1)$	is understood in the distributional sense: there exists a finite signed 	Radon measure \(\mu_w\) such that
		\[
		\int_{\mathbb{CP}^1}
		w\,\Delta_0\varphi\,dA_0
		=
		\int_{\mathbb{CP}^1}
		\varphi\,d\mu_w
		\qquad
		\text{for every }
		\varphi\in C^\infty(\mathbb{CP}^1).
		\]
where \(\Delta_0\) is the Laplacian of unit round sphere $\mathbb{S}^2$.

Then by Troyanov~\cite[Theorem~7.3(b)]{Troyanov22}, there exist a smooth metric \(h\) on \(X\) and		\(v\in\operatorname{Pot}(X,h)\) such that
$d_X=d_{h,v}$;	where
\[
d_{h,v}(x,y)
:=
\inf_{\gamma}
\int_0^1 e^{v(\gamma(t))}|\dot\gamma(t)|_h\,dt,
\]
and the infimum is over absolutely continuous curves
$\gamma:[0,1]\to X$ joining $x$ to $y$.
	
	\begin{notation}[Fixed conformal parametrization of $X$]
Fix a conformal homeomorphism
\[
F:\CP^1\longrightarrow X.
\]
If $h$ and $v$ are the smooth metric and potential on $X$ given above, write
\[
F^*h=e^{2\psi}g_0,
\qquad
\widetilde v:=v\circ F,
\qquad
u:=\widetilde v+\psi=v\circ F+\psi.
\]
Define the pulled-back distance
\[
d_X^F(p,q):=d_X(F(p),F(q)).
\]
Then
\[
d_X^F=d_{g_0,u}.
\]
After this declaration only, one may identify $X$ with $\CP^1$ and
suppress the superscript $F$.
\end{notation}	
		
Thus, in the sense of singular conformal metrics,
\begin{align}
g_X=e^{2u}g_0,\qquad dA_X=e^{2u}\,dA_0. \label{eq:alex-conformal-model}
\end{align}



Let \(\omega_X\) denote the Alexandrov curvature measure.  In the fixed round
conformal model,
\begin{align}
\omega_X:=dA_0-\Delta_0u. \nonumber 
\end{align}
in the distributional sense.  In a local coordinate in which
\(g_X=e^{2U}|dz|^2\), the same formula is
\begin{align}
\omega_X=-\Delta_{\mathrm{euc}}U \nonumber
\end{align}
as a Radon measure (see \cite[Theorem~7.1.1, \S7.3,
pp.~112--119]{Reshetnyak1993GeometryIV}).

Recall \cite[Theorem~2.0]{Machigashira1998Gaussian} identifies the synthetic lower bound with
the measure inequality
\begin{align}
\omega_X\geq dA_X=e^{2u}dA_0. \label{eq:cp-curv-lower-measure}
\end{align}

\begin{lemma}\label{lem:u-has-lower-bound}
{$u\geq -C$.
}
\end{lemma}

\pf
{Let \(G\) be the Green kernel of \(-\Delta_0\), normalized to have zero
mean.  Since
\(-\Delta_0u=\omega_X-dA_0\), the Green representation is
\begin{align}
u(x)=\overline u+ \int_{\CP^1}G(x,y)\,d\omega_X(y) - \int_{\CP^1}G(x,y)\,dA_0(y). \nonumber
\end{align}
where $\overline u:=\frac{1}{\operatorname{Area}_{g_0}(\CP^1)}
\int_{\CP^1}u\,dA_0.$ 

Since the canonical Reshetnyak representative satisfies
\(u\in L^1(\CP^1,dA_0)\), one has \(\overline u\in\mathbb R\).

Since the Green kernel of \(-\Delta_0\) has a positive logarithmic
singularity, there is \(C_G<\infty\) such that
\[
G(x,y)\geq-C_G
\qquad (x\neq y).
\]

As \(\omega_X\) is a positive Radon measure and
\(\omega_X(\CP^1)=4\pi\) from \eqref{eq:alex-gauss-bonnet}, it follows that
\[
\int_{\CP^1}G(x,y)\,d\omega_X(y)\geq-4\pi C_G.
\]
Moreover,
\[
x\longmapsto\int_{\CP^1}G(x,y)\,dA_0(y)
\]
is bounded; in fact, it vanishes for the zero-mean normalization.
Consequently,
\[
u(x)\geq \overline u-4\pi C_G-C_0,
\]
and hence \(u\geq-C\).
}
\qed

\begin{notation}\label{notation:nuX}
{We set
\begin{align}
\nu_X:=\omega_X-dA_X\geq0. \nonumber
\end{align}
}
\end{notation}

For every \(p\in X\),
\begin{align}
\omega_X(\{p\})=2\pi-L(\Sigma_pX), \label{eq:alex-angle-defect}
\end{align}
where \(L(\Sigma_pX)\) is the length of the space of directions (see \cite[Corollary~2.2]{Machigashira1998Gaussian}).  Hence every atom has the form
\(2\pi\gamma\) with \(0<\gamma<1\).  The Gauss--Bonnet identity 
\begin{align}
\omega_X(X)=4\pi, \label{eq:alex-gauss-bonnet}
\end{align}
(see \cite[p.~861]{Machigashira1998Gaussian}).

\begin{notation}\label{not:alexandrov-sphere-data}
We use the fixed conformal model \eqref{eq:alex-conformal-model}.  

For \(m\in\mathbb Z_{\geq0}\), set
\begin{align}
L_m:=K_{\CP^1}^{-m}=(T^{1,0}\CP^1)^{\otimes m},
\qquad
L_0:=\CP^1\times\C.
\label{eq:alex-Lm-definition}
\end{align}
All these are smooth holomorphic line bundles over the fixed conformal
model \(\CP^1\); only their Hermitian metrics will depend on the metric
under consideration.

Let \(h_{m,0}\) be the Hermitian metric on \(L_m\) induced by the unit
round metric \(g_0\).  Explicitly, if \(e_1,e_2\) is a local oriented
\(g_0\)-orthonormal frame and
\begin{align}
Z_0:=\frac{1}{\sqrt2}(e_1-i e_2),
\nonumber
\end{align}
then \(Z_0\) is a unitary local frame of \(T^{1,0}\CP^1\), and
\begin{align}
h_{m,0}\bigl(Z_0^{\otimes m},Z_0^{\otimes m}\bigr)=1,
\qquad
h_{0,0}\equiv1.
\label{eq:alex-hm0-definition}
\end{align}

Whenever \(Y\) denotes \(X\), the round metric \(g_0\), or one of the
smooth conformal metrics used below, write
\begin{align}
g_Y=e^{2u_Y}g_0,
\qquad
dA_Y=e^{2u_Y}dA_0,
\nonumber
\end{align}
where \(u_X:=u\), and where the index \(Y=0\) means \(u_0:=0\).
The metric induced by \(g_Y\) on \(L_m\) is
\begin{align}
h_{m,Y}:=e^{2m u_Y}h_{m,0}
\label{eq:alex-hmY-definition}
\end{align}
almost everywhere on \(\CP^1\).  In the singular case \(Y=X\), this is
a measurable Hermitian metric.  Its values, and the value of \(u\), on
a \(dA_0\)-null set may be chosen arbitrarily.

For the round reference metric \(g_0\), equivalently for \(Y=0\), this gives
\begin{align}
H_m^0
:=
L_0^2(L_m)
=
L^2(\CP^1,L_m;h_{m,0},dA_0).
\label{eq:alex-Hm0-definition}
\end{align}
Thus
\begin{align}
\langle s,t\rangle_{H_m^0}
&=
\int_{\CP^1}h_{m,0}(s,t)\,dA_0, \quad \quad 
\|s\|_{H_m^0}^2= \int_{\CP^1}|s|_{h_{m,0}}^2\,dA_0.
\nonumber
\end{align}

Let \(\Gamma_{\mathrm{meas}}(L_m)\) denote the measurable sections of
\(L_m\), and write \(s\sim_Y t\) when \(s=t\) \(dA_Y\)-almost
everywhere.  We define
\begin{align}
L_Y^2(L_m)
:={}&
\left\{
s\in\Gamma_{\mathrm{meas}}(L_m):
\int_{\CP^1}|s|_{h_{m,Y}}^2\,dA_Y<\infty
\right\}\big/\!\sim_Y .
\label{eq:alex-L2Y-definition}
\end{align}
Its Hermitian inner product and norm are
\begin{align}
\langle s,t\rangle_{L_Y^2(L_m)}
&:=
\int_{\CP^1}h_{m,Y}(s,t)\,dA_Y
\nonumber\\
&=
\int_{\CP^1}
h_{m,0}(s,t)e^{2(m+1)u_Y}\,dA_0,
\label{eq:alex-L2Y-inner-product}\\
\|s\|_{L_Y^2(L_m)}^2
&=
\int_{\CP^1}
|s|_{h_{m,0}}^2e^{2(m+1)u_Y}\,dA_0.
\nonumber
\end{align}
Thus \(L_Y^2(L_m)\) is the weighted \(L^2\)-space of sections of the
fixed smooth bundle \(L_m\), with fiber metric \(h_{m,Y}\) and base
measure \(dA_Y\).  We set
\begin{align}
H_m^Y:=L_Y^2(L_m).
\label{eq:alex-HmY-definition}
\end{align}
In particular,
\begin{align}
H_0^X=L^2(\CP^1,dA_X),
\nonumber
\end{align}
which is identified with \(L^2(X,dA_X)\) through the fixed conformal
model.

If \(z=x+iy\) is a holomorphic coordinate,
\[
g_0=e^{2\psi}|dz|^2,
\qquad
U_Y:=\psi+u_Y,
\]
and
\[
s=f(\partial_z)^{\otimes m},
\]
then the exact local density formula is
\begin{align}
|s|_{h_{m,Y}}^2\,dA_Y
=
2^{-m}|f|^2e^{2(m+1)U_Y}\,dx\,dy.
\label{eq:alex-local-L2-density}
\end{align}

We next define the metric identification.  With \(e_1,e_2\) as above
and dual coframe \(\theta^1,\theta^2\), put
\begin{align}
\zeta_0
:=
\frac{1}{\sqrt2}(\theta^1+i\theta^2),
\qquad
\bar\zeta_0
:=
\frac{1}{\sqrt2}(\theta^1-i\theta^2).
\nonumber
\end{align}
Then \(Z_0\) is a unitary frame of \(T^{1,0}\CP^1\), while
\(\bar\zeta_0\) is a unitary frame of
\(\Lambda^{0,1}T^*\CP^1\).  Define
\begin{align}
\mathcal I_{m,0}:
\Lambda^{0,1}T^*\CP^1\otimes L_m
&\longrightarrow L_{m+1},
\nonumber\\
\mathcal I_{m,0}
\bigl(\bar\zeta_0\otimes Z_0^{\otimes m}\bigr)
&:=
Z_0^{\otimes(m+1)}.
\label{eq:alex-Im0-definition}
\end{align}
This definition does not depend on the chosen oriented
\(g_0\)-orthonormal frame.  The map \(\mathcal I_{m,0}\) is fiberwise
complex-linear and unitary.  In particular,
\begin{align}
\mathcal I_{m,0}^*=\mathcal I_{m,0}^{-1}.
\nonumber
\end{align}
Equivalently, in the coordinate above,
\begin{align}
\mathcal I_{m,0}
\bigl(d\bar z\otimes(\partial_z)^{\otimes m}\bigr)
=
2e^{-2\psi}(\partial_z)^{\otimes(m+1)}.
\label{eq:alex-Im0-coordinate}
\end{align}

Let \(h_Y^{0,1}\) be the Hermitian metric on
\(\Lambda^{0,1}T^*\CP^1\) induced by \(g_Y\).  Since
\begin{align}
h_Y^{0,1}=e^{-2u_Y}h_0^{0,1},
\nonumber
\end{align}
the corresponding measurable fiber map is
\begin{align}
\mathcal I_{m,Y}
:=
e^{-2u_Y}\mathcal I_{m,0}:
\Lambda^{0,1}T^*\CP^1\otimes L_m
\longrightarrow L_{m+1}
\label{eq:alex-ImY-definition}
\end{align}
almost everywhere on \(\CP^1\).  It is unitary from
\[
\bigl(
\Lambda^{0,1}T^*\CP^1\otimes L_m,\,
h_Y^{0,1}\otimes h_{m,Y}
\bigr)
\]
to
\[
(L_{m+1},h_{m+1,Y});
\]
that is,
\begin{align}
|\mathcal I_{m,Y}\eta|_{h_{m+1,Y}}
=
|\eta|_{h_Y^{0,1}\otimes h_{m,Y}}
\qquad\text{almost everywhere}.
\label{eq:alex-ImY-unitary}
\end{align}
Indeed, the squares of both sides scale from their round values by
\(e^{2(m-1)u_Y}\).

Equivalently, for the measurable unitary frames
\begin{align}
Z_Y:=e^{-u_Y}Z_0,
\qquad
\bar\zeta_Y:=e^{u_Y}\bar\zeta_0,
\nonumber
\end{align}
one has
\begin{align}
\mathcal I_{m,Y}
\bigl(\bar\zeta_Y\otimes Z_Y^{\otimes m}\bigr)
=
Z_Y^{\otimes(m+1)}.
\label{eq:alex-ImY-frame}
\end{align}
In particular,
\begin{align}
\mathcal I_{m,X}
=
e^{-2u}\mathcal I_{m,0}
\nonumber
\end{align}
almost everywhere on \(\CP^1\).  In the coordinate above,
\begin{align}
\mathcal I_{m,X}
\bigl(d\bar z\otimes(\partial_z)^{\otimes m}\bigr)
=
2e^{-2U}(\partial_z)^{\otimes(m+1)},
\qquad
U:=u+\psi.
\label{eq:alex-ImX-coordinate}
\end{align}

Assume \(L^1_{\mathrm{loc}}(L_m)\) denotes the space of measurable
sections \(s\) of the smooth line bundle \(L_m\) such that
\[
\int_K |s|_{h_{m,0}}\,dA_0<\infty
\]
for every compact \(K\subset\CP^1\), where \(h_{m,0}\) is the
Hermitian metric induced by \(g_0\). Equivalently, in every local
holomorphic coordinate \(z\), if
\[
s=f(\partial_z)^{\otimes m},
\]
then \(f\in L^1_{\mathrm{loc}}(dx\,dy)\). This definition is
independent of the chosen smooth background Hermitian metric and
smooth positive area form.

The operator
\begin{align}
B_m^X:\operatorname{Dom}(B_m^X)\subset H_m^X
\longrightarrow H_{m+1}^X \nonumber
\end{align}
is defined in \eqref{eq:BDol-definition} as follows,
\begin{align}
\operatorname{Dom}(B_m^X)
:=
\Bigl\{s\in H_m^X\cap L^1_{\mathrm{loc}}(L_m):
\ &\bar\partial_m s
\text{ is represented by a measurable }
L_m\text{-valued }(0,1)\text{-form }\eta, \nonumber\\
&\sqrt2\,\mathcal I_{m,X}\eta\in H_{m+1}^X
\Bigr\}, \nonumber
\end{align}
and, for \(s\) in this domain,
\begin{align}
B_m^Xs:=\sqrt2\,\mathcal I_{m,X}\eta. \nonumber
\end{align}

Equivalently, if \(g_0=e^{2\psi}|dz|^2\),
\(U:=u+\psi\), and
\(s=f(\partial_z)^{\otimes m}\), then, up to fixed positive
smooth factors, the local graph norm is
\begin{align}
\int |f|^2e^{2(m+1)U}\,dx\,dy
+
\int |\partial_{\bar z}f|^2e^{2mU}\,dx\,dy.
\label{eq:alex-local-Bm-graph-norm}
\end{align}
The transformation laws under holomorphic coordinate changes show
that this definition is intrinsic.
\end{notation}

\begin{lemma}\label{lem:alex-Bm-closed-dense}
For every \(m\geq0\), the operator $B_m^X:\operatorname{Dom}(B_m^X)\subset H_m^X
\longrightarrow H_{m+1}^X$ is closed and densely defined.
\end{lemma}

\begin{proof}
Suppose that
\[
s_j\longrightarrow s\quad\text{in }H_m^X,
\qquad
B_m^Xs_j\longrightarrow t\quad\text{in }H_{m+1}^X.
\]
In a holomorphic coordinate write
\[
s_j=f_j(\partial_z)^{\otimes m},
\qquad
\partial_{\bar z}f_j=g_j
\]
distributionally.  

The conformal factor \(U\) is locally bounded
from below by the Green-kernel representation used in
Lemma~\ref{lem:u-has-lower-bound}.  Hence the weighted convergences
imply
\[
f_j\longrightarrow f,
\qquad
g_j\longrightarrow g
\]
in ordinary \(L^2_{\mathrm{loc}}\).  Passing to the limit in the
distributional identities gives $\partial_{\bar z}f=g$.

Therefore \(s\in\operatorname{Dom}(B_m^X)\) and \(B_m^Xs=t\).

For density, let
\[
S_m:=
\left\{
p\in\CP^1:
\omega_X(\{p\})\geq\frac{2\pi}{m+1}
\right\}.
\]
Since \(\omega_X(\CP^1)=4\pi\), the set \(S_m\) is finite.  On
\(\CP^1\setminus S_m\), every point has a coordinate disc \(D\)
such that
\[
\omega_X(D)<\frac{2\pi}{m+1}.
\]
The logarithmic-potential representation and Jensen's inequality
then imply the local integrability of both
\[
e^{2(m+1)U}
\qquad\text{and}\qquad
e^{2mU}.
\]
Consequently, smooth sections supported in
\(\CP^1\setminus S_m\) belong to \(\operatorname{Dom}(B_m^X)\).
Cutting off near the finitely many points of \(S_m\), followed by
ordinary local smoothing, shows that these sections are dense in
\(H_m^X\).
\end{proof}

\begin{notation}
The Hilbert adjoint of $B_m^X$ is \((B_m^X)^*\). The shifted forms are
\begin{align}
\mathfrak a_m^X(s,t) &:=\langle B_m^Xs,B_m^Xt\rangle_{H_{m+1}^X} +m(m+1)\langle s,t\rangle_{H_m^X}, &\operatorname{Dom}(\mathfrak a_m^X)&=\operatorname{Dom}(B_m^X), \label{eq:alex-am-form}\\
\mathfrak c_m^X(s,t) &:=\langle (B_m^X)^*s,(B_m^X)^*t\rangle_{H_m^X} +m(m+1)\langle s,t\rangle_{H_{m+1}^X}, &\operatorname{Dom}(\mathfrak c_m^X)&=\operatorname{Dom}((B_m^X)^*). \label{eq:alex-cm-form}
\end{align}
Their associated self-adjoint operators are denoted by
\begin{align}
A_m^X=(B_m^X)^*B_m^X+m(m+1)I, \qquad C_m^X=B_m^X(B_m^X)^*+m(m+1)I. \nonumber
\end{align}

For \(l\geq1\), define
\[
V_l^X:=\ker B_l^X\cap\operatorname{Dom}\bigl((B_{l-1}^X)^*\bigr).
\]
Thus \(V_l^X\) is the subspace of \(\ker B_l^X\) consisting of
vectors that lie in the domain of the preceding adjoint
\((B_{l-1}^X)^*\). It is the terminal-level space used in the rigidity argument.
\end{notation}

Let $p_\tau(x,y)$ be the heat kernel of $(\CP^1,g_0)$.  For a
function $f$ and a finite Radon measure $\mu$, define
\[
(P_\tau f)(x)
:=
\int_{\CP^1}p_\tau(x,y)f(y)\,dA_0(y),
\]
and
Let $\mu\in\mathcal M(\CP^1)$ and let $\tau>0$.  Define
\[
f_{\tau,\mu}(x)
:=
\int_{\CP^1}p_\tau(x,y)\,d\mu(y).
\]
Then $f_{\tau,\mu}\in C^\infty(\CP^1)$, and we define the finite signed Radon measure
\[
P_\tau\mu
:=
f_{\tau,\mu}\,dA_0.
\]

Set
\begin{align}
u_\tau:=P_\tau u, \qquad g_\tau:=e^{2u_\tau}g_0. \label{def:g_tau-by-heat}
\end{align}

\begin{lemma}\label{lem:cp-utau-domination}
There is \(C<\infty\), independent of \(0<\tau\leq1\), such that
\begin{align}
u_\tau&\geq-C, \label{eq:cp-utau-lower}\\
u_\tau&\leq u+\tau \quad\quad dA_0\text{-a.e.} \label{eq:cp-utau-upper}
\end{align}
Moreover, for every fixed \(q\geq0\),
\begin{align}
e^{qu_\tau}F&\longrightarrow e^{qu}F &&\text{strongly in }L^2(dA_0) \quad\text{if }e^{qu}F\in L^2(dA_0), \label{eq:cp-positive-weight-conv}\\
e^{-qu_\tau}\phi&\longrightarrow e^{-qu}\phi &&\text{strongly in }L^2(dA_0) \quad\text{for every smooth }\phi. \label{eq:cp-negative-weight-conv}
\end{align}
\end{lemma}

\begin{proof}
From Lemma \ref{lem:u-has-lower-bound}, the positivity of the heat kernel gives \eqref{eq:cp-utau-lower}.

The positivity of \(\omega_X=(1-\Delta_0u)dA_0\) gives
\(\Delta_0u\leq1\) distributionally.  Therefore
\begin{align}
P_\tau u-u=\int_0^\tau P_s\Delta_0u\,ds\leq\tau, \nonumber
\end{align}
which proves \eqref{eq:cp-utau-upper}.  

Heat regularization converges to \(u\) almost everywhere and in \(L^1\).  For the positive weights,
\begin{align}
e^{2qu_\tau}|F|^2\leq e^{2q}e^{2qu}|F|^2, \nonumber
\end{align}
and for the negative weights,
\begin{align}
|e^{-qu_\tau}\phi|\leq e^{qC}|\phi|. \nonumber
\end{align}

Dominated convergence proves
\eqref{eq:cp-positive-weight-conv}--
\eqref{eq:cp-negative-weight-conv}.
\end{proof}

\begin{lemma}\label{lem:cp-heat-preserves-curv}
For every \(\tau>0\), \(g_\tau\) is smooth and
\begin{align}
K_{g_\tau}dA_{g_\tau}&=P_\tau\omega_X, \quad \quad K_{g_\tau}\geq1. \label{eq:cp-Kgtau-ge-one}
\end{align}
For every sequence \(\tau_i\downarrow0\), the smooth metrics $h_i:=e^{2P_{\tau_i}u}g_0$ satisfy 
\begin{align}
\sup_{p,q\in S^2} \bigl|d_{h_i}(p,q)-d_X(p,q))\bigr|\longrightarrow0. \label{eq:alex-uniform-distance-approx}
\end{align}
\end{lemma}

\begin{proof}
\textbf{Step (1)}. The heat semigroup commutes with \(\Delta_0\) on distributions, so
\begin{align}
K_{g_\tau}dA_{g_\tau} =(1-\Delta_0u_\tau)dA_0 =P_\tau((1-\Delta_0u)dA_0) =P_\tau\omega_X. \nonumber
\end{align}
By \eqref{eq:cp-curv-lower-measure},
\(P_\tau\omega_X\geq P_\tau(e^{2u})dA_0\).  If
\(p_\tau(x,y)\) is the heat kernel, Jensen's inequality gives
\begin{align}
e^{2u_\tau(x)} \leq\int p_\tau(x,y)e^{2u(y)}\,dA_0(y). \nonumber
\end{align}
Thus \(P_\tau\omega_X\geq e^{2u_\tau}dA_0=dA_{g_\tau}\), which is
\eqref{eq:cp-Kgtau-ge-one}.

\textbf{Step (2)}. Put
\[
\omega_i:=P_{\tau_i}\omega_X,
\qquad
u_i:=P_{\tau_i}u.
\]
Then
\[
-\Delta_0u_i\,dA_0
=
\omega_i-dA_0.
\]
Thus, in the notation of Reshetnyak's convergence theorem \cite[Theorem~7.3.1]{Reshetnyak1993GeometryIV}, we may take
\[
\mu_i^+:=\omega_i=P_{\tau_i}\omega_X,
\qquad
\mu_i^-:=dA_0.
\]
These are nonnegative measures of equal total mass \(4\pi\), and
\[
\mu_i^+\rightharpoonup\omega_X,
\qquad
\mu_i^-\rightharpoonup dA_0.
\]
Moreover, every atom of \(\omega_X\) has mass strictly less than
\(2\pi\), so the no-cusp condition holds.

Set $V_i:=\int_{S^2}e^{2u_i}\,dA_0, V:=\int_{S^2}e^{2u}\,dA_0.$ Lemma~\ref{lem:cp-utau-domination}, with \(q=1\) and \(F=1\), gives
\[
e^{u_i}\longrightarrow e^u
\qquad\text{strongly in }L^2(dA_0).
\]
Consequently,
\[
\begin{aligned}
\lim_{i\rightarrow\infty}\|e^{2u_i}-e^{2u}\|_{L^1(dA_0)}
&\leq \lim_{i\rightarrow\infty}
\|e^{u_i}-e^u\|_{L^2(dA_0)}
\bigl(\|e^{u_i}\|_{L^2(dA_0)}
+\|e^u\|_{L^2(dA_0)}
\bigr)= 0.
\end{aligned}
\]

In particular \(V_i\to V>0\). Therefore all the hypotheses of
Reshetnyak's convergence theorem are satisfied.

By \eqref{eq:alex-angle-defect}, every atom of the limiting curvature measure \(\omega_X\) has mass strictly less than \(2\pi\); this is the no-cusp atom condition in Reshetnyak's convergence theorem. 

Thus all hypotheses of \cite[Theorem $7.3.1$]{Reshetnyak1993GeometryIV} are satisfied.  The conclusion is exactly
\eqref{eq:alex-uniform-distance-approx}.  
\end{proof}

\begin{lemma}\label{lem:alex-spectral-convergence}
We have 
\begin{align}
\lim_{\tau_k\rightarrow 0}\lambda_j(S^2,g_{\tau_k})= \lambda_j(X) , \quad \quad \forall j\geq 0. \label{eq:alex-spectral-convergence}
\end{align}
\end{lemma}

\begin{proof}
Lemma~\ref{lem:cp-heat-preserves-curv} gives Gromov--Hausdorff convergence. The areas converge by Lemma~\ref{lem:cp-utau-domination}, hence the convergence is non-collapse. 

Therefore \cite[Theorem~1.1]{Shioya2001} applies and gives
\eqref{eq:alex-spectral-convergence}.
\end{proof}

\section{The weak Bochner--Kodaira inequality}
\label{sec:alexandrov-weak-bk}

For \(Y=X\) or \(Y=g_\tau\), Notation~\ref{not:alexandrov-sphere-data}
gives the unitary identification
\begin{align}
J_{m,Y}:H_m^Y\longrightarrow H_m^0,
\qquad
J_{m,Y}s=e^{(m+1)u_Y}s.
\nonumber
\end{align}
Indeed, \eqref{eq:alex-L2Y-inner-product} gives
\[
\|J_{m,Y}s\|_{H_m^0}^2=\|s\|_{H_m^Y}^2.
\]

\begin{definition}
Assume \(V_i,V\) are closed subspaces a Hilbert space \(H\), we say \(V_i\) converge in the Mosco sense to $V$, if weak limits of sequences \(v_i\in V_i\) belong to \(V\), and every
\(v\in V\) is a strong limit of some \(v_i\in V_i\).
\end{definition}

Define
\begin{align}
\widehat B_m^Y:=J_{m+1,Y}B_m^YJ_{m,Y}^{-1}: H_m^0\supset\operatorname{Dom}(\widehat B_m^Y) \longrightarrow H_{m+1}^0. \nonumber
\end{align}

For \(Y=X\) or \(Y=g_\tau\), define
\begin{align}
\operatorname{graph}(\widehat B_m^Y)
:=
\left\{
\bigl(v,\widehat B_m^Yv\bigr):
v\in\operatorname{Dom}(\widehat B_m^Y)
\right\}
\subset H_m^0\oplus H_{m+1}^0. \nonumber
\end{align}
Since \(\widehat B_m^Y\) is closed, its graph is a closed linear subspace of
\(H_m^0\oplus H_{m+1}^0\).

\begin{lemma}\label{lem:cp-graph-convergence}
For every fixed \(m\geq0\) and every \(\tau_i\downarrow0\),
\begin{align}
\operatorname{graph}(\widehat B_m^{g_{\tau_i}}) \xrightarrow{\mathrm{Mosco}} \operatorname{graph}(\widehat B_m^X) \nonumber
\end{align}
in \(H_m^0\oplus H_{m+1}^0\).
\end{lemma}

\begin{proof}
\textbf{Step (1)}. Write
\begin{align}
u_i:=u_{\tau_i},
\qquad
D_m:=B_m^{g_0}:C^\infty(L_m)\longrightarrow C^\infty(L_{m+1}),
\nonumber
\end{align}
and let \(D_m^\dagger\) denote the formal adjoint of \(D_m\) with
respect to the fixed round Hermitian metrics and \(dA_0\).

For \(Y=X\) or \(Y=g_{\tau_i}\), let \(u_Y=u\) or \(u_Y=u_i\),
respectively.  We claim that
\begin{align}
(F,G)\in\operatorname{graph}(\widehat B_m^Y)
\quad\Longleftrightarrow\quad
D_m\bigl(e^{-(m+1)u_Y}F\bigr)
=
e^{-mu_Y}G
\quad\text{in }\mathcal D'(\CP^1,L_{m+1}).
\label{eq:cp-fixed-round-graph-equation}
\end{align}

For \(Y=g_{\tau_i}\), the conformal covariance of the ladder operator gives
\begin{align}
B_m^{g_{\tau_i}}
=
e^{-2u_i}D_m.
\nonumber
\end{align}
Hence, if \(F=J_{m,g_{\tau_i}}s=e^{(m+1)u_i}s\), then
\begin{align}
\widehat B_m^{g_{\tau_i}}F
&=
J_{m+1,g_{\tau_i}}
B_m^{g_{\tau_i}}
J_{m,g_{\tau_i}}^{-1}F
=
e^{(m+2)u_i}
e^{-2u_i}
D_m\bigl(e^{-(m+1)u_i}F\bigr)
=
e^{mu_i}
D_m\bigl(e^{-(m+1)u_i}F\bigr),
\nonumber
\end{align}
which is equivalent to \eqref{eq:cp-fixed-round-graph-equation}.

For \(Y=X\), the same equation is precisely the maximal
distributional definition of \(B_m^X\).  

\textbf{Step (2)}
Suppose $(F_i,G_i)\in \operatorname{graph}(\widehat B_m^{g_{\tau_i}})$ and
\begin{align}
(F_i,G_i)
\rightharpoonup
(F,G)
\quad\text{weakly in }
H_m^0\oplus H_{m+1}^0.
\label{eq:cp-graph-weak-convergence}
\end{align}
By \eqref{eq:cp-fixed-round-graph-equation},
\begin{align}
D_m\bigl(e^{-(m+1)u_i}F_i\bigr)
=
e^{-mu_i}G_i
\quad\text{in }\mathcal D'(\CP^1,L_{m+1}).
\label{eq:cp-graph-equation-i}
\end{align}

Let \(\Phi\in C^\infty(\CP^1,L_{m+1})\).  Testing
\eqref{eq:cp-graph-equation-i} against \(\Phi\) gives
\begin{align}
\left\langle
G_i,e^{-mu_i}\Phi
\right\rangle_{H_{m+1}^0}
=
\left\langle
F_i,e^{-(m+1)u_i}D_m^\dagger\Phi
\right\rangle_{H_m^0}.
\label{eq:cp-global-test-identity-i}
\end{align}

The bundle-valued version of
\eqref{eq:cp-negative-weight-conv} gives
\begin{align}
e^{-mu_i}\Phi
&\longrightarrow e^{-mu}\Phi
&&\text{strongly in }H_{m+1}^0,
\nonumber\\
e^{-(m+1)u_i}D_m^\dagger\Phi
&\longrightarrow
e^{-(m+1)u}D_m^\dagger\Phi
&&\text{strongly in }H_m^0.
\label{eq:cp-negative-bundle-convergence}
\end{align}
Indeed, this follows directly from the pointwise bundle norm, the
uniform lower bound \(u_i\geq-C\), and dominated convergence.

Passing to the limit in
\eqref{eq:cp-global-test-identity-i}, using the weak convergence
\eqref{eq:cp-graph-weak-convergence}, gives
\begin{align}
\left\langle
G,e^{-mu}\Phi
\right\rangle_{H_{m+1}^0}
=
\left\langle
F,e^{-(m+1)u}D_m^\dagger\Phi
\right\rangle_{H_m^0}.
\label{eq:cp-global-test-identity-limit}
\end{align}
Since this holds for every smooth test section \(\Phi\), it says exactly that
\begin{align}
D_m\bigl(e^{-(m+1)u}F\bigr)=e^{-mu}G
\quad\text{in }\mathcal D'(\CP^1,L_{m+1}).
\nonumber
\end{align}
By \eqref{eq:cp-fixed-round-graph-equation},
\begin{align}
(F,G)\in\operatorname{graph}(\widehat B_m^X).
\nonumber
\end{align}
This proves the weak Mosco condition.

\textbf{Step (3)}.
Let
\begin{align}
(F,G)\in\operatorname{graph}(\widehat B_m^X).
\nonumber
\end{align}
Set
\begin{align}
s:=e^{-(m+1)u}F,
\qquad
h:=D_ms=e^{-mu}G.
\label{eq:cp-limit-s-h}
\end{align}
The lower bound \(u\geq-C\) implies that \(s\) and \(h\) are locally
\(L^2\), and the second identity in
\eqref{eq:cp-limit-s-h} is understood distributionally.

Define
\begin{align}
F_i:=e^{(m+1)u_i}s,
\qquad
G_i:=e^{mu_i}h.
\label{eq:cp-global-recovery-pair}
\end{align}
Then
\begin{align}
e^{-(m+1)u_i}F_i=s,
\qquad
D_ms=h=e^{-mu_i}G_i.
\nonumber
\end{align}
Therefore
\begin{align}
D_m\bigl(e^{-(m+1)u_i}F_i\bigr)
=
e^{-mu_i}G_i
\quad\text{distributionally}.
\label{eq:cp-recovery-graph-equation}
\end{align}

We next prove the strong convergence globally.  This is the global
bundle-valued form of \eqref{eq:cp-positive-weight-conv} and the
positive-weight counterpart of Step~(2).  There the lower bound
\(u_i\geq-C\) controls the negative weights multiplying smooth test
sections.  Here \eqref{eq:cp-utau-upper} controls the positive weights,
while \(F\in H_m^0\) and \(G\in H_{m+1}^0\) provide the required global
\(L^2\)-majorants.

As recalled in the proof of Lemma~\ref{lem:cp-utau-domination},
\[
u_i\longrightarrow u
\qquad dA_0\text{-almost everywhere},
\]
and \eqref{eq:cp-utau-upper} gives
\[
u_i\leq u+\tau_i
\qquad dA_0\text{-almost everywhere}.
\]

After changing representatives on a \(dA_0\)-null set if necessary,
the definitions \eqref{eq:cp-limit-s-h} and \eqref{eq:cp-global-recovery-pair} therefore give
\begin{align}
F_i
&=
e^{(m+1)(u_i-u)}F,
&
G_i
&=
e^{m(u_i-u)}G
\qquad dA_0\text{-almost everywhere}.
\nonumber
\end{align}
Since \(u_i\to u\) almost everywhere, both scalar factors converge
almost everywhere to \(1\).  Moreover, for all sufficiently large \(i\),
one has \(\tau_i\leq1\), and hence
\begin{align}
|F_i-F|_{h_{m,0}}^2
&=
\left|
e^{(m+1)(u_i-u)}-1
\right|^2
|F|_{h_{m,0}}^2
\leq
\bigl(1+e^{m+1}\bigr)^2
|F|_{h_{m,0}}^2,
\nonumber\\
|G_i-G|_{h_{m+1,0}}^2
&=
\left|
e^{m(u_i-u)}-1
\right|^2
|G|_{h_{m+1,0}}^2
\leq
\bigl(1+e^m\bigr)^2
|G|_{h_{m+1,0}}^2.
\nonumber
\end{align}
The two functions on the right are integrable because
\[
F\in H_m^0,
\qquad
G\in H_{m+1}^0.
\]

Dominated convergence in the fixed global bundle norms therefore yields
\begin{align}
\|F_i-F\|_{H_m^0}^2
&=
\int_{\CP^1}
|F_i-F|_{h_{m,0}}^2\,dA_0
\longrightarrow0,
\nonumber\\
\|G_i-G\|_{H_{m+1}^0}^2
&=
\int_{\CP^1}
|G_i-G|_{h_{m+1,0}}^2\,dA_0
\longrightarrow0.
\nonumber
\end{align}
Consequently,
\begin{align}
F_i\longrightarrow F
\quad\text{in }H_m^0,
\qquad
G_i\longrightarrow G
\quad\text{in }H_{m+1}^0.
\label{eq:cp-recovery-strong-convergence}
\end{align}

It follows from
\eqref{eq:cp-recovery-graph-equation} and
\eqref{eq:cp-fixed-round-graph-equation} that
\begin{align}
(F_i,G_i)
\in
\operatorname{graph}(\widehat B_m^{g_{\tau_i}}).
\nonumber
\end{align}
Together with \eqref{eq:cp-recovery-strong-convergence}, this proves the
strong recovery condition and therefore the Mosco convergence of the
graphs.
\end{proof}

If $\mathfrak{q}$ is a closed nonnegative form on a Hilbert space, we define
the extended quadratic functionals on \(H_0\):
\begin{align}
\mathfrak q[v]
&:=
\begin{cases}
\mathfrak{q}(v, v),
&
(v, v)\in\operatorname{Dom}(q),\\[2mm]
+\infty,
&
(v, v)\notin\operatorname{Dom}(q),
\end{cases}
\nonumber
\end{align}

\begin{definition}
Assume \(\mathfrak q_i, \mathfrak q\) are closed nonnegative forms on a Hilbert space \(H\), we say that \(\mathfrak q_i\) Mosco converge to \(\mathfrak q\) if:
\begin{enumerate}
\item whenever \(v_i\rightharpoonup v\),
$\displaystyle \mathfrak q[v]\leq\varliminf_i\mathfrak q_i[v_i]$;
\item for every \((v, v)\in\operatorname{Dom}(\mathfrak q)\), there are
\(v_i\to v\) strongly with
\(\mathfrak q_i[v_i]\to\mathfrak q[v]\).
\end{enumerate}
\end{definition}

\begin{lemma}\label{lem:cp-graph-to-form}
Let \(H_0\) and \(H_1\) be Hilbert spaces, and let
\begin{align}
T_i:\operatorname{Dom}(T_i)\subset H_0\longrightarrow H_1,
\qquad
T:\operatorname{Dom}(T)\subset H_0\longrightarrow H_1
\nonumber
\end{align}
be closed densely defined linear operators.  Assume that
\begin{align}
\operatorname{graph}(T_i)
\xrightarrow{\mathrm{Mosco}}
\operatorname{graph}(T)
\nonumber
\end{align}
as closed subspaces of \(H_0\oplus H_1\).

Fix \(\alpha\geq0\).  Define nonnegative forms on \(H_0\) by
\begin{align}
\mathfrak q_i^\alpha(v,w)
&:=
\langle T_i v,T_i w\rangle_{H_1}
+
\alpha\langle v,w\rangle_{H_0},
\qquad
\operatorname{Dom}(\mathfrak q_i^\alpha)
=
\operatorname{Dom}(T_i)\times \operatorname{Dom}(T_i),
\nonumber\\
\mathfrak q^\alpha(v,w)
&:=
\langle Tv,Tw\rangle_{H_1}
+
\alpha\langle v,w\rangle_{H_0},
\qquad
\operatorname{Dom}(\mathfrak q^\alpha)
=
\operatorname{Dom}(T)\times \operatorname{Dom}(T).
\nonumber
\end{align}
Then
\begin{align}
\mathfrak q_i^\alpha
\xrightarrow{\mathrm{Mosco}}
\mathfrak q^\alpha
\qquad\text{as closed nonnegative forms on }H_0.
\nonumber
\end{align}
\end{lemma}

\begin{proof}
\textbf{Step (1)}. We first verify that the forms appearing in the statement are indeed
densely defined closed nonnegative forms on \(H_0\).

Since \(T_i\) is densely defined,
\(\operatorname{Dom}(\mathfrak q_i^\alpha)
=\operatorname{Dom}(T_i)\) is dense in \(H_0\).  Since \(T_i\) is closed,
its domain is complete with respect to the graph norm
\begin{align}
\|v\|_{\operatorname{graph}(T_i)}^2
:=
\|v\|_{H_0}^2+\|T_i v\|_{H_1}^2.
\nonumber
\end{align}
The form norm associated with \(\mathfrak q_i^\alpha\) is
\begin{align}
\|v\|_{\mathfrak q_i^\alpha}^2
&:=
\|v\|_{H_0}^2+\mathfrak q_i^\alpha[v]
=
(\alpha+1)\|v\|_{H_0}^2+\|T_i v\|_{H_1}^2.
\nonumber
\end{align}
For \(v\in\operatorname{Dom}(T_i)\),
\begin{align}
\|v\|_{\operatorname{graph}(T_i)}^2
\leq
\|v\|_{\mathfrak q_i^\alpha}^2
\leq
(\alpha+1)
\|v\|_{\operatorname{graph}(T_i)}^2.
\nonumber
\end{align}
Thus the form norm is equivalent to the graph norm.  Consequently
\(\operatorname{Dom}(\mathfrak q_i^\alpha)\) is complete in the form
norm, so \(\mathfrak q_i^\alpha\) is closed.  The same argument applies
to \(T\) and \(\mathfrak q^\alpha\).

\textbf{Step (2)}. We now prove the two Mosco conditions.

Let
\begin{align}
v_i\rightharpoonup v
\qquad\text{weakly in }H_0.
\nonumber
\end{align}
We must prove
\begin{align}
\mathfrak q^\alpha[v]
\leq
\varliminf_{i\to\infty}\mathfrak q_i^\alpha[v_i].
\label{eq:cp-graph-form-liminf}
\end{align}
Put
\begin{align}
L:=
\varliminf_{i\to\infty}\mathfrak q_i^\alpha[v_i].
\nonumber
\end{align}
If \(L=+\infty\), then
\eqref{eq:cp-graph-form-liminf} is automatic.  Suppose therefore that
\(L<+\infty\).  Choose a subsequence, denoted by \(i_k\), such that
\begin{align}
\mathfrak q_{i_k}^\alpha[v_{i_k}]
\longrightarrow L.
\label{eq:cp-graph-form-subsequence}
\end{align}
After discarding finitely many terms, all the quantities in
\eqref{eq:cp-graph-form-subsequence} are finite, and hence
\begin{align}
v_{i_k}\in\operatorname{Dom}(T_{i_k}).
\nonumber
\end{align}
Moreover,
\begin{align}
\|T_{i_k}v_{i_k}\|_{H_1}^2
\leq
\mathfrak q_{i_k}^\alpha[v_{i_k}],
\nonumber
\end{align}
so the sequence \(T_{i_k}v_{i_k}\) is bounded in \(H_1\).  By weak
compactness of bounded subsets of a Hilbert space, after passing to a
further subsequence, which we do not relabel, there is \(w\in H_1\)
such that
\begin{align}
T_{i_k}v_{i_k}\rightharpoonup w
\qquad\text{weakly in }H_1.
\nonumber
\end{align}
Therefore
\begin{align}
\bigl(v_{i_k},T_{i_k}v_{i_k}\bigr)
\rightharpoonup
(v,w)
\qquad\text{weakly in }H_0\oplus H_1.
\nonumber
\end{align}
For every \(k\),
\begin{align}
\bigl(v_{i_k},T_{i_k}v_{i_k}\bigr)
\in\operatorname{graph}(T_{i_k}).
\nonumber
\end{align}

The weak-limit condition in the Mosco convergence of the graphs therefore
implies $(v,w)\in\operatorname{graph}(T)$. Thus
\begin{align}
v\in\operatorname{Dom}(T),
\qquad
w=Tv.
\label{eq:cp-graph-form-identification}
\end{align}

We have weak convergence
\begin{align}
\bigl(T_{i_k}v_{i_k},\sqrt{\alpha}\,v_{i_k}\bigr)
\rightharpoonup
\bigl(Tv,\sqrt{\alpha}\,v\bigr)
\qquad\text{in }H_1\oplus H_0.
\nonumber
\end{align}
Weak lower semicontinuity of the squared norm in \(H_1\oplus H_0\)
now gives
\begin{align}
\mathfrak q^\alpha[v]
&=
\|Tv\|_{H_1}^2+\alpha\|v\|_{H_0}^2
=
\left\|
\bigl(Tv,\sqrt{\alpha}\,v\bigr)
\right\|_{H_1\oplus H_0}^2
\leq
\varliminf_{k\to\infty}
\left\|
\bigl(T_{i_k}v_{i_k},\sqrt{\alpha}\,v_{i_k}\bigr)
\right\|_{H_1\oplus H_0}^2
\nonumber\\
&=
\varliminf_{k\to\infty}
\mathfrak q_{i_k}^\alpha[v_{i_k}]
=
L.
\nonumber
\end{align}
By the definition of \(L\), this is precisely
\eqref{eq:cp-graph-form-liminf}.

\textbf{Step (3)}.
Let
\begin{align}
v\in\operatorname{Dom}(\mathfrak q^\alpha)
=
\operatorname{Dom}(T).
\nonumber
\end{align}
Then
\begin{align}
(v,Tv)\in\operatorname{graph}(T).
\nonumber
\end{align}
The strong recovery condition in the Mosco convergence of the graphs
provides pairs
\begin{align}
(v_i,w_i)\in\operatorname{graph}(T_i)
\nonumber
\end{align}
such that
\begin{align}
(v_i,w_i)\longrightarrow(v,Tv)
\qquad\text{strongly in }H_0\oplus H_1.
\label{eq:cp-graph-form-strong-graph-recovery}
\end{align}
Since \((v_i,w_i)\in\operatorname{graph}(T_i)\), we have
\begin{align}
v_i\in\operatorname{Dom}(T_i),
\qquad
w_i=T_i v_i.
\nonumber
\end{align}
The strong convergence in
\eqref{eq:cp-graph-form-strong-graph-recovery} therefore means
\begin{align}
v_i\longrightarrow v
\quad\text{strongly in }H_0,
\qquad
T_i v_i\longrightarrow Tv
\quad\text{strongly in }H_1.
\nonumber
\end{align}

Consequently,
\begin{align}
\mathfrak q_i^\alpha[v_i]
&=
\|T_i v_i\|_{H_1}^2+\alpha\|v_i\|_{H_0}^2
\longrightarrow
\|Tv\|_{H_1}^2+\alpha\|v\|_{H_0}^2
=
\mathfrak q^\alpha[v].
\nonumber
\end{align}
Thus \(v_i\) is a recovery sequence for \(v\).  This proves the strong
recovery condition and completes the proof of Mosco convergence.
\end{proof}

\begin{lemma}\label{lem:cp-adjoint-graphs}
Under the hypotheses of Lemma~\ref{lem:cp-graph-to-form},
\begin{align}
\operatorname{graph}(T_i^*)\xrightarrow{\mathrm{Mosco}} \operatorname{graph}(T^*) \nonumber
\end{align}
in \(H_1\oplus H_0\).
\end{lemma}

\begin{proof}
For a closed densely defined \(T\),
\begin{align}
\operatorname{graph}(T^*) =\mathcal U(\operatorname{graph}(T)^\perp), \qquad \mathcal U(a,b)=(b,-a). \nonumber
\end{align}
Mosco convergence of closed subspaces is equivalent to strong convergence of
their orthogonal projections, and is therefore preserved by orthogonal
complements and by the fixed unitary \(\mathcal U\).
\end{proof}

For \(Y=X\) or \(Y=g_\tau\), let \(\mathfrak a_m^Y\) and
\(\mathfrak c_m^Y\) be the forms in
\eqref{eq:alex-am-form}--\eqref{eq:alex-cm-form}, with \(X\) replaced by
\(Y\).

\begin{prop}\label{prop:cp-mosco-ladder-forms}
For every \(m\geq0\), as \(\tau_i\downarrow0\),
\begin{align}
\mathfrak a_{m+1}^{g_{\tau_i}} \xrightarrow{\mathrm{Mosco}}\mathfrak a_{m+1}^X, \qquad \mathfrak c_m^{g_{\tau_i}} \xrightarrow{\mathrm{Mosco}}\mathfrak c_m^X, \nonumber
\end{align}
after the unitary identifications \(J_{m+ 1,Y}\).
\end{prop}

\begin{proof}
Apply Lemmas~\ref{lem:cp-graph-convergence} and
\ref{lem:cp-graph-to-form} to \(B_{m+1}\), with shift
\((m+1)(m+2)\), to obtain the first convergence.  Apply
Lemmas~\ref{lem:cp-graph-convergence}, \ref{lem:cp-adjoint-graphs}, and
\ref{lem:cp-graph-to-form} to \(B_m^*\), with shift \(m(m+1)\), to obtain
the second.  Unitarity gives
\begin{align}
(\widehat B_m^Y)^*=J_{m,Y}(B_m^Y)^*J_{m+1,Y}^{-1}, \nonumber
\end{align}
so the pulled-back adjoint forms are precisely \(\mathfrak c_m^Y\).
\end{proof}

\begin{prop}\label{prop:weak-Bochner-Kodaira}
For every \(m\geq0\),
\begin{align}
\operatorname{Dom}(\mathfrak c_m^X) &\subset\operatorname{Dom}(\mathfrak a_{m+1}^X), \quad 
\mathfrak c_m^X(t,t)\geq\mathfrak a_{m+1}^X(t,t), \qquad \forall t\in\operatorname{Dom}(\mathfrak c_m^X)). \nonumber
\end{align}
\end{prop}

\begin{proof}
For each smooth \(g_{\tau_i}\), Lemma~\ref{lem:complex-curvature-ladder}
gives
\begin{align}
\mathfrak c_m^{g_{\tau_i}}(t,t)- \mathfrak a_{m+1}^{g_{\tau_i}}(t,t) =2(m+1)\int(K_{g_{\tau_i}}-1)|t|^2\,dA_{g_{\tau_i}}\geq0 \label{ineq-am+1-cm}
\end{align}
by Lemma~\ref{lem:cp-heat-preserves-curv}.

From Proposition~\ref{prop:cp-mosco-ladder-forms}, for \(v\in\operatorname{Dom}(\mathfrak c_m^X)\), we can find a sequence $v_i\in \mathfrak c_m^{g_{\tau_i}}$ converging to $v$ in $H_{m+ 1}^0$ with $\displaystyle \lim_i\mathfrak c_m^{g_{\tau_i}}(v_i, v_i)=\mathfrak c_m^X(v, v)$. 

From (\ref{ineq-am+1-cm}) and Proposition~\ref{prop:cp-mosco-ladder-forms}, we get
\begin{align}
\mathfrak a_{m+ 1}^X(v, v) \leq\varliminf_i\mathfrak a_{m+1}^{g_{\tau_i}}(v_i, v_i) \leq\lim_i\mathfrak c_m^{g_{\tau_i}}(v_i, v_i)=\mathfrak c_m^X(v, v). \nonumber
\end{align}
\end{proof}

\begin{lemma}\label{lem:atomic-kernel-holomorphic}
For every \(m\geq0\),
\begin{align}
\ker B_m^X\subset H^0(\CP^1,K_{\CP^1}^{-m}), \qquad \dim_{\C}\ker B_m^X\leq2m+1. \nonumber
\end{align}
\end{lemma}

\begin{proof}
Fix \(m\geq0\), and let
\begin{align}
s\in\ker B_m^X.
\nonumber
\end{align}
Thus
\begin{align}
s\in\operatorname{Dom}(B_m^X)\subset H_m^X,
\qquad
B_m^Xs=0.
\nonumber
\end{align}

We first compare the singular \(H_m^X\)-norm with the fixed smooth round
norm.  Since
\begin{align}
g_X=e^{2u}g_0,
\nonumber
\end{align}
the Hermitian metrics induced on
\begin{align}
L_m=K_{\CP^1}^{-m}
=(T^{1,0}\CP^1)^{\otimes m}
\nonumber
\end{align}
satisfy
\begin{align}
h_{m,X}=e^{2mu}h_{m,0},
\qquad
dA_X=e^{2u}dA_0.
\nonumber
\end{align}
Consequently,
\begin{align}
\|s\|_{H_m^X}^2
&=
\int_{\CP^1}|s|_{h_{m,X}}^2\,dA_X
\nonumber\\
&=
\int_{\CP^1}
|s|_{h_{m,0}}^2e^{2(m+1)u}\,dA_0.
\label{eq:atomic-HX-H0-comparison}
\end{align}

Lemma~\ref{lem:u-has-lower-bound} gives
\begin{align}
u\geq-C
\qquad dA_0\text{-a.e. on }\CP^1
\label{eq:atomic-u-lower}
\end{align}
for some finite constant \(C\).  Therefore
\begin{align}
\|s\|_{H_m^0}^2
&=
\int_{\CP^1}|s|_{h_{m,0}}^2\,dA_0
\leq
e^{2(m+1)C}
\int_{\CP^1}
|s|_{h_{m,0}}^2e^{2(m+1)u}\,dA_0
=
e^{2(m+1)C}\|s\|_{H_m^X}^2
<\infty.
\label{eq:atomic-HX-embeds-H0}
\end{align}
In particular $s\in H_m^0.$

Now fix a holomorphic coordinate disc \(D\), write
\begin{align}
g_0=e^{2\psi}|dz|^2,
\qquad
g_X=e^{2U}|dz|^2,
\qquad
U=u+\psi,
\qquad
s=f(\partial_z)^{\otimes m}.
\nonumber
\end{align}
For the round metric,
\begin{align}
|(\partial_z)^{\otimes m}|_{h_{m,0}}^2
=
2^{-m}e^{2m\psi},
\qquad
dA_0=e^{2\psi}\,dx\,dy.
\nonumber
\end{align}
Hence, for every \(D'\Subset D\),
\begin{align}
\int_{D'}|s|_{h_{m,0}}^2\,dA_0
=
2^{-m}
\int_{D'}|f|^2e^{2(m+1)\psi}\,dx\,dy.
\label{eq:atomic-round-local-norm}
\end{align}
Because \(\psi\) is smooth, the function
\(e^{2(m+1)\psi}\) is bounded above and bounded away from zero on
\(\overline{D'}\).  From
\eqref{eq:atomic-HX-embeds-H0} and
\eqref{eq:atomic-round-local-norm}, it follows that
\begin{align}
f\in L^2(D',dx\,dy).
\nonumber
\end{align}
Since \(D'\) has finite Euclidean measure,
\begin{align}
f\in L^1(D',dx\,dy).
\nonumber
\end{align}
As \(D'\Subset D\) was arbitrary,
\begin{align}
f\in L^2_{\mathrm{loc}}(D)
\subset L^1_{\mathrm{loc}}(D).
\label{eq:atomic-f-local-integrability}
\end{align}

We next use the kernel equation.  By the maximal definition of \(B_m^X\),
the distribution \(\bar\partial_m s\) is represented by a measurable
\(L_m\)-valued \((0,1)\)-form \(\eta\), and
\begin{align}
B_m^Xs=\sqrt2\,\mathcal I_{m,X}\eta.
\nonumber
\end{align}
Since \(B_m^Xs=0\) and \(\mathcal I_{m,X}\) is a fiberwise isomorphism
almost everywhere, one has
\begin{align}
\eta=0
\qquad\text{a.e.}
\nonumber
\end{align}
Therefore
\begin{align}
\bar\partial_m s=0
\qquad\text{in }\mathcal D'(D).
\nonumber
\end{align}
Because \((\partial_z)^{\otimes m}\) is a holomorphic local frame, this
equation is exactly
\begin{align}
\partial_{\bar z}f=0
\qquad\text{in }\mathcal D'(D).
\label{eq:atomic-dbar-f-zero}
\end{align}

For completeness, we spell out the regularity argument.  In distributions,
\begin{align}
\Delta_{\mathrm{euc}}f
=
4\partial_z\partial_{\bar z}f
=
0.
\nonumber
\end{align}
Together with
\eqref{eq:atomic-f-local-integrability}, the Weyl lemma for the Euclidean
Laplacian, applied to the real and imaginary parts of \(f\), shows that
\(f\) agrees almost everywhere with a smooth harmonic function.  Equation
\eqref{eq:atomic-dbar-f-zero} then holds pointwise, and hence \(f\) is
holomorphic on \(D\).

On overlaps of holomorphic coordinate charts, the resulting holomorphic
coefficients satisfy the holomorphic transition law of
\(K_{\CP^1}^{-m}\).  Indeed, they already represent the same measurable
section almost everywhere, and two holomorphic functions which agree
almost everywhere agree everywhere.  The local representatives therefore
assemble to a global holomorphic section.  Hence
\begin{align}
\ker B_m^X
\subset
H^0(\CP^1,K_{\CP^1}^{-m}).
\nonumber
\end{align}

Lemma~\ref{lem:CP1-anticanonical-cohomology} gives
$
\dim_{\C}H^0(\CP^1,K_{\CP^1}^{-m})=2m+1,
$
therefore
$
\dim_{\C}\ker B_m^X\leq2m+1.
$
\end{proof}

\begin{lemma}\label{lem:alex-ladder-compactness}
For every \(m\geq0\), the embedding of
\(\operatorname{Dom}(B_m^X)\), equipped with its graph norm, into \(H_m^X\)
is compact.  Consequently \(A_m^X\), \(C_m^X\) have compact resolvent.
\end{lemma}

\begin{proof}
\textbf{Step (1)}. Put
\begin{align}
\alpha_r:=r(r+1), \qquad r\geq0. \nonumber
\end{align}
We prove first that the essential spectrum of every operator is empty.

Lemma \ref{lem:atomic-kernel-holomorphic} gives
\begin{align}
\ker B_r^X
&\subset H^0(\CP^1,K_{\CP^1}^{-r}),
&
\dim_{\C}\ker B_r^X
&\leq2r+1.
\label{eq:alex-compactness-finite-kernel}
\end{align}

Second,
\begin{align}
\ker(B_r^X)^*=\{0\}.
\label{eq:alex-compactness-adjoint-kernel}
\end{align}
Indeed, if $t\in\ker(B_r^X)^*$, then
$t\in\operatorname{Dom}(\mathfrak c_r^X)$.  By
Proposition~\ref{prop:weak-Bochner-Kodaira},
$t\in\operatorname{Dom}(\mathfrak a_{r+1}^X)$ and
\begin{align}
\alpha_{r+1}\|t\|_{H_{r+1}^X}^2
&\leq \mathfrak a_{r+1}^X(t,t)
\leq \mathfrak c_r^X(t,t)
=\alpha_r\|t\|_{H_{r+1}^X}^2.
\nonumber
\end{align}
Since $\alpha_{r+1}-\alpha_r=2(r+1)>0$, one has $t=0$.

We next compare the essential spectra of the two partners at one
level.  Let
\begin{align}
B_r^X=U_r|B_r^X|
\nonumber
\end{align}
be the polar decomposition.  By
\eqref{eq:alex-compactness-adjoint-kernel},
\begin{align}
\overline{\operatorname{Ran}B_r^X}
=
\bigl(\ker(B_r^X)^*\bigr)^\perp
=
H_{r+1}^X.
\nonumber
\end{align}
Hence the partial isometry $U_r$ restricts to a unitary map
\begin{align}
U_r:(\ker B_r^X)^\perp\longrightarrow H_{r+1}^X.
\nonumber
\end{align}
Set
\begin{align}
T_r
:=
(B_r^X)^*B_r^X\big|_{(\ker B_r^X)^\perp}.
\nonumber
\end{align}
The polar-decomposition functional calculus gives
\begin{align}
B_r^X(B_r^X)^*=U_rT_rU_r^*.
\nonumber
\end{align}
Moreover, $(\ker B_r^X)^\perp$ is a reducing subspace for
$(B_r^X)^*B_r^X$, and therefore
\begin{align}
A_r^X
&=
\alpha_r I_{\ker B_r^X}
\oplus(T_r+\alpha_r I),
\nonumber\\
C_r^X
&=
U_r(T_r+\alpha_r I)U_r^*.
\nonumber
\end{align}
The first summand of $A_r^X$ acts on the finite-dimensional space
$\ker B_r^X$ by \eqref{eq:alex-compactness-finite-kernel}.  A
finite-dimensional direct summand has empty essential spectrum, while
unitary equivalence preserves essential spectrum.  Hence
\begin{align}
\sigma_{\mathrm{ess}}(A_r^X)
=
\sigma_{\mathrm{ess}}(C_r^X).
\label{eq:alex-partner-essential-spectrum}
\end{align}

For a self-adjoint operator $T$ bounded from below on a Hilbert space
$H$, define
\begin{align}
\Lambda_{\mathrm{ess}}(T)
:=
\begin{cases}
\inf\sigma_{\mathrm{ess}}(T),
&\sigma_{\mathrm{ess}}(T)\neq\varnothing,\\
+\infty,
&\sigma_{\mathrm{ess}}(T)=\varnothing.
\end{cases}
\nonumber
\end{align}
If $\mathfrak q_T$ is the closed form of $T$, the max--min principle
and its form-domain version imply
\begin{align}
\Lambda_{\mathrm{ess}}(T)
=
\sup_{\substack{L\subset H\\ \dim L<\infty}}
\ 
\inf_{\substack{
v\in\operatorname{Dom}(\mathfrak q_T)\cap L^\perp\\
\|v\|_H=1}}
\mathfrak q_T(v,v).
\label{eq:alex-bottom-essential-maxmin}
\end{align}
Indeed, for each fixed value of $\dim L$, the inner variational
quantity is the corresponding max--min value; taking the supremum over
all finite dimensions gives the bottom of the essential spectrum.
See
\cite[Theorem~8.1.1 and Proposition~8.1.3,
pp.~125--127]{NonnenmacherSpectralTheory2023}.

It follows in particular that quadratic-form order is monotone at the
bottom of the essential spectrum: if $S\geq T$ in the form sense, then
\begin{align}
\Lambda_{\mathrm{ess}}(S)
\geq
\Lambda_{\mathrm{ess}}(T).
\label{eq:alex-essential-form-monotonicity}
\end{align}
For completeness, fix a finite-dimensional $L\subset H$.  The form
order means
\begin{align}
\operatorname{Dom}(\mathfrak q_S)
&\subset\operatorname{Dom}(\mathfrak q_T),
&
\mathfrak q_S(v,v)
&\geq\mathfrak q_T(v,v)
\quad
(v\in\operatorname{Dom}(\mathfrak q_S)).
\nonumber
\end{align}

Consequently,
\begin{align}
&\inf_{\substack{
v\in\operatorname{Dom}(\mathfrak q_S)\cap L^\perp\\
\|v\|=1}}
\mathfrak q_S(v,v)\geq
\inf_{\substack{
v\in\operatorname{Dom}(\mathfrak q_S)\cap L^\perp\\
\|v\|=1}}
\mathfrak q_T(v,v)
\geq
\inf_{\substack{
v\in\operatorname{Dom}(\mathfrak q_T)\cap L^\perp\\
\|v\|=1}}
\mathfrak q_T(v,v).
\nonumber
\end{align}
Taking the supremum over $L$ proves
\eqref{eq:alex-essential-form-monotonicity}.

Apply \eqref{eq:alex-essential-form-monotonicity} to the form order
\begin{align}
C_r^X\geq A_{r+1}^X
\nonumber
\end{align}
from Proposition~\ref{prop:weak-Bochner-Kodaira}, and then use
\eqref{eq:alex-partner-essential-spectrum}.  This yields
\begin{align}
\Lambda_{\mathrm{ess}}(A_r^X)
=
\Lambda_{\mathrm{ess}}(C_r^X)
\geq
\Lambda_{\mathrm{ess}}(A_{r+1}^X)
\qquad(r\geq0).
\label{eq:alex-essential-ladder-step}
\end{align}
On the other hand,
\begin{align}
\mathfrak a_N^X(s,s)
&=
\|B_N^Xs\|_{H_{N+1}^X}^2
+
\alpha_N\|s\|_{H_N^X}^2
\geq
\alpha_N\|s\|_{H_N^X}^2,
\qquad
s\in\operatorname{Dom}(\mathfrak a_N^X).
\nonumber
\end{align}
Thus $A_N^X\geq\alpha_N I$ as a self-adjoint operator, so
\[
\sigma(A_N^X)\subset[\alpha_N,+\infty)
\]
and therefore
\begin{align}
\Lambda_{\mathrm{ess}}(A_N^X)
\geq
\alpha_N
=
N(N+1).
\nonumber
\end{align}
Iterating \eqref{eq:alex-essential-ladder-step}, for every $N>m$ we
obtain
\begin{align}
\Lambda_{\mathrm{ess}}(A_m^X)
&\geq
\Lambda_{\mathrm{ess}}(A_{m+1}^X)
\geq\cdots\geq
\Lambda_{\mathrm{ess}}(A_N^X)
\nonumber\\
&\geq N(N+1).
\nonumber
\end{align}
Since this holds for every $N>m$, it follows that
\begin{align}
\Lambda_{\mathrm{ess}}(A_m^X)=+\infty,
\qquad
\sigma_{\mathrm{ess}}(A_m^X)=\varnothing.
\nonumber
\end{align}
A self-adjoint operator has empty essential spectrum if and only if it
has compact resolvent; see
\cite[Proposition~7.2.4, p.~117]
{NonnenmacherSpectralTheory2023}.
Thus $A_m^X$ has compact resolvent.  Equation
\eqref{eq:alex-partner-essential-spectrum} also gives
\begin{align}
\sigma_{\mathrm{ess}}(C_m^X)=\varnothing,
\nonumber
\end{align}
so $C_m^X$ has compact resolvent as well.

\textbf{Step (2)}. It remains to prove the compact graph embedding in the statement.  By
the representation theorem for the closed form $\mathfrak a_m^X$,
\begin{align}
\operatorname{Dom}\bigl((A_m^X+I)^{1/2}\bigr)
=
\operatorname{Dom}(\mathfrak a_m^X)
=
\operatorname{Dom}(B_m^X),
\nonumber
\end{align}
and, for $s\in\operatorname{Dom}(B_m^X)$,
\begin{align}
\bigl\|(A_m^X+I)^{1/2}s\bigr\|_{H_m^X}^2
&=
\mathfrak a_m^X(s,s)+\|s\|_{H_m^X}^2
\nonumber\\
&=
\|B_m^Xs\|_{H_{m+1}^X}^2
+
(\alpha_m+1)\|s\|_{H_m^X}^2.
\label{eq:alex-compactness-form-graph-norm}
\end{align}
The norm in \eqref{eq:alex-compactness-form-graph-norm} is equivalent
to the graph norm of $B_m^X$.  Since $A_m^X$ has compact resolvent,
$(A_m^X+I)^{-1}$ is compact.  Its positive square root
$(A_m^X+I)^{-1/2}$ is compact as well, by the spectral calculus.

Now let $(s_j)$ be bounded in the graph norm of $B_m^X$, and set
\begin{align}
f_j:=(A_m^X+I)^{1/2}s_j.
\nonumber
\end{align}
By \eqref{eq:alex-compactness-form-graph-norm}, $(f_j)$ is bounded in
$H_m^X$, and
\begin{align}
s_j=(A_m^X+I)^{-1/2}f_j.
\nonumber
\end{align}
The compactness of $(A_m^X+I)^{-1/2}$ gives a subsequence of $(s_j)$
converging in $H_m^X$.  Hence
\begin{align}
\operatorname{Dom}(B_m^X)\hookrightarrow H_m^X
\nonumber
\end{align}
is compact when $\operatorname{Dom}(B_m^X)$ is equipped with its graph
norm.  This proves all assertions.
\end{proof}

The scalar case has the expected interpretation:
\begin{align}
A_0^X=(B_0^X)^*B_0^X=-\Delta_X. \label{eq:alex-A0-laplacian}
\end{align}
Indeed, the form \(\|B_0^Xf\|^2\) is the conformally invariant Dirichlet
energy \(\int_X|\nabla f|^2dA_X\), which defines the canonical Alexandrov
Laplacian.

\section{A saturated round counting threshold excludes curvature atoms}\label{sec:alexandrov-atomic-obstruction}

\begin{definition}
A finite Radon measure \(\mu\) is atomless if
\(\mu(\{p\})=0\) for every \(p\).
\end{definition}

Since \(dA_X\) is atomless, \(\nu_X\) and \(\omega_X\) have the same atoms.

\begin{lemma}\label{lem:atomic-counting-saturation}
Let \(l\geq1\).  If
\begin{align}
\mathcal N_{-\Delta_X}\bigl(l(l+1)\bigr)=(l+1)^2, \label{eq:alex-saturation-hypothesis}
\end{align}
then
\begin{align}
V_l^X&=\ker B_l^X, &\dim_{\C}V_l^X&=2l+1, \nonumber\\
\mathfrak c_{l-1}^X(s,s) &=\mathfrak a_l^X(s,s) =l(l+1)\|s\|_{H_l^X}^2 &&\text{for every }s\in V_l^X. \nonumber
\end{align}
Moreover \(\dim_{\C}\ker B_m^X=2m+1\) for \(0\leq m\leq l\).
\end{lemma}

\begin{proof}
We apply Proposition~\ref{prop:abstract-ladder-saturation} to the concrete
singular ladder
\begin{align}
H_m=H_m^X,\qquad B_m=B_m^X, \qquad \alpha_m=m(m+1), \qquad \rho_m=2m+1. \nonumber
\end{align}
The operators are closed and densely defined by
Notation~\ref{not:alexandrov-sphere-data}; compact resolvent is
Lemma~\ref{lem:alex-ladder-compactness}; the form-domain inclusion and form
order are Proposition~\ref{prop:weak-Bochner-Kodaira}; and the kernel bounds
\(r_m\leq2m+1\) are Lemma~\ref{lem:atomic-kernel-holomorphic}.  

Since \(A_0^X=-\Delta_X\), hypothesis \eqref{eq:alex-saturation-hypothesis} is exactly \eqref{eq:abstract-ladder-saturation-hypothesis}.  

Therefore \eqref{eq:abstract-ladder-maximal-kernels}, \eqref{eq:abstract-ladder-top-space}, and \eqref{eq:abstract-ladder-top-form-equality} give all stated conclusions.
\end{proof}

\begin{lemma}\label{lem:atomic-logarithmic-mean}
Let \(p\) be an atom of \(\nu_X\), with
\(\nu_X(\{p\})=2\pi\gamma\), \(0<\gamma<1\).  In a holomorphic coordinate
\(z\) centered at \(p\), write \(g_X=e^{2U}|dz|^2\) and
\begin{align}
\overline U(r):=\frac1{2\pi}\int_0^{2\pi}U(re^{i\theta})\,d\theta. \nonumber
\end{align}
Then there is $C, r_{p, \epsilon}> 0$ such that 
\begin{align}
e^{\overline{U}(r)}\geq C\cdot (r_{p, \epsilon})^{\gamma- \epsilon}\cdot r^{-\gamma+ \epsilon}, \quad \quad \forall r\in (0, r_{p, \epsilon}). \nonumber 
\end{align}
\end{lemma}

\begin{proof}
Locally, \(-\Delta_{\mathrm{euc}}U=\omega_X\).  The Gauss--Green formula for
circular means gives, for almost every \(r\),
\begin{align}
-2\pi r\,\overline U'(r)=\omega_X(B_r(p)),\label{equ-of-U'}
\end{align}

From $\displaystyle \lim_{t\rightarrow 0}\omega_X(B_t(p))= 2\pi \gamma$, for any $\epsilon> 0$ there is $r_{p, \epsilon}> 0$ such that 
\begin{align}
\omega_X(B_t(p))\geq 2\pi(\gamma- \epsilon), \quad \quad \forall t\in (0, r_{p, \epsilon}). \label{ineq-omeag-X}
\end{align}

Now from (\ref{equ-of-U'}), for any $r\in (0, r_{p, \epsilon})$, we have
\begin{align}
\overline{U}(r_{p, \epsilon})- \overline{U}(r)= -\frac{1}{2\pi}\int_r^{r_{p,\epsilon}}\frac{\omega_X(B_t(p))}{t}dt. \label{integ-equ}
\end{align}

By (\ref{ineq-omeag-X}) and (\ref{integ-equ}), using Lemma \ref{lem:u-has-lower-bound}, we obtain
\begin{align}
\overline{U}(r)\geq \overline{U}(r_{p, \epsilon})+ \frac{1}{2\pi}\int_r^{r_{p,\epsilon}}\frac{\omega_X(B_t(p))}{t}dt\geq -C+ \ln(\frac{r_{p, \epsilon}}{r})^{\gamma- \epsilon}. \nonumber 
\end{align}
Therefore
\begin{align}
e^{\overline{U}(r)}\geq C\cdot (r_{p, \epsilon})^{\gamma- \epsilon}\cdot r^{-\gamma+ \epsilon}\geq C(p, \epsilon, \gamma)\cdot r^{-\gamma+ \epsilon}. \nonumber 
\end{align}
\end{proof}

\begin{lemma}\label{lem:atomic-adjoint-domain-vanishing}
Let \(p\in X\) be an atom of \(\nu_X\), then
\begin{align}
s(p)= 0, \quad \quad \quad \forall s\in V_l^X,\ l\geq1. \nonumber
\end{align}
\end{lemma}

\begin{proof}
First, write
\begin{align}
\nu_X(\{p\})=2\pi\gamma, \qquad 0<\gamma<1. \nonumber
\end{align}

By Lemma~\ref{lem:atomic-kernel-holomorphic}, \(s\) is a holomorphic section
of \(K_{\CP^1}^{-l}\).  Choose a local holomorphic coordinate \(z\) centered
at \(p\), write \(g_X=e^{2U}|dz|^2\), and write
\begin{align}
s=f(z)(\partial_z)^{\otimes l}, \qquad f(z)=z^k a(z), \qquad a(0)\neq0. \nonumber
\end{align}
We prove that \(k\geq l\gamma>0\).  Since \(k\) is an integer, the conclusion
follows.

The target Hilbert space of \((B_{l-1}^X)^*\) is \(H_{l-1}^X\).  In the
coordinate \(z\), its local weight is, up to a smooth positive factor,
\begin{align}
e^{2l U}\,dx\,dy. \nonumber
\end{align}
The formal adjoint coefficient is
\begin{align}
D_l f := \partial_z f+2l U_z f = e^{-2l U}\partial_z(e^{2l U}f), \label{eq:atomic-Dell-definition}
\end{align}
where the last expression is understood distributionally.  Since
\(s\in\operatorname{Dom}((B_{l-1}^X)^*)\), the distribution \(D_l f\) is
represented by an \(H_{l-1}^X\)-coefficient.

Thus, on a sufficiently small coordinate disc,
\begin{align}
\int |D_l f|^2 e^{2l U}\,dx\,dy<\infty. \label{eq:atomic-Dell-L2}
\end{align}
Moreover, the local norm condition for \(s\in H_l^X\) gives local
integrability of \(|f|^2e^{2l U}\).  Indeed, the \(H_l^X\)-weight is,
up to a fixed smooth positive factor, \(e^{2(l+1)U}\,dx\,dy\), and \(U\) is
locally bounded from below.

For \(0<r\ll1\), set
\begin{align}
I(r):= \int_0^{2\pi} |f(re^{i\theta})|^2 e^{2l U(re^{i\theta})}\,d\theta . \nonumber
\end{align}
We first derive a lower bound for \(I(r)\) from the atom.  Since
\(a(0)\neq0\), after shrinking the coordinate disc there is a constant
\(c_0>0\) such that
\begin{align}
|f(re^{i\theta})|^2\geq c_0 r^{2k}. \nonumber
\end{align}

By Jensen's inequality and Lemma~\ref{lem:atomic-logarithmic-mean}, for every
\(\varepsilon>0\) and all sufficiently small \(r\),
\begin{align}
I(r)&\geq c_0 r^{2k}\int_0^{2\pi}e^{2lU(re^{i\theta})}\,d\theta \geq c_1 r^{2k}\exp\bigl(2l\overline U(r)\bigr) \geq c_2 r^{2k-2l\gamma+2l\varepsilon}. \nonumber
\end{align}
Equivalently, with \(r=e^{-t}\),
\begin{align}
I(e^{-t})^{1/2} \geq c_3\exp\bigl((l\gamma-k-l\varepsilon)t\bigr) \qquad \text{for }t\gg1. \label{eq:atomic-J-exponential-lower}
\end{align}

Put $J(t):=I(e^{-t})^{1/2}. $ For a.e. \(r\), the distributional identity
\eqref{eq:atomic-Dell-definition} gives
\begin{align}
\frac{d}{dr}I(r) = 2\Re\int_0^{2\pi} e^{i\theta}D_l f(re^{i\theta}) \overline{f(re^{i\theta})} e^{2l U(re^{i\theta})}\,d\theta. \label{eq:atomic-I-derivative}
\end{align}

Here is the precise meaning of this identity.  Put
\(Q:=|f|^2e^{2l U}\).  Since
\eqref{eq:atomic-Dell-definition} says distributionally that
\(e^{2l U}D_l f=\partial_z(e^{2l U}f)\), and since \(f\) is
holomorphic, one has
\begin{align}
\partial_z Q=e^{2l U}D_l f\,\overline f \nonumber
\end{align}
in the sense of distributions on punctured coordinate annuli.  Hence, using
\(\partial_r=e^{i\theta}\partial_z+e^{-i\theta}\partial_{\bar z}\),
\begin{align}
\partial_r Q =2\Re\bigl(e^{i\theta}D_l f\,\overline f\,e^{2l U}\bigr) \nonumber
\end{align}
distributionally on each such annulus.  The two local \(L^2\)-bounds above imply,
by Cauchy--Schwarz, that the right-hand side is locally integrable.  Therefore
the standard slicing theorem for weak derivatives gives that \(I(r)\) is locally
absolutely continuous in \(r\) and that \eqref{eq:atomic-I-derivative} holds for
a.e. \(r\).  


By Cauchy's inequality,
\begin{align}
|I'(r)| \leq 2 \left( \int_0^{2\pi}|D_l f(re^{i\theta})|^2e^{2l U(re^{i\theta})} \,d\theta \right)^{1/2} I(r)^{1/2}. \nonumber
\end{align}
Consequently, for a.e. \(t\) with \(J(t)>0\),
\begin{align}
|J'(t)|^2\leq e^{-2t}\int_0^{2\pi}|D_l f(e^{-t}e^{i\theta})|^2e^{2lU(e^{-t}e^{i\theta})}\,d\theta. \label{eq:atomic-Jprime-bound}
\end{align}

Integrating
\eqref{eq:atomic-Jprime-bound} and using \eqref{eq:atomic-Dell-L2}, we obtain
\begin{align}
\int_T^\infty |J'(t)|^2\,dt<\infty \nonumber
\end{align}
for every sufficiently large \(T\).  Choose \(T\) such that \(J(T)<\infty\).
Then, for \(t\geq T\),
\begin{align}
J(t) \leq J(T) + \left(\int_T^t |J'(\tau)|^2\,d\tau\right)^{1/2} (t-T)^{1/2} \leq C(1+t^{1/2}). \label{eq:atomic-J-polynomial-upper}
\end{align}
Thus the adjoint-domain condition implies at most polynomial growth of
\(J(t)\) as \(t\to\infty\).

Suppose now that \(k<l\gamma\).  Choose \(\varepsilon>0\) so small that
\begin{align}
l\gamma-k-l\varepsilon>0. \nonumber
\end{align}
Then \eqref{eq:atomic-J-exponential-lower} says that \(J(t)\) grows at least
exponentially, while \eqref{eq:atomic-J-polynomial-upper} says that it grows at
most like \(t^{1/2}\).

This contradiction proves \(k\geq l\gamma>0\).  So \(s(p)=0\).
\end{proof}

\begin{prop}\label{prop:atomic-alexandrov-counting-obstruction}
If
\begin{align}
\mathcal N_{-\Delta_X}\bigl(l(l+1)\bigr)=(l+1)^2 \nonumber
\end{align}
for some \(l\geq1\), then \(\nu_X\) is atomless.
\end{prop}

\begin{proof}
By Lemma~\ref{lem:atomic-counting-saturation},
\(\dim_{\C}V_l^X=2l+1\).  If \(p\) were an atom of \(\nu_X\),
Lemma~\ref{lem:atomic-adjoint-domain-vanishing} would imply that every
section in \(V_l^X\) vanishes at \(p\).  Lemma~\ref{lem:atomic-kernel-holomorphic}
then gives
\begin{align}
V_l^X\subset \{s\in H^0(\CP^1,K_{\CP^1}^{-l}):s(p)=0\}. \nonumber
\end{align}
By the evaluation statement in
Lemma~\ref{lem:CP1-anticanonical-cohomology}, the space on the right has
complex dimension \(2l\), a contradiction.
\end{proof}

\section{The atomless Bochner--Kodaira defect identity}\label{sec:alexandrov-atomless-defect}

The weak form order of Proposition~\ref{prop:weak-Bochner-Kodaira} retains
only nonnegativity of the limiting defect.  Rigidity needs the defect measure
itself.  Under atomlessness, the missing term is recovered directly from the maximal adjoint domain.

\begin{lemma}\label{lem:no-atomic-residue-lsc}
Assume \(\omega_X\) is atomless.  Let \(l\geq1\) and
\(s\in V_l^X\).  Set \(q:= |s|_{h_{l,X}}^2\), then
\begin{align}
\|(B_{l-1}^X)^*s\|_{H_{l-1}^X}^2 =2l\int_X q\,d\omega_X. \label{eq:no-atomic-residue-identity}
\end{align}
In particular, the integral on the right is finite.
\end{lemma}

\begin{proof}
\textbf{Step (1)}. By Lemma~\ref{lem:atomic-kernel-holomorphic}, \(s\) is holomorphic.  In a
holomorphic coordinate with \(g_X=e^{2U}|dz|^2\), write
\begin{align}
s=f(z)(\partial_z)^{\otimes l}, \qquad q=2^{-l}|f|^2e^{2lU}. \nonumber
\end{align}

We use the conventions
\begin{align}
\partial_z
:=
\frac12(\partial_x-i\partial_y),
\qquad
\partial_{\bar z}
:=
\frac12(\partial_x+i\partial_y).
\nonumber
\end{align}

For $g_X=e^{2U}|dz|^2$, the oriented orthonormal frame and the associated unitary frames are
\begin{align}
e_1&=e^{-U}\partial_x,
&
e_2&=e^{-U}\partial_y,
\nonumber\\
Z_X
&:=
\frac1{\sqrt2}(e_1-i e_2)
=
\sqrt2\,e^{-U}\partial_z,
&
\bar\zeta_X
&:=
\frac1{\sqrt2}(\theta^1-i\theta^2)
=
\frac1{\sqrt2}e^U d\bar z.
\label{eq:adjoint-q-unitary-frames}
\end{align}
Consequently,
\begin{align}
\partial_z
=
2^{-1/2}e^UZ_X,
\qquad
d\bar z
=
\sqrt2\,e^{-U}\bar\zeta_X.
\nonumber
\end{align}
For every \(r\geq0\) and every measurable coefficient \(v\), we therefore
have the exact local norm formula
\begin{align}
\left|
v(\partial_z)^{\otimes r}
\right|_{g_X}^2\,dA_X
=
2^{-r}|v|^2e^{2(r+1)U}\,dx\,dy.
\label{eq:adjoint-q-local-bundle-norm}
\end{align}

We next record the exact coordinate formula for \(B_m^X\).  If
\begin{align}
\phi=a(\partial_z)^{\otimes m},
\nonumber
\end{align}
then
\begin{align}
\bar\partial_m\phi
&=
(\partial_{\bar z}a)\,
d\bar z\otimes(\partial_z)^{\otimes m}
\nonumber\\
&=
2^{(1-m)/2}e^{(m-1)U}
(\partial_{\bar z}a)\,
\bar\zeta_X\otimes Z_X^{\otimes m}.
\nonumber
\end{align}
Since
\begin{align}
B_m^X=\sqrt2\,\mathcal I_{m,X}\bar\partial_m,
\qquad
\mathcal I_{m,X}
\bigl(\bar\zeta_X\otimes Z_X^{\otimes m}\bigr)
=
Z_X^{\otimes(m+1)},
\nonumber
\end{align}
and
\begin{align}
Z_X^{\otimes(m+1)}
=
2^{(m+1)/2}e^{-(m+1)U}
(\partial_z)^{\otimes(m+1)},
\nonumber
\end{align}
it follows that
\begin{align}
B_m^X
\bigl(a(\partial_z)^{\otimes m}\bigr)
=
2\sqrt2\,e^{-2U}
(\partial_{\bar z}a)
(\partial_z)^{\otimes(m+1)}.
\label{eq:adjoint-q-local-B-formula}
\end{align}

Put
\begin{align}
w:=(B_{l-1}^X)^*s,
\qquad
w=b(\partial_z)^{\otimes(l-1)}
\nonumber
\end{align}
on \(D\).  Let \(a\in C_c^\infty(D)\), and set
\begin{align}
\phi:=a(\partial_z)^{\otimes(l-1)}.
\nonumber
\end{align}
Using \eqref{eq:adjoint-q-local-bundle-norm},
\eqref{eq:adjoint-q-local-B-formula}, and the convention that the
Hermitian inner product is linear in its first variable, we obtain
\begin{align}
\left\langle B_{l-1}^X\phi,s\right\rangle_{H_l^X}
&=
2^{3/2-l}
\int_D
(\partial_{\bar z}a)\,
\overline f\,e^{2lU}\,dx\,dy,
\label{eq:adjoint-q-left-pairing}
\\
\left\langle\phi,w\right\rangle_{H_{l-1}^X}
&=
2^{1-l}
\int_D
a\,\overline b\,e^{2lU}\,dx\,dy.
\label{eq:adjoint-q-right-pairing}
\end{align}
The defining relation for the Hilbert adjoint says that
\eqref{eq:adjoint-q-left-pairing} and
\eqref{eq:adjoint-q-right-pairing} are equal for every
\(a\in C_c^\infty(D)\).  Distributional integration by parts therefore gives
\begin{align}
-\sqrt2\,
\partial_{\bar z}\bigl(e^{2lU}\overline f\bigr)
=
e^{2lU}\overline b
\qquad\text{in }\mathcal D'(D).
\nonumber
\end{align}
After complex conjugation,
\begin{align}
\partial_z\bigl(e^{2lU}f\bigr)
=
-2^{-1/2}e^{2lU}b
\qquad\text{in }\mathcal D'(D).
\label{eq:adjoint-q-distributional-adjoint}
\end{align}
Thus the distribution on the left of
\eqref{eq:adjoint-q-distributional-adjoint} is represented by a measurable
function, and we may define
\begin{align}
D_lf
:=
e^{-2lU}\partial_z\bigl(e^{2lU}f\bigr)
=
-2^{-1/2}b.
\label{eq:adjoint-q-Dl-definition}
\end{align}
Equivalently,
\begin{align}
(B_{l-1}^X)^*s
=
-\sqrt2\,D_lf\,(\partial_z)^{\otimes(l-1)}.
\label{eq:adjoint-q-local-adjoint-formula}
\end{align}
When \(U\) is smooth, and distributionally in the present sense,
\begin{align}
D_lf
=
\partial_zf+2lU_zf.
\nonumber
\end{align}

Formula \eqref{eq:adjoint-q-local-bundle-norm}, applied with \(r=l-1\),
now gives
\begin{align}
\left|
(B_{l-1}^X)^*s
\right|_{g_X}^2\,dA_X
&=
2^{1-l}
\left|\sqrt2\,D_lf\right|^2
e^{2lU}\,dx\,dy=
2^{2-l}|D_lf|^2e^{2lU}\,dx\,dy.
\label{eq:adjoint-q-local-adjoint-density}
\end{align}

In particular $D_lf\in L^2_{\mathrm{loc}}
\bigl(D,e^{2lU}dx\,dy\bigr).$

\textbf{Step (2)}. We next compute the weak derivative of \(q\).  By the local boundedness
of the holomorphic function \(f\), and
\eqref{eq:adjoint-q-local-adjoint-density}, all the products in the
following calculation are locally integrable.  

Since
\(\partial_z\overline f=0\), we obtain distributionally
\begin{align}
\partial_zq
&=
2^{-l}
\partial_z\bigl(e^{2lU}f\overline f\bigr)
=
2^{-l}
\partial_z\bigl(e^{2lU}f\bigr)\overline f
=
2^{-l}e^{2lU}D_lf\,\overline f.
\label{eq:adjoint-q-derivative}
\end{align}
The right-hand side belongs to \(L^1_{\mathrm{loc}}\) by
Cauchy--Schwarz.  Since \(q\) is real-valued, this also identifies
\(\partial_{\bar z}q\), and hence
\begin{align}
q\in W^{1,1}_{\mathrm{loc}}(D).
\nonumber
\end{align}

On the set where \(q>0\), formula
\eqref{eq:adjoint-q-derivative} gives
\begin{align}
\frac{4|\partial_zq|^2}{q}
&=
4
\frac{
2^{-2l}e^{4lU}|D_lf|^2|f|^2
}{
2^{-l}e^{2lU}|f|^2
}
=
2^{2-l}|D_lf|^2e^{2lU}.
\label{eq:adjoint-q-coordinate-energy}
\end{align}

Finally, write the fixed round metric in this coordinate as
\begin{align}
g_0=e^{2\psi}|dz|^2.
\nonumber
\end{align}
For the weak gradient of the real-valued function \(q\),
\begin{align}
|\nabla_0q|^2
=
e^{-2\psi}
\left(
|\partial_xq|^2+|\partial_yq|^2
\right)
=
4e^{-2\psi}|\partial_zq|^2,
\nonumber
\end{align}
while
\begin{align}
dA_0=e^{2\psi}\,dx\,dy.
\nonumber
\end{align}

Therefore
\begin{align}
\frac{|\nabla_0q|^2}{q}\,dA_0
&=
\frac{4|\partial_zq|^2}{q}\,dx\,dy
=
2^{2-l}|D_lf|^2e^{2lU}\,dx\,dy
=
\left|
(B_{l-1}^X)^*s
\right|_{g_X}^2\,dA_X.
\label{eq:adjoint-q-density-identity}
\end{align}

If \(s\not\equiv0\), its holomorphic coefficient has only finitely many
zeros.  Thus \(\{q=0\}\) has \(dA_0\)-measure zero.  We define
\(|\nabla_0q|^2/q\) to be zero on this set; this does not change the
integral.  If \(s\equiv0\), both sides below vanish identically.  

Since both sides of \eqref{eq:adjoint-q-density-identity} are globally defined
densities, the local identities agree on coordinate overlaps.  Integrating
over \(X\) gives
\begin{align}
\left\|(B_{l-1}^X)^*s\right\|_{H_{l-1}^X}^2
=
\int_X
\frac{|\nabla_0q|^2}{q}\,dA_0.
\label{eq:adjoint-q-energy}
\end{align}

\textbf{Step (2)}. If \(Z(s)\) is the finite zero divisor of \(s\), then the Poincar\'e--Lelong formula
\cite[p.~388]{GriffithsHarrisPAG}
gives
\begin{align}
\Delta_0\log|f|
=
2\pi\sum_{p\in Z(s)}
\operatorname{ord}_p(s)\,\delta_p
\label{eq:local-PL-f}
\end{align}
as distributions.  Here we use the convention
\(\Delta_0=\operatorname{div}_{g_0}\nabla_0\), so that
\(-\Delta_0\) is nonnegative and, in a Euclidean coordinate,
\(\Delta_{\mathrm{euc}}\log|z|=2\pi\delta_0\).

Since
\begin{align}
\log q=-l\log2+2\log|f|+2lU, \quad \quad -\Delta_0U=\omega_X,\nonumber
\end{align}
we obtain
\begin{align}
-\Delta_0\log q
=
2l\,\omega_X
-
4\pi\sum_{p\in Z(s)}
\operatorname{ord}_p(s)\,\delta_p
\label{eq:PL-q}
\end{align}
as measure-valued distributions.  

Set
\begin{align}
E(q)
:=
\int_X\frac{|\nabla_0q|^2}{q}\,dA_0
<\infty,
\label{eq:q-total-fisher-energy}
\end{align}
where the integrand is defined to be zero on \(\{q=0\}\).

Because \(q\in L^1_{\mathrm{loc}}(dA_0)\) and
\eqref{eq:q-total-fisher-energy} is finite, one has
\begin{align}
\sqrt q\in W^{1,2}_{\mathrm{loc}},
\qquad
|\nabla_0\sqrt q|^2
=
\frac14\frac{|\nabla_0q|^2}{q}
\quad\text{a.e. on }\{q>0\}.
\nonumber
\end{align}

For \(R>0\), put
\begin{align}
\phi_R(x)\vcentcolon= \min\{q(x),R\}, \quad \quad \forall x\in X.
\nonumber
\end{align}
It follows that
\(\phi_R\in W^{1,2}_{\mathrm{loc}}\cap L^\infty\), and the
Sobolev chain rule gives
\begin{align}
\nabla_0\phi_R
=
\mathbf 1_{\{q<R\}}\nabla_0q
\qquad\text{a.e.}
\label{eq:truncation-chain-rule}
\end{align}

Write
\begin{align}
Z(s)=\{p_1,\ldots,p_N\}.
\nonumber
\end{align}
Choose mutually disjoint holomorphic coordinate discs
\(D_\rho(p_\alpha)\), \(1\leq\alpha\leq N\), containing no zeros
other than their centers.  For \(0<\varepsilon<\rho\), set
\begin{align}
\Omega_\varepsilon
:=
X\setminus
\bigcup_{\alpha=1}^N
\overline{D_\varepsilon(p_\alpha)}.
\nonumber
\end{align}

Let \(n_\varepsilon\) be the outward unit normal of
\(\Omega_\varepsilon\).  On an inner boundary circle this normal
points into the deleted disc.  We also write
\(\nu_\alpha\) for the unit normal to
\(\partial D_\varepsilon(p_\alpha)\) pointing away from \(p_\alpha\);
thus
\begin{align}
n_\varepsilon=-\nu_\alpha
\qquad\text{on }\partial D_\varepsilon(p_\alpha).
\nonumber
\end{align}

On \(\Omega_\varepsilon\), the divisor terms in
\eqref{eq:PL-q} are absent, and hence
\begin{align}
-\Delta_0\log q
=
2l\,\omega_X
\qquad\text{on }\Omega_\varepsilon
\label{eq:PL-q-punctured}
\end{align}
in the distributional sense.  

Testing
\eqref{eq:PL-q-punctured} by \(\phi_R\) and using the weak
Green formula gives, for almost every admissible \(\varepsilon\),
\begin{align}
2l\int_{\Omega_\varepsilon}\phi_R\,d\omega_X
&=
\int_{\Omega_\varepsilon}
\left\langle
\nabla_0\log q,\nabla_0\phi_R
\right\rangle
\,dA_0
-
\int_{\partial\Omega_\varepsilon}
\phi_R\,\partial_{n_\varepsilon}\log q\,ds_0
\nonumber\\
&=
\int_{\Omega_\varepsilon\cap\{0<q<R\}}
\frac{|\nabla_0q|^2}{q}\,dA_0
+
\sum_{\alpha=1}^N
\int_{\partial D_\varepsilon(p_\alpha)}
\phi_R\,\partial_{\nu_\alpha}\log q\,ds_0.
\label{eq:punctured-truncation-green}
\end{align}
Indeed, away from \(Z(s)\), $\nabla_0\log q=\frac{\nabla_0q}{q}$, and therefore \eqref{eq:truncation-chain-rule} gives
\begin{align}
\left\langle
\nabla_0\log q,\nabla_0\phi_R
\right\rangle
=
\mathbf 1_{\{0<q<R\}}
\frac{|\nabla_0q|^2}{q}
\qquad\text{a.e.}
\label{eq:truncated-interior-integrand}
\end{align}

It remains to show that the boundary fluxes in
\eqref{eq:punctured-truncation-green} vanish along a sequence
\(\varepsilon_j\downarrow0\).

For almost every \(r\in(0,\rho)\), set
\begin{align}
e_\alpha(r)
:=
\int_{\partial D_r(p_\alpha)}
\frac{|\nabla_{\mathrm{euc}}q|^2}{q}\,
ds_{\mathrm{euc}}.
\nonumber
\end{align}
The coarea formula and
\eqref{eq:q-total-fisher-energy} give
\begin{align}
\int_0^\rho e_\alpha(r)\,dr
=
\int_{D_\rho(p_\alpha)}
\frac{|\nabla_{\mathrm{euc}}q|^2}{q}\,
dx\,dy
<\infty.
\label{eq:circle-energy-integrable}
\end{align}

\textbf{Step (3)}. Put
\begin{align}
e(r):=\sum_{\alpha=1}^N e_\alpha(r).
\nonumber
\end{align}
Then \(e\in L^1(0,\rho)\).

We claim that, for every \(\delta,\eta>0\), the set
\begin{align}
\left\{
r\in (0,\delta):
r e(r)<\eta
\right\}
\label{eq:good-radius-positive-measure-set}
\end{align}
has positive Lebesgue measure.  Otherwise, one would have
\begin{align}
e(r)\geq\frac{\eta}{r}
\qquad\text{for a.e. }r\in(0,\delta),
\nonumber
\end{align}
and hence
\begin{align}
\int_0^\delta e(r)\,dr
\geq
\eta\int_0^\delta\frac{dr}{r}
=
+\infty,
\nonumber
\end{align}
contrary to \(e\in L^1(0,\rho)\).

Set \(\varepsilon_0:=\rho\).  Inductively, having chosen
\(\varepsilon_{j-1}\), put
\begin{align}
\delta_j
:=
\min\left\{
\frac{\rho}{2^j},
\frac{\varepsilon_{j-1}}{2}
\right\}.
\nonumber
\end{align}

By \eqref{eq:good-radius-positive-measure-set}, we may choose $\varepsilon_j\in \cap(0,\delta_j)$
such that $\varepsilon_j e(\varepsilon_j)<\frac1j.$ Then
\begin{align}
\lim_{j\rightarrow\infty}\varepsilon_j e(\varepsilon_j)= 0.
\label{eq:good-radii}
\end{align}

\textbf{Step (4)}. In the holomorphic coordinate \(z\) centered at \(p_\alpha\), write
\[
g_0=e^{2\psi_\alpha}|dz|^2.
\]
Let \(\nu_{\alpha,0}\) and
\(\nu_{\alpha,\mathrm{euc}}\) denote the unit normals to
\(\partial D_r(p_\alpha)\), with respect to \(g_0\) and the Euclidean
metric, respectively, both pointing away from \(p_\alpha\).  Then
\[
\nu_{\alpha,0}
=
e^{-\psi_\alpha}\nu_{\alpha,\mathrm{euc}},
\qquad
ds_0=e^{\psi_\alpha}ds_{\mathrm{euc}},
\]
and hence
\[
\partial_{\nu_{\alpha,0}}f\,ds_0
=
\partial_{\nu_{\alpha,\mathrm{euc}}}f\,
ds_{\mathrm{euc}}.
\]
In particular,
\[
\int_{\partial D_r(p_\alpha)}
\phi_R\,\partial_{\nu_{\alpha,0}}\log q\,ds_0
=
\int_{\partial D_r(p_\alpha)}
\phi_R\,\partial_{\nu_{\alpha,\mathrm{euc}}}\log q\,
ds_{\mathrm{euc}}.
\]

For \(q>0\), since \(\phi_R=\min\{q,R\}\), one has $\frac{\phi_R^2}{q}\leq R$. Consequently, by Cauchy--Schwarz,
\begin{align}
\left|
\int_{\partial D_{\varepsilon_j}(p_\alpha)}
\phi_R\,\partial_{\nu_{\alpha,0}}\log q\,ds_0
\right|
&=
\left|
\int_{\partial D_{\varepsilon_j}(p_\alpha)}
\frac{\phi_R}{q}\,
\partial_{\nu_{\alpha,\mathrm{euc}}}q\,
ds_{\mathrm{euc}}
\right|
\nonumber\\
&\leq
\left(
\int_{\partial D_{\varepsilon_j}(p_\alpha)}
\frac{
|\partial_{\nu_{\alpha,\mathrm{euc}}}q|^2
}{q}\,
ds_{\mathrm{euc}}
\right)^{1/2}
\left(
\int_{\partial D_{\varepsilon_j}(p_\alpha)}
\frac{\phi_R^2}{q}\,
ds_{\mathrm{euc}}
\right)^{1/2}
\nonumber\\
&\leq
\bigl(
2\pi R\varepsilon_j e_\alpha(\varepsilon_j)
\bigr)^{1/2}.
\label{eq:puncture-flux-estimate}
\end{align}

Summing over \(\alpha\) and using Cauchy--Schwarz in
\(\mathbb R^N\),
\begin{align}
\lim_{j\rightarrow\infty}\sum_{\alpha=1}^N
\left|
\int_{\partial D_{\varepsilon_j}(p_\alpha)}
\phi_R\,\partial_{\nu_\alpha}\log q\,ds_0
\right|
\leq
\lim_{j\rightarrow\infty}\bigl(
2\pi NR\varepsilon_j e(\varepsilon_j)
\bigr)^{1/2}=0 \label{eq:all-puncture-fluxes-vanish}
\end{align}
by \eqref{eq:good-radii}.

We now let \(j\to\infty\) in
\eqref{eq:punctured-truncation-green}.  Since the integrand in
\eqref{eq:truncated-interior-integrand} is dominated by the integrable
function $\frac{|\nabla_0q|^2}{q}$, we have
\begin{align}
\int_{\Omega_{\varepsilon_j}\cap\{0<q<R\}}
\frac{|\nabla_0q|^2}{q}\,dA_0
\longrightarrow
\int_{\{0<q<R\}}
\frac{|\nabla_0q|^2}{q}\,dA_0.
\label{eq:punctured-interior-limit}
\end{align}

Moreover, \(\omega_X\) is atomless and \(\phi_R\leq R\), so
\begin{align}
\left|
\int_X\phi_R\,d\omega_X
-
\int_{\Omega_{\varepsilon_j}}\phi_R\,d\omega_X
\right|
&\leq
R\sum_{\alpha=1}^N
\omega_X\bigl(\overline{D_{\varepsilon_j}(p_\alpha)}\bigr)
\longrightarrow0.
\label{eq:punctured-measure-limit}
\end{align}
Here we used
\begin{align}
\omega_X\bigl(\overline{D_{\varepsilon_j}(p_\alpha)}\bigr)
\longrightarrow
\omega_X(\{p_\alpha\})=0.
\nonumber
\end{align}
Combining
\eqref{eq:punctured-truncation-green},
\eqref{eq:all-puncture-fluxes-vanish},
\eqref{eq:punctured-interior-limit}, and
\eqref{eq:punctured-measure-limit}, we obtain
\begin{align}
\int_{\{0<q<R\}}
\frac{|\nabla_0q|^2}{q}\,dA_0
=
2l\int_X \phi_R\,d\omega_X.
\label{eq:truncated-defect-identity}
\end{align}
Since the integrand is defined to be zero on \(\{q=0\}\), the
left-hand side may equivalently be written as $\int_{\{q<R\}}
\frac{|\nabla_0q|^2}{q}\,dA_0$.

Finally, as \(R\to\infty\),
\begin{align}
\mathbf 1_{\{q<R\}}
\frac{|\nabla_0q|^2}{q}
\uparrow
\frac{|\nabla_0q|^2}{q}
\quad\text{a.e.},
\qquad
\phi_R(x)\uparrow q(x).
\nonumber
\end{align}
Monotone convergence in
\eqref{eq:truncated-defect-identity} therefore yields
\begin{align}
\int_X\frac{|\nabla_0q|^2}{q}\,dA_0
=
2l\int_Xq\,d\omega_X,
\nonumber
\end{align}
which, together with \eqref{eq:adjoint-q-energy}, proves
\eqref{eq:no-atomic-residue-identity}.
\end{proof}

\begin{prop}\label{prop:atomless-BK-defect-inequality}
Assume \(\nu_X\) is atomless.  Then, for every \(l\geq1\) ,
\begin{align}
\mathfrak c_{l-1}^X(s,s)-\mathfrak a_l^X(s,s) =2l\int_X|s|_{h_{l, X}}^2\,d\nu_X, \quad \quad \forall s\in V_l^X. \nonumber
\end{align}
\end{prop}

\begin{proof}
Because \(B_l^Xs=0\), the definitions of the two shifted forms give
\begin{align}
\mathfrak c_{l-1}^X(s,s)-\mathfrak a_l^X(s,s) =\|(B_{l-1}^X)^*s\|^2-2l\|s\|_{H_l^X}^2. \nonumber
\end{align}
Apply Lemma~\ref{lem:no-atomic-residue-lsc} and use
\(\|s\|_{H_l^X}^2=\int|s|^2dA_X\) and
\(\nu_X=\omega_X-dA_X\).
\end{proof}

\section{Eigenvalue comparison and rigidity}\label{sec:alexandrov-comparison-rigidity}

\begin{prop}\label{prop:alex-ordered-comparison-by-approx}
For every \(j\geq1\),
\begin{align}
\lambda_j(X)\geq\lambda_j(\mathbb S^2,g_{\mathrm{round}}). \nonumber
\end{align}
\end{prop}

\begin{proof}
Choose $g_{\tau_i}$ from (\ref{def:g_tau-by-heat}), 
then each \(g_{\tau_i}\) has \(K_{g_{\tau_i}}\geq1\) by Lemma \ref{lem:cp-heat-preserves-curv}. 

So Theorem~\ref{thm:cluster-bottom-comparison} gives the inequality for
\(g_{\tau_i}\).  Pass to the limit using
Lemma~\ref{lem:alex-spectral-convergence}.
\end{proof}

\begin{cor}\label{cor:alex-counting-comparison-by-approx}
For every \(l\geq1\),
\begin{align}
N_{-\Delta_X}^{<}\bigl(l(l+1)\bigr)&\leq l^2, \quad \quad 
\mathcal N_{-\Delta_X}\bigl(l(l+1)\bigr)\leq(l+1)^2. \label{eq:alex-counting-comparison}
\end{align}
\end{cor}

\begin{proof}
Apply Corollary~\ref{cor:abstract-two-sphere-ladder} directly to
\begin{align}
H_m=H_m^X, \qquad B_m=B_m^X, \qquad \alpha_m=m(m+1). \nonumber
\end{align}
The compact-resolvent hypothesis is
Lemma~\ref{lem:alex-ladder-compactness}, the form order is
Proposition~\ref{prop:weak-Bochner-Kodaira}, and
\(\dim\ker B_m^X\leq2m+1\) is
Lemma~\ref{lem:atomic-kernel-holomorphic}.  

Finally,
\(A_0^X=-\Delta_X\) by \eqref{eq:alex-A0-laplacian}.  Thus
\eqref{eq:alex-counting-comparison} is exactly
\eqref{eq:abstract-two-sphere-threshold-counts}.  The approximation argument
in Proposition~\ref{prop:alex-ordered-comparison-by-approx} gives the same
ordered comparison independently.
\end{proof}

\begin{lemma}\label{lem:alex-zero-defect-roundness}
If \(\nu_X=0\), then \(X\) is isometric to the unit round sphere.
\end{lemma}

\begin{proof}

The Alexandrov Bonnet--Myers diameter theorem
\cite[Theorem~3.6]{BGP}
gives
\[
\operatorname{diam}X\leq\pi.
\]
Fix \(p\in X\) and choose \(R>\pi\).  Then \(B_R(p)=X\), while the
radius-\(R\) ball in the simply connected two-dimensional model of
constant curvature \(1\) is the whole unit sphere.  

By \eqref{eq:alex-gauss-bonnet},
\begin{align}
\mathcal H^2(X)=\int_X 1dA_X=\omega_X(X)=4\pi. \nonumber
\end{align}

Since both balls have two-dimensional Hausdorff measure \(4\pi\), the equality case of
the Alexandrov volume-comparison theorem
\cite[Theorem~10.2]{BGP}
shows that \(X\) is isometric to the unit round sphere.
\end{proof}

\begin{theorem}[Alexandrov counting rigidity]
\label{claim:alexandrov-link-rigidity-needed}
Let \(X\) be a compact Alexandrov two-sphere with curvature at least \(1\).
If, for some \(l\geq1\),
\begin{align}
\mathcal N_{-\Delta_X}\bigl(l(l+1)\bigr)=(l+1)^2, \nonumber
\end{align}
then \(X\) is isometric to the unit round sphere.
\end{theorem}

\begin{proof}
By Proposition~\ref{prop:atomic-alexandrov-counting-obstruction}, the defect
measure \(\nu_X\) is atomless.  By
Lemma~\ref{lem:atomic-counting-saturation}, the concrete application of
Proposition~\ref{prop:abstract-ladder-saturation} gives
\begin{align}
V_l^X=\ker B_l^X, \qquad \dim_{\C}V_l^X=2l+1, \nonumber
\end{align}
and
\begin{align}
\mathfrak c_{l-1}^X(s,s)=\mathfrak a_l^X(s,s) \qquad(s\in V_l^X). \nonumber
\end{align}
Proposition~\ref{prop:atomless-BK-defect-inequality} therefore yields
\begin{align}
\int_X|s|_{h_{l, X}}^2\,d\nu_X=0 \qquad\text{for every }s\in V_l^X. \label{eq:alex-each-section-zero}
\end{align}

By Lemma~\ref{lem:atomic-kernel-holomorphic} and the dimension equality,
\begin{align}
V_l^X=H^0(\CP^1,K_{\CP^1}^{-l}). \nonumber
\end{align}

For every \(p\in\CP^1\), the evaluation map is surjective by
Lemma~\ref{lem:CP1-anticanonical-cohomology}; hence some
\(s_p\in V_l^X\) satisfies \(s_p(p)\neq0\).  

In a coordinate centered at
\(p\), its holomorphic coefficient is therefore bounded away from zero on a
smaller neighborhood.  The local conformal factor is bounded from below, so
\(|s_p|_{h_{l, X}}^2\geq c_p>0\) almost everywhere on that neighborhood.
Equation~\eqref{eq:alex-each-section-zero} then implies that \(\nu_X\)
vanishes there.  

Compactness gives finitely many such neighborhoods, and hence
\(\nu_X=0\).

Lemma~\ref{lem:alex-zero-defect-roundness} now shows that \(X\) is round.
\end{proof}

\part{Eigenvalues comparison and dimension of harmonic functions}

\section{A counterexample in dimension three}\label{sec:S3-counterexample}\label{sec:harmonic-growth}

This section explains why the two-dimensional ladder should not be regarded as the shadow of a straightforward higher-dimensional theorem.  The obstruction is not only technical.  In dimension two, the scalar spherical harmonic multiplicities are reproduced by the line-bundle dimensions $\dim H^0(\CP^1,K^{-m})=2m+1$, and the curvature commutator has the correct sign under $K\geq1$.  In dimension three, a Ricci lower bound does not supply an analogous first-order ladder with the round multiplicities, and the exact eigenvalue comparison is in fact false.

The example below is included to mark the boundary of the method.  It gives a completely explicit counterexample to the statement that $\Ric_g\geq2g$ on $S^3$ should imply
$\lambda_i(S^3,g)\geq\lambda_i(\mathbb S^3,g_{\mathrm{round}})$ for every $i\geq1$.

\begin{theorem}[An explicit conformal counterexample on $S^3$]\label{thm:s3-failure-2}
Let $g_0=g_{\mathrm{round}}$ be the unit round metric on
$S^3\subset\mathbb C^2$, where
$S^3=\{(z,w)\in\mathbb C^2: |z|^2+|w|^2=1\}$.  Define
\begin{align}
q:=|z|^2,\qquad \psi(q):=\frac{49}{50}-31q+\frac{75}{4}q^2-6q^3. \nonumber
\end{align}
For every sufficiently small $\varepsilon>0$, the conformal metric
\begin{align}
\widehat g_\varepsilon:=e^{2\varepsilon\psi}g_0 \nonumber
\end{align}
satisfies
\begin{align}
\Ric_{\widehat g_\varepsilon}\geq 2\widehat g_\varepsilon, \nonumber
\end{align}
but
\begin{align}
\lambda_{7714}(S^3,\widehat g_\varepsilon) < 840 = \lambda_{7714}(\mathbb S^3,g_0), \nonumber
\end{align}
where eigenvalues are counted with the convention $\lambda_0=0$.
\end{theorem}

\begin{proof}
We use the standard coordinates
\begin{align}
z=e^{i\theta}\sin t,\qquad w=e^{i\phi}\cos t,\qquad 0<t<\frac{\pi}{2}, \nonumber
\end{align}
so that $q=\sin^2t$ and
\begin{align}
g_0=dt^2+\sin^2t\,d\theta^2+\cos^2t\,d\phi^2. \nonumber
\end{align}
The role of the Hopf variable is summarized in Figure~\ref{fig:S3-hopf-coordinate}: the conformal deformation is constant on each Hopf torus and is controlled by a single polynomial profile in $q$.

\begin{figure}[H]
\centering
\resizebox{0.88\textwidth}{!}{%
\begin{tikzpicture}[>=Latex, every node/.style={font=\small}]
\draw[->, thick] (-0.45,0) -- (9.0,0) node[right] {$q=|z|^2=\sin^2t$};
\foreach \x/\lab in {0/$0$,4.2/$q$,8.4/$1$}{
  \draw[thick] (\x,0.08)--(\x,-0.08) node[below=4pt] {\lab};
}
\node[below=18pt] at (0,0) {$z=0$};
\node[below=18pt] at (8.4,0) {$w=0$};
\node[align=center, above=23pt] at (0,0) {Hopf circle\\$|w|=1$};
\node[align=center, above=23pt] at (8.4,0) {Hopf circle\\$|z|=1$};
\draw[thick] (0,0.78) ellipse (0.30 and 0.12);
\draw[thick] (8.4,0.78) ellipse (0.30 and 0.12);
\foreach \x/\a/\b in {1.45/0.38/0.14,2.8/0.58/0.20,4.2/0.78/0.27,5.6/0.58/0.20,6.95/0.38/0.14}{
  \draw[thick] (\x,1.08) ellipse[x radius=\a,y radius=\b];
  \draw[densely dashed] (\x-\a,1.08) arc[start angle=180,end angle=360,x radius=\a,y radius=\b];
  \draw (\x-0.18,1.08) ellipse[x radius=0.15,y radius=\b];
}
\node[align=center] at (4.2,1.78) {Clifford tori $T_q=\{|z|^2=q,\ |w|^2=1-q\}$};
\draw[->, thick] (4.2,0.16) -- (4.2,0.78);
\node[right] at (4.35,0.48) {$0<t<\pi/2$};
\draw[->, thick] (1.15,-1.05) -- (7.35,-1.05) node[right] {$q$};
\draw[thick] (1.15,-0.75) .. controls (3.2,-0.98) and (5.6,-1.45) .. (7.35,-1.78);
\node[left] at (1.15,-0.75) {$\psi(q)$};
\node[align=center] at (4.2,-2.18) {$\widehat g_\varepsilon=e^{2\varepsilon\psi(q)}g_0$ is constant along each $T_q$.};
\end{tikzpicture}%
}
\caption{Hopf-coordinate schematic for Theorem~\ref{thm:s3-failure-2}.  The deformation depends only on $q$, so the curvature and Rayleigh-quotient calculations reduce to one-dimensional formulae in $q$.}
\label{fig:S3-hopf-coordinate}
\end{figure}
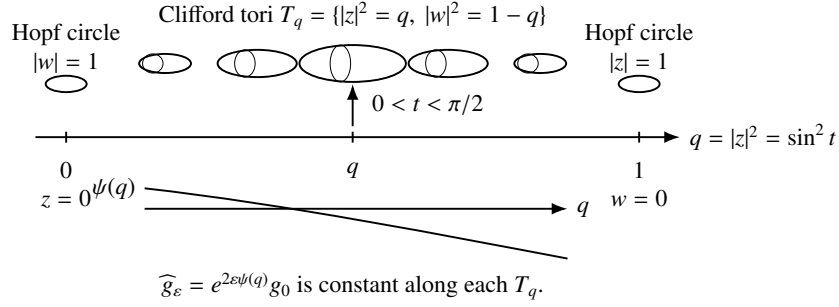

For a conformal change $\widehat g=e^{2u}g$ in dimension $3$, one has
\begin{align}
\Ric_{\widehat g} = \Ric_g-\nabla^2u-(\Delta u)g+du\otimes du-|\nabla u|^2g. \nonumber
\end{align}
Taking $u=\varepsilon\psi$ and using $\Ric_{g_0}=2g_0$, we get
\begin{align}
\Ric_{\widehat g_\varepsilon}-2\widehat g_\varepsilon &= 2(1-e^{2\varepsilon\psi})g_0 -\varepsilon\nabla^2\psi -\varepsilon(\Delta\psi)g_0 +\varepsilon^2(d\psi\otimes d\psi-|\nabla\psi|^2g_0) \nonumber\\
&= \varepsilon T+O(\varepsilon^2), \nonumber
\end{align}
where the remainder is uniform on $S^3$ and
\begin{align}
T:=-\nabla^2\psi-(\Delta\psi)g_0-4\psi g_0. \nonumber
\end{align}

Let
\begin{align}
e_t:=\partial_t,\qquad e_\theta:=(\sin t)^{-1}\partial_\theta,\qquad e_\phi:=(\cos t)^{-1}\partial_\phi. \nonumber
\end{align}
For any function $F=F(q)$, direct differentiation gives
\begin{align}
\nabla^2F(e_t,e_t) &= 4q(1-q)F''+2(1-2q)F', \nonumber\\
\nabla^2F(e_\theta,e_\theta) &= 2(1-q)F', \nonumber\\
\nabla^2F(e_\phi,e_\phi) &= -2qF', \nonumber\\
\Delta F &= 4q(1-q)F''+4(1-2q)F'. \nonumber
\end{align}
Therefore, in the orthonormal frame $\{e_t,e_\theta,e_\phi\}$, the tensor $T$ is diagonal.  Substituting
\begin{align}
\psi(q)=\frac{49}{50}-31q+\frac{75}{4}q^2-6q^3 \nonumber
\end{align}
gives
\begin{align}
T_{tt} &= \frac{2}{25}+(1-q)(480q^2-591q+182), \nonumber\\
T_{\theta\theta} &= \frac{2}{25}+23+3(1-q)(100q^2-134q+53), \nonumber\\
T_{\phi\phi} &= \frac{2}{25}+6(1-q)(50q^2-61q+20). \nonumber
\end{align}
Figure~\ref{fig:S3-polynomial-data} records the sign structure of these formulae.  The picture is schematic; the proof uses the exact discriminant calculation immediately below.

\begin{figure}[H]
\centering
\resizebox{0.92\textwidth}{!}{%
\begin{tikzpicture}[>=Latex, every node/.style={font=\small}]
\begin{scope}[shift={(0,0)}]
  \draw[->, thick] (0,-1.85) -- (0,1.25) node[above] {$\psi$};
  \draw[->, thick] (-0.15,0) -- (4.65,0) node[right] {$q$};
  \draw (0,0.06)--(0,-0.06) node[below=3pt] {$0$};
  \draw (4.2,0.06)--(4.2,-0.06) node[below=3pt] {$1$};
  \draw[thick] (0,0.85) .. controls (1.1,0.40) and (2.8,-0.65) .. (4.2,-1.45);
  \node[anchor=west] at (0.18,0.93) {$\psi(0)=49/50$};
  \node[anchor=east] at (4.2,-1.58) {$\psi(1)=-1727/100$};
  \node[align=center] at (2.1,-2.22) {(a) conformal profile};
\end{scope}
\begin{scope}[shift={(6.2,0)}]
  \draw[->, thick] (0,-0.25) -- (0,2.65) node[above] {$T_{aa}$};
  \draw[->, thick] (-0.15,0) -- (4.65,0) node[right] {$q$};
  \draw[densely dashed] (0,0.42) -- (4.35,0.42) node[right] {$2/25$};
  \draw (0,0.06)--(0,-0.06) node[below=3pt] {$0$};
  \draw (4.2,0.06)--(4.2,-0.06) node[below=3pt] {$1$};
  \draw[thick] (0,2.20) .. controls (1.0,1.60) and (2.7,1.35) .. (4.2,0.50);
  \draw[thick, dashed] (0,1.65) .. controls (1.0,1.95) and (2.7,1.15) .. (4.2,0.72);
  \draw[thick, dotted] (0,1.12) .. controls (1.0,1.45) and (2.7,0.88) .. (4.2,0.46);
  \node[anchor=west] at (0.22,2.25) {$T_{tt}$};
  \node[anchor=west] at (0.50,1.82) {$T_{\theta\theta}$};
  \node[anchor=west] at (0.32,1.08) {$T_{\phi\phi}$};
  \node[align=center] at (2.1,-2.22) {(b) all diagonal entries stay above $2/25$};
\end{scope}
\end{tikzpicture}%
}
\caption{The one-variable polynomial data in the conformal counterexample.  The lower bound for $T$ is proved by showing that the three displayed quadratic factors are everywhere positive, so each diagonal entry is at least $2/25$.}
\label{fig:S3-polynomial-data}
\end{figure}
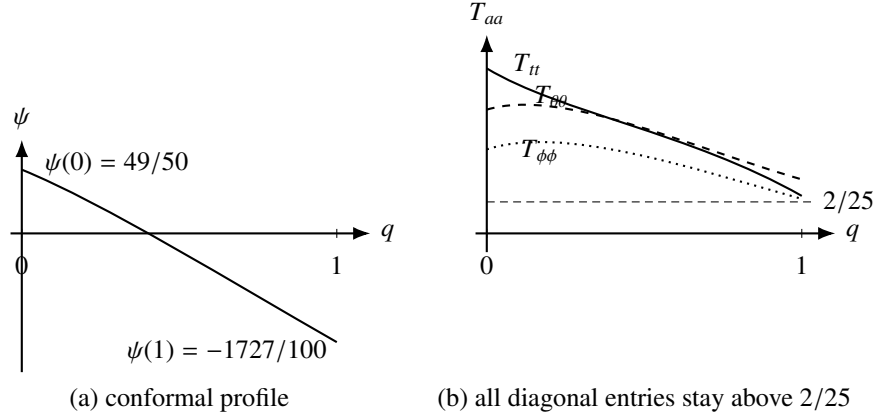

The three quadratic factors are positive on $\mathbb R$, since their discriminants are
\begin{align}
591^2-4\cdot480\cdot182&=-159<0, \nonumber\\
134^2-4\cdot100\cdot53&=-3244<0, \nonumber\\
61^2-4\cdot50\cdot20&=-279<0. \nonumber
\end{align}
Since $0\leq q\leq1$, it follows that
\begin{align}
T\geq \frac{2}{25}g_0. \label{eq:S3-T-lower}
\end{align}
By \eqref{eq:S3-T-lower}, after shrinking $\varepsilon>0$ if necessary,
\begin{align}
\Ric_{\widehat g_\varepsilon}-2\widehat g_\varepsilon\geq0. \nonumber
\end{align}

It remains to prove that an eigenvalue drops.  For $k\geq1$, consider the round spherical harmonic
\begin{align}
f_k(z,w):=\operatorname{Re}(w^k)=(1-q)^{k/2}\cos(k\phi). \nonumber
\end{align}
Let $\mathcal R_\varepsilon$ be the Rayleigh quotient of the positive operator $-\Delta_{\widehat g_\varepsilon}$.  Since
\begin{align}
|\nabla f_k|_{\widehat g_\varepsilon}^2\,dV_{\widehat g_\varepsilon} &= e^{\varepsilon\psi}|\nabla f_k|_{g_0}^2\,dV_{g_0}, \nonumber\\
f_k^2\,dV_{\widehat g_\varepsilon} &= e^{3\varepsilon\psi}f_k^2\,dV_{g_0}, \nonumber
\end{align}
and since
\begin{align}
dV_{g_0}=\sin t\cos t\,dt\,d\theta\,d\phi=\frac12\,dq\,d\theta\,d\phi, \nonumber
\end{align}
we can compute the quotient explicitly.  First,
\begin{align}
|\nabla f_k|_{g_0}^2 = k^2(1-q)^{k-1}\bigl(q\cos^2(k\phi)+\sin^2(k\phi)\bigr). \nonumber
\end{align}
Moreover,
\begin{align}
\int_0^{2\pi}\int_0^{2\pi} \bigl(q\cos^2(k\phi)+\sin^2(k\phi)\bigr)\,d\theta\,d\phi &= 2\pi^2(1+q), \nonumber\\
\int_0^{2\pi}\int_0^{2\pi} \cos^2(k\phi)\,d\theta\,d\phi &= 2\pi^2. \nonumber
\end{align}
Thus
\begin{align}
\mathcal R_\varepsilon(f_k) &= k^2\frac{A_k(\varepsilon)}{B_k(\varepsilon)}, \nonumber\\
A_k(\varepsilon) &= \int_0^1(1-q)^{k-1}(1+q)e^{\varepsilon\psi(q)}\,dq, \nonumber\\
B_k(\varepsilon) &= \int_0^1(1-q)^ke^{3\varepsilon\psi(q)}\,dq. \nonumber
\end{align}
At $\varepsilon=0$,
\begin{align}
A_k(0)=\frac{k+2}{k(k+1)},\qquad B_k(0)=\frac{1}{k+1},\qquad \mathcal R_0(f_k)=k(k+2). \nonumber
\end{align}
Differentiating at $\varepsilon=0$ gives
\begin{align}
\mathcal R_\varepsilon(f_k) = k(k+2)\bigl(1+\Gamma_k\varepsilon+O(\varepsilon^2)\bigr), \nonumber
\end{align}
where
\begin{align}
\Gamma_k := \frac{k(k+1)}{k+2} \int_0^1(1-q)^{k-1}(1+q)\psi(q)\,dq - 3(k+1)\int_0^1(1-q)^k\psi(q)\,dq. \label{eq:S3-gamma}
\end{align}

We now take $k=28$.  The only integral identity needed is
\begin{align}
\int_0^1q^r(1-q)^s\,dq=\frac{r!\,s!}{(r+s+1)!}. \nonumber
\end{align}
Let $\gamma_r$ denote the contribution to $\Gamma_{28}$ in \eqref{eq:S3-gamma} from the monomial $q^r$.  For $r=0,1,2,3$, the above beta integral gives
\begin{align}
\gamma_0=-2,\qquad \gamma_1=-\frac{29}{450},\qquad \gamma_2=-\frac{28}{6975},\qquad \gamma_3=-\frac{9}{24800}. \nonumber
\end{align}
Consequently,
\begin{align}
\Gamma_{28} &= -2\cdot\frac{49}{50} -\frac{29}{450}(-31) -\frac{28}{6975}\cdot\frac{75}{4} -\frac{9}{24800}(-6) \nonumber\\
&= -\frac{3941}{111600}<0. \label{eq:S3-gamma-28}
\end{align}
By \eqref{eq:S3-gamma-28}, for all sufficiently small $\varepsilon>0$,
\begin{align}
\mathcal R_\varepsilon(f_{28})<28\cdot30=840. \label{eq:S3-f28-bound}
\end{align}

Let $E_{<28}$ be the real vector space spanned by all round spherical harmonics of degrees $0,1,\ldots,27$.  The degree $j$ eigenspace on the round $S^3$ has dimension $(j+1)^2$, hence
\begin{align}
D:=\dim E_{<28} = \sum_{j=0}^{27}(j+1)^2 = \frac{28\cdot29\cdot57}{6} = 7714. \label{eq:S3-D}
\end{align}
On the round sphere,
\begin{align}
\sup_{0\neq u\in E_{<28}}\mathcal R_0(u) = 27\cdot29 = 783 < 840. \nonumber
\end{align}
Since $E_{<28}$ is finite dimensional and $\widehat g_\varepsilon\rightarrow g_0$ smoothly, after shrinking $\varepsilon>0$ once more we have
\begin{align}
\sup_{0\neq u\in E_{<28}}\mathcal R_\varepsilon(u)<840. \label{eq:S3-low-bound}
\end{align}

The conformal factor depends only on $q$, and hence is independent of $\theta$ and $\phi$.  Therefore the $L^2(\widehat g_\varepsilon)$ inner product and the Dirichlet form have no cross terms between different $\phi$-Fourier frequencies.  Every element of $E_{<28}$ has $\phi$-Fourier frequencies $|m|\leq27$, while $f_{28}$ has only the frequencies $m=\pm28$.  Define
\begin{align}
F:=E_{<28}\oplus\operatorname{span}\{f_{28}\}. \nonumber
\end{align}
Then
\begin{align}
\dim F=D+1=7715. \nonumber
\end{align}
The final min--max test space is depicted in Figure~\ref{fig:S3-minmax-picture}: the low-degree harmonic block and the single perturbed degree-$28$ test function both lie below the round threshold $840$.

\begin{figure}[H]
\centering
\resizebox{0.90\textwidth}{!}{%
\begin{tikzpicture}[>=Latex, every node/.style={font=\small}]
\draw[->, thick] (-0.2,0) -- (9.6,0) node[right] {test space};
\draw[->, thick] (0,-0.45) -- (0,4.05) node[above] {Rayleigh level};
\draw[densely dashed, thick] (0,3.15) -- (8.9,3.15) node[right] {$840=28\cdot30$};

\draw[fill=gray!12, thick, rounded corners] (0.75,0.75) rectangle (5.05,2.35);
\node[align=center] at (2.90,1.55) {$E_{<28}$\\degrees $0,\ldots,27$\\$\sup \mathcal R_\varepsilon<840$};
\draw[fill=gray!25, thick, rounded corners] (5.75,0.75) rectangle (6.75,2.35);
\node[align=center] at (6.25,1.55) {$f_{28}$\\$\mathcal R_\varepsilon(f_{28})<840$};

\draw[<->] (0.75,0.36) -- (5.05,0.36) node[midway, below=3pt] {$D=7714$};
\draw[<->] (5.75,0.36) -- (6.75,0.36) node[midway, below=3pt] {$1$};
\draw[thick] (0.75,-0.18) -- (6.75,-0.18);
\node[below=10pt] at (3.75,-0.18) {$F=E_{<28}\oplus\operatorname{span}\{f_{28}\}$, $\dim F=D+1$};
\draw[->, thick] (7.25,1.55) -- (8.45,1.55);
\node[align=center, anchor=west] at (8.52,1.55) {min--max gives\\$\lambda_D(S^3,\widehat g_\varepsilon)<840$};
\end{tikzpicture}%
}
\caption{The min--max mechanism in the $S^3$ counterexample.  The direct sum $F$ has dimension $D+1$, and every nonzero vector in $F$ has Rayleigh quotient below the bottom of the round degree-$28$ cluster.}
\label{fig:S3-minmax-picture}
\end{figure}
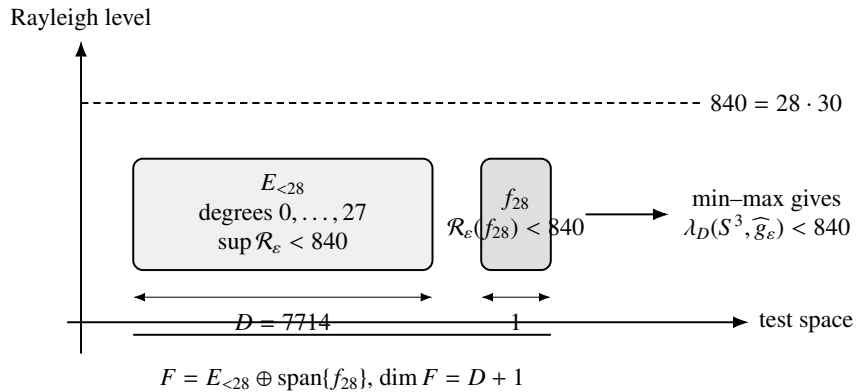

By \eqref{eq:S3-f28-bound}, \eqref{eq:S3-low-bound}, and the vanishing of the cross terms,
\begin{align}
\sup_{0\neq u\in F}\mathcal R_\varepsilon(u)<840. \label{eq:S3-F-bound}
\end{align}
The min--max principle applied to \eqref{eq:S3-F-bound} gives
\begin{align}
\lambda_D(S^3,\widehat g_\varepsilon)<840. \label{eq:S3-minmax-bound}
\end{align}
On the round unit $S^3$, the degree $l$ eigenvalue is $l(l+2)$ with multiplicity $(l+1)^2$.  By \eqref{eq:S3-D}, the bottom index of the degree $28$ cluster is exactly $D=7714$, and therefore
\begin{align}
\lambda_D(\mathbb S^3,g_0)=28\cdot30=840. \label{eq:S3-round-value}
\end{align}
Combining \eqref{eq:S3-minmax-bound} and \eqref{eq:S3-round-value} gives
\begin{align}
\lambda_{7714}(S^3,\widehat g_\varepsilon) < \lambda_{7714}(\mathbb S^3,g_0). \nonumber
\end{align}
This proves the theorem.
\end{proof}

\begin{remark}
The numerical choices are made only to keep the verification short.  The curvature check reduces to three quadratics with negative discriminant, and the eigenvalue drop is detected by the single degree $28$ spherical harmonic $f_{28}=\operatorname{Re}(w^{28})$.  
\end{remark}

\section{The dimension of polynomial growth harmonic functions}\label{sec:three-dimensional-yau}

This section returns to the original motivation.  The eigenvalue comparison on a two-sphere is not merely an isolated compact result; it is the cross-section statement predicted by the sharp dimension problem for polynomial-growth harmonic functions.  For a cone $C(X)$, homogeneous harmonic functions are controlled by the spectrum of the cross-section $X$.  Thus a finite counting comparison on $X$ becomes a dimension comparison for harmonic functions on the cone, and, through tangent-cone-at-infinity arguments, for noncompact manifolds.

The smooth theorem proves the needed comparison when the cross-section is a smooth positively curved two-sphere.  For a three-manifold with nonnegative sectional curvature and positive asymptotic volume ratio, however, the cross-section at infinity is naturally an Alexandrov two-sphere.  Hence the equality case forces us to ask for a singular version of the same finite counting rigidity.  This is the point where the complex-geometric language becomes more than a simplification: it provides the right framework for curvature measures, singular Hermitian metrics, and holomorphic sections on the conformal sphere.

We now isolate the part of the sharp Yau dimension problem which is genuinely suggested by the two-dimensional eigenvalue comparison proved above. 

For \(k\geq0\), put \(\rho(x):=d_g(p,x)\) and define
\begin{align}
\mathscr H_k(M):=\{u\in C^\infty(M):\Delta_g u=0,\ |u(x)|\leq C(\rho(x)^k+1)\text{ for some }C>0\}, \nonumber\\
h_k(M):=\dim\mathscr H_k(M). \nonumber
\end{align}

In this section \((M^3,g)\) is complete with sectional curvature \(K_g\geq0\) and
\begin{align}
\operatorname{AVR}(M):=\lim_{r\to\infty}\frac{\operatorname{vol}B_p(r)}{r^3}>0. \label{eq:avr-def}
\end{align}
The positivity condition \eqref{eq:avr-def} is what we call maximal volume growth.

For the Euclidean model,
\begin{align}
h_d(\mathbb R^3)=\sum_{j=0}^d(2j+1)=(d+1)^2,\qquad d\in\mathbb Z_{\geq0}. \label{eq:R3-hd}
\end{align}

Earlier work of Peter Li--Jiaping Wang \cite{LiWang1999} gave the corresponding asymptotic upper estimates under nonnegative sectional curvature.

Xian-Tao Huang's theorem upgrades the asymptotic-cone philosophy to the finite counting inequality under maximal volume growth and uniqueness of the tangent cone.

\begin{theorem}\label{thm:huang-finite-dimensional-comparison}
Let \((M^n,g)\) be a complete noncompact Riemannian manifold with \(\Ric_g\geq0\) and maximal volume growth. Assume that \(M\) has a unique tangent cone at infinity, denoted by
\begin{align}
(C(X),d_{C(X)},m_{C(X)}), \nonumber
\end{align}
where \(C(X)\) is the metric cone over the cross-section \((X,d_X,m_X)\). Then, for every \(k>0\),
\begin{align}
h_k(M^n)\leq\mathcal N_{-\Delta_X}\bigl(k(k+n-2)\bigr). \label{eq:huang-finite-dimensional-comparison}
\end{align}
\end{theorem}

\begin{proof}
This is \cite[Theorem~1.6]{Huang2021}.
\end{proof}

We now prove Theorem~\ref{thm:three-dim-positive-avr-bound-1} in the notation of this section.

\begin{proof}[Proof of Theorem~\ref{thm:three-dim-positive-avr-bound-1}]
\textbf{Step (1)}  For $k=0$, the Cheng--Yau gradient estimate for bounded harmonic functions under $\operatorname{Ric}_g\geq0$ \cite{ChengYau1975} gives $h_0(M)=1=h_0(\mathbb R^3)$.  Since $K_g\geq0$, the metric space $(M,d_g)$ is a three-dimensional Alexandrov space with curvature bounded below by $0$.

By Shioya's structure theorem for the limit cone of a noncompact Alexandrov space with nonnegative curvature \cite[Proposition~1.1, pp.~209--211]{ShioyaMassRays}, the blow-downs converge, without passing to a subsequence, to the unique tangent cone at infinity:
\begin{align}
(M,r_i^{-2}g,p)\longrightarrow C(X),\qquad r_i\to\infty. \label{eq:three-dim-tangent-cone}
\end{align}
Here $X=M(\infty)$ is the ideal boundary of $M$, and $C(X)$ is the Euclidean cone over $X$.  Since \eqref{eq:avr-def} is positive, the convergence in \eqref{eq:three-dim-tangent-cone} is noncollapsed and the limit cone has Hausdorff dimension $3$.  Hence $X$ has Hausdorff dimension $2$, and the standard cone criterion gives Alexandrov curvature at least $1$ on $X$ \cite[Proposition~4.2.3, pp.~13--14]{BGP}.

\textbf{Step (2)}.  The pointed form of Perelman's stability theorem implies that $C(X)$ is a topological three-manifold near its vertex $o$; see \cite[Theorem~7.11, pp.~125--126]{KapovitchStability}.  Moreover,
\begin{align}
B_o(\varepsilon)\setminus\{o\}\cong(0,\varepsilon)\times X \nonumber
\end{align}
for every $\varepsilon>0$.  With integer coefficients, local homology at the vertex therefore satisfies
\begin{align}
H_j(C(X),C(X)\setminus\{o\};\mathbb Z)\cong H_j(\mathbb R^3,\mathbb R^3\setminus\{0\};\mathbb Z). \nonumber
\end{align}
Because $C(X)$ is contractible and $C(X)\setminus\{o\}$ deformation retracts onto $X$, the long exact sequence of the pair also gives
\begin{align}
H_j(C(X),C(X)\setminus\{o\};\mathbb Z)\cong\widetilde H_{j-1}(X;\mathbb Z). \nonumber
\end{align}
Thus $X$ has the reduced integral homology of $S^2$.  By \cite[\S12.9.3, p.~52]{BGP}, every compact two-dimensional Alexandrov space is a compact topological surface, possibly with boundary.  The identity $H_2(X;\mathbb Z)\cong\mathbb Z$ rules out boundary and gives orientability, while $H_1(X;\mathbb Z)=0$.  The classification of compact connected surfaces now shows that $X$ is homeomorphic to $S^2$.

\textbf{Step (3)}  For $k\geq1$, \eqref{eq:huang-finite-dimensional-comparison} with $n=3$ and Corollary~\ref{cor:alex-counting-comparison-by-approx} give
\begin{align}
h_k(M)\leq\mathcal N_{-\Delta_X}\bigl(k(k+1)\bigr)\leq(k+1)^2=h_k(\mathbb R^3). \nonumber
\end{align}
Together with the case $k=0$ and \eqref{eq:R3-hd}, this proves \eqref{eq:positive-avr-sharp-bound-1}.

\textbf{Step (4)}  Suppose that $h_k(M)=(k+1)^2$ for some $k\in\mathbb Z^+$.  Equality in the preceding chain forces
\begin{align}
\mathcal N_{-\Delta_X}\bigl(k(k+1)\bigr)=(k+1)^2. \nonumber
\end{align}
The Alexandrov counting-rigidity theorem, Theorem~\ref{claim:alexandrov-link-rigidity-needed}, makes $X$ the unit round sphere.  Hence the tangent cone in \eqref{eq:three-dim-tangent-cone} is $\mathbb R^3$ and $\operatorname{AVR}(M)=\omega_3$.  Bishop--Gromov monotonicity says that
\begin{align}
r\longmapsto\frac{\operatorname{vol}B_p(r)}{\omega_3r^3} \nonumber
\end{align}
is nonincreasing.  Its limit as $r\downarrow0$ is $1$, and its limit as $r\to\infty$ is $\operatorname{AVR}(M)/\omega_3=1$; it is therefore identically one.  The equality case of Bishop--Gromov implies that $(M^3,g)$ is isometric to $\mathbb R^3$.
\end{proof}

\begin{remark}[Why a single-eigenvalue Obata statement is insufficient]
Theorem \ref{claim:alexandrov-link-rigidity-needed} cannot be replaced by a naive Obata-type statement of the form
\begin{align}
\lambda_j(X)=\lambda_j(\mathbb S^2,g_{\rm round})\quad\Longrightarrow\quad X\text{ is isometric to }(\mathbb S^2,g_{\rm round}). \nonumber
\end{align}
Indeed, to disprove such a statement it is enough to find one Alexandrov two-sphere \(X\) with curvature \(\geq1\), not isometric to the round unit sphere, but satisfying \(\lambda_1(X)=2\). The football metrics in Example \ref{ex:football-single-eigenvalue-equality} give exactly such examples. Thus the rigidity statement needed in Theorem \ref{claim:alexandrov-link-rigidity-needed} must be formulated as a counting equality at a whole round threshold, not as equality of one eigenvalue.
\end{remark}

\begin{example}[A nonround Alexandrov sphere with \(\lambda_1=2\)]
\label{ex:football-single-eigenvalue-equality}
Let \(0<c<1\). On \((0,\pi)\times(\mathbb R/2\pi\mathbb Z)\), set
\begin{align}
g_c=dr^2+c^2\sin^2 r\,d\theta^2. \nonumber
\end{align}
Let \(X_c\) be the metric completion obtained by adding the two endpoints \(N=\{r=0\}\) and \(S=\{r=\pi\}\). This is the standard spherical football. The original two-cone classification and construction of spherical metrics of constant curvature on \(S^2\) with two conical singularities is due to Troyanov \cite[pp.~296--306]{TroyanovTwoCones}. The explicit Friedrichs spectrum of the spherical football is computed in Mazzeo--Zhu \cite[Lemmas~2--3, arXiv version, pp.~24--25]{MazzeoZhuSphericalMetrics}.  The metric completion and the eigenfunction responsible for the equality \(\lambda_1(X_c)=2\) are shown in Figure~\ref{fig:football-geometry}.

\begin{figure}[H]
\centering
\resizebox{0.86\textwidth}{!}{%
\begin{tikzpicture}[>=Latex, every node/.style={font=\small}]
\begin{scope}[x=1cm,y=1cm]
  \coordinate (N) at (0,2.20);
  \coordinate (S) at (0,-2.20);
  \draw[thick] (N) .. controls (1.10,1.08) and (1.10,-1.08) .. (S)
               .. controls (-1.10,-1.08) and (-1.10,1.08) .. cycle;
  \draw[densely dashed] (N) .. controls (0.36,0.98) and (0.36,-0.98) .. (S);
  \draw[densely dashed] (N) .. controls (-0.36,0.98) and (-0.36,-0.98) .. (S);
  \draw[dotted] (-1.08,0) arc[start angle=180,end angle=360,x radius=1.08,y radius=0.23];
  \draw (1.08,0) arc[start angle=0,end angle=180,x radius=1.08,y radius=0.23];
  \fill (N) circle (1.4pt) node[above] {$N$};
  \fill (S) circle (1.4pt) node[below] {$S$};
  \draw (0,1.93) arc[start angle=-90,end angle=-50,radius=0.32];
  \node[right] at (0.34,1.94) {$2\pi c$};
  \draw[<->] (-1.40,0) -- (1.40,0) node[midway,above=3pt] {$2c\sin r$};
  \node[align=center] at (0,-2.92) {(a) the completed football};
\end{scope}
\begin{scope}[xshift=5.0cm,x=1.25cm,y=1.25cm]
  \draw[->, thick] (0,-1.15) -- (0,1.25) node[above] {$y$};
  \draw[->, thick] (-0.1,0) -- (3.45,0) node[right] {$r$};
  \draw (0,0.06)--(0,-0.06) node[below] {$0$};
  \draw (3.14159,0.06)--(3.14159,-0.06) node[below] {$\pi$};
  \draw[thick,domain=0:3.14159,samples=100,smooth]
       plot (\x,{0.78*sin(\x r)});
  \node[anchor=west] at (2.05,0.58) {$c\sin r$};
  \draw[thick,dashed,domain=0:3.14159,samples=100,smooth]
       plot (\x,{0.95*cos(\x r)});
  \node[anchor=west] at (0.10,1.08) {$u=\cos r$};
  \node[align=center] at (1.57,-1.42) {(b) profile radius and axial mode};
\end{scope}
\end{tikzpicture}%
}
\caption{The football metric $g_c=dr^2+c^2\sin^2 r\,d\theta^2$.  The cone angle at each added endpoint is $2\pi c$, and the axial mode $u=\cos r$ gives the eigenvalue $2$ independently of $c$.}
\label{fig:football-geometry}
\end{figure}
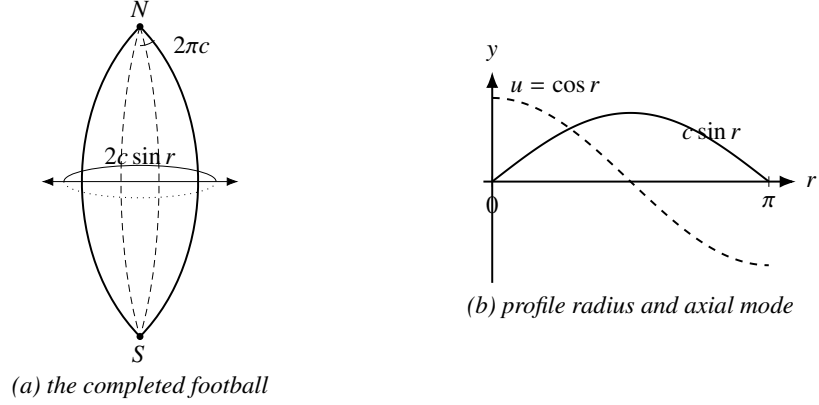

Near \(r=0\), since \(\sin r=r+O(r^3)\), one has
\begin{align}
g_c=dr^2+c^2r^2d\theta^2+O(r^4)d\theta^2. \nonumber
\end{align}
The same expansion holds near \(r=\pi\). Hence \(N\) and \(S\) are conical points with cone angle \(2\pi c\). Since \(c<1\), these cone angles are strictly smaller than \(2\pi\). On the regular part \(X_c^{\rm reg}:=X_c\setminus\{N,S\}\), the Gaussian curvature is
\begin{align}
K_{g_c}=-\frac{(c\sin r)''}{c\sin r}=1. \nonumber
\end{align}
The curvature measure is therefore
\begin{align}
\omega_{X_c}=dA_{g_c}+2\pi(1-c)\delta_N+2\pi(1-c)\delta_S\geq dA_{g_c}. \nonumber
\end{align}
Thus \(X_c\) is an Alexandrov two-sphere with curvature at least \(1\). It is not isometric to the round unit sphere because it has two conical points.

By the spectrum formula in \cite[Lemmas~2--3, arXiv version, pp.~24--25]{MazzeoZhuSphericalMetrics},
\begin{align}
\operatorname{Spec}(-\Delta_{X_c})=\left\{\left(\frac{j}{c}+\ell\right)\left(\frac{j}{c}+\ell+1\right):j,\ell\in\mathbb Z_{\geq0}\right\}, \nonumber
\end{align}
with multiplicity one for \(j=0\) and multiplicity two for \(j>0\). The low-lying part of this spectrum is represented in Figure~\ref{fig:football-spectrum-lattice}; it separates the meridian mode at eigenvalue $2$ from all angular modes when $0<c<1$.

\begin{figure}[H]
\centering
\resizebox{0.82\textwidth}{!}{%
\begin{tikzpicture}[x=1.10cm,y=0.82cm,>=Latex, every node/.style={font=\small}]
\draw[->, thick] (-0.2,0) -- (7.2,0) node[right] {$j$};
\draw[->, thick] (0,-0.25) -- (0,4.45) node[above] {$\ell$};
\foreach \j in {0,1,2,3,4}{\draw (\j,0.07)--(\j,-0.07) node[below] {$\j$};}
\foreach \l in {0,1,2,3,4}{\draw (0.07,\l)--(-0.07,\l) node[left] {$\l$};}
\foreach \j in {0,1,2,3,4}{
  \foreach \l in {0,1,2,3,4}{\fill (\j,\l) circle (1.2pt);}
}
\draw[very thick] (0,0) circle (4pt);
\draw[very thick] (0,1) circle (4pt);
\node[anchor=west] at (0.18,0.15) {$\lambda_{0,0}=0$};
\node[anchor=west] at (0.18,1.15) {$\lambda_{0,1}=2$};
\draw[densely dashed] (0.5,-0.05) -- (0.5,4.15);
\node[align=left,anchor=west] (ang) at (4.85,3.65) {$j\geq1$ modes\\have $\lambda_{j,\ell}>2$};
\draw[->] (ang.west) -- (2.2,3.0);
\node[align=left,anchor=west] (mer) at (4.85,2.05) {$j=0,\ \ell\geq2$\\gives $\lambda\geq6$};
\draw[->] (mer.west) -- (0.03,2.0);
\draw[decorate,decoration={brace,mirror}] (-0.18,-0.02) -- (-0.18,1.02) node[midway,left=11pt] {low modes};
\end{tikzpicture}%
}
\caption{The lattice of spectral modes for the spherical football.  Only $(j,\ell)=(0,0)$ and $(0,1)$ lie at or below the threshold $2$; all other modes are strictly above $2$ when $0<c<1$.}
\label{fig:football-spectrum-lattice}
\end{figure}
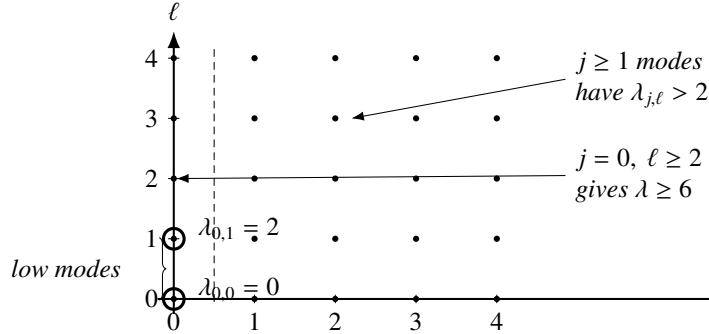

The mode \(j=0\), \(\ell=1\) gives the eigenfunction
\begin{align}
u(r,\theta)=\cos r,\qquad -\Delta_{X_c}u=2u. \nonumber
\end{align}
Since \(0<c<1\), every mode with \(j\geq1\) has eigenvalue
\begin{align}
\frac{j}{c}\left(\frac{j}{c}+1\right)>2. \nonumber
\end{align}
Also the modes \(j=0\), \(\ell\geq2\) have eigenvalues at least \(6\). Therefore the only eigenvalues not exceeding \(2\) are \(0\) and \(2\), and hence
\begin{align}
\lambda_1(X_c)=2=\lambda_1(\mathbb S^2,g_{\rm round}). \nonumber
\end{align}
This gives a nonround Alexandrov two-sphere with curvature \(\geq1\) and first positive eigenvalue equal to the first positive round eigenvalue. Thus single ordered-eigenvalue equality is not a valid Alexandrov rigidity hypothesis.
\end{example}

\end{document}